\documentclass[11pt]{amsart}
\usepackage[english]{babel}
\usepackage[utf8]{inputenc}
\usepackage[inner=1.1in, outer=1.1in, bottom =2.9cm, top =2.9cm]{geometry}
\makeatletter
\newcommand*\Cdot{\mathpalette\Cdot@{.5}}
\newcommand*\Cdot@[2]{\mathbin{\vcenter{\hbox{\scalebox{#2}{$\m@th#1\bullet$}}}}}
\makeatother

\usepackage{amsmath, amssymb, amsfonts, amsthm}
\usepackage{graphicx, overpic}
\usepackage{mathrsfs}
\usepackage{mathtools}
\usepackage[usenames,dvipsnames,svgnames,table]{xcolor}
\usepackage{enumerate}
\usepackage{enumitem}
\usepackage{shuffle} %Shuffle
\usepackage{lipsum}
\usepackage{color}
\usepackage{phoenician}
\usepackage{xkeyval}
\usepackage{mathrsfs}
\usepackage{tikz}
\usepackage{pst-node} % Commutative Diagrams
\usepackage{tikz-cd} % Commutative Diagrams
\usepackage[OT2,T1]{fontenc} %Shuffle
\usepackage{thm-restate} 
\usepackage{float}
\usepackage{eucal}
\usepackage{microtype}
\usepackage[strict]{changepage} %Margins
\usepackage{esvect}
\usepackage{etoolbox}
\usepackage{wrapfig}
\usepackage{caption}
\usepackage{cite}
\usepackage{mathrsfs}
\DeclareMathAlphabet{\mathpzc}{OT1}{pzc}{m}{it}
\usepackage{bbding}
\usepackage{array}
\usepackage{makecell}
\usepackage{stmaryrd}
\usepackage{lmodern}
\usepackage{amsbsy}
\usepackage{esvect}
\usepackage{bbding}
\usepackage{thm-restate}
\usepackage[toc,page, titletoc,title]{appendix}
\graphicspath{ {figures/} }
\usepackage{url}
\usepackage[pagebackref]{hyperref} %PUT LAST
\makeatletter
\providecommand*{\twoheadrightarrowfill@}{%
  \arrowfill@\relbar\relbar\twoheadrightarrow
}
\providecommand*{\twoheadleftarrowfill@}{%
  \arrowfill@\twoheadleftarrow\relbar\relbar
}
\providecommand*{\xtwoheadrightarrow}[2][]{%
  \ext@arrow 0579\twoheadrightarrowfill@{#1}{#2}%
}
\providecommand*{\xtwoheadleftarrow}[2][]{%
  \ext@arrow 5097\twoheadleftarrowfill@{#1}{#2}%
}
\makeatother
\makeatletter
\newcommand*{\relrelbarsep}{.386ex}
\newcommand*{\relrelbar}{%
  \mathrel{%
    \mathpalette\@relrelbar\relrelbarsep
  }%
}
\newcommand*{\@relrelbar}[2]{%
  \raise#2\hbox to 0pt{$\m@th#1\relbar$\hss}%
  \lower#2\hbox{$\m@th#1\relbar$}%
}
\providecommand*{\rightrightarrowsfill@}{%
  \arrowfill@\relrelbar\relrelbar\rightrightarrows
}
\providecommand*{\leftleftarrowsfill@}{%
  \arrowfill@\leftleftarrows\relrelbar\relrelbar
}
\providecommand*{\xrightrightarrows}[2][]{%
  \ext@arrow 0359\rightrightarrowsfill@{#1}{#2}%
}
\providecommand*{\xleftleftarrows}[2][]{%
  \ext@arrow 3095\leftleftarrowsfill@{#1}{#2}%
}
\makeatother
\makeatletter
\newcommand{\colim@}[2]{%
  \vtop{\m@th\ialign{##\cr
    \hfil$#1\operator@font colim$\hfil\cr
    \noalign{\nointerlineskip\kern1.5\ex@}#2\cr
    \noalign{\nointerlineskip\kern-\ex@}\cr}}%
}
\newcommand{\colim}{%
  \mathop{\mathpalette\colim@{\rightarrowfill@\scriptscriptstyle}}\nmlimits@
}
\renewcommand{\varprojlim}{%
  \mathop{\mathpalette\varlim@{\leftarrowfill@\scriptscriptstyle}}\nmlimits@
}
\renewcommand{\varinjlim}{%
  \mathop{\mathpalette\varlim@{\rightarrowfill@\scriptscriptstyle}}\nmlimits@
}
\makeatother

\DeclareSymbolFont{cyrletters}{OT2}{wncyr}{m}{n}
\DeclareMathSymbol{\Sh}{\mathalpha}{cyrletters}{"58}

\makeatletter
\newcommand*\bigcdot{\mathpalette\bigcdot@{.5}}
\newcommand*\bigcdot@[2]{\mathbin{\vcenter{\hbox{\scalebox{#2}{$\m@th#1\bullet$}}}}}
\makeatother

\usetikzlibrary{tikzmark,decorations.pathreplacing}
\usetikzlibrary{shapes,shadows,arrows}
\usetikzlibrary{decorations.markings}
\usetikzlibrary{positioning,calc}
\tikzset{near start abs/.style={xshift=1cm}}
\DeclareSymbolFont{symbolsC}{U}{txsyc}{m}{n}
\DeclareMathSymbol{\Searrow}{\mathrel}{symbolsC}{117}

\hypersetup{
	colorlinks=true,
	linkcolor=black,
	filecolor=black,      
	urlcolor=blue,
	citecolor= 	red,
}
\DeclareSymbolFont{extraup}{U}{zavm}{m}{n}
\DeclareMathSymbol{\varheart}{\mathalpha}{extraup}{86}
\DeclareMathSymbol{\vardiamond}{\mathalpha}{extraup}{87}
\theoremstyle{definition}
\newtheorem{thm}{Theorem}[section]
\newtheorem{cor}{Corollary}[thm]
\newtheorem{lem}[thm]{Lemma}
\newtheorem{prop}[thm]{Proposition}

\theoremstyle{definition}

\newtheorem{remark}{Remark}[section]

\newcommand{\gG}{\Gamma}

\newcommand{\ga}{\alpha}
\newcommand{\gd}{\delta}
\newcommand{\gb}{\beta}

\newcommand{\gf}{\varphi}
\newcommand{\go}{\omega}
\newcommand{\bN}{\mathbb{N}}

\newcommand{\bZ}{\mathbb{Z}}

\newcommand{\cC}{\CMcal{C}}

\newcommand{\cP}{\CMcal{P}}

\newcommand{\cF}{\CMcal{F}}
\newcommand{\cR}{\CMcal{R}}
\newcommand{\cG}{\CMcal{G}}
\newcommand{\cW}{\CMcal{W}}
\newcommand{\cI}{\CMcal{I}}

\newcommand{\cB}{\CMcal{B}}

\newcommand{\cS}{\CMcal{S}}

\newcommand{\id}{\operatorname{id}}

\newcommand{\inv}{\operatorname{inv}}

\newcommand{\FG}{\operatorname{FG}}

\newcommand{\Aut}{\operatorname{Aut}}

\newcommand{\la}{\langle}
\newcommand{\ra}{\rangle}
\newcommand{\wt}{\widetilde}
\newcommand{\til}{\tilde}
\definecolor{Red}{rgb}{0.8,0,0.2}
\newcommand{\GG}[1]{}

\makeatletter
\def\@footnotecolor{red}
\define@key{Hyp}{footnotecolor}{%
 \HyColor@HyperrefColor{#1}\@footnotecolor%
}
\def\@footnotemark{%
    \leavevmode
    \ifhmode\edef\@x@sf{\the\spacefactor}\nobreak\fi
    \stepcounter{Hfootnote}%
    \global\let\Hy@saved@currentHref\@currentHref
    \hyper@makecurrent{Hfootnote}%
    \global\let\Hy@footnote@currentHref\@currentHref
    \global\let\@currentHref\Hy@saved@currentHref
    \hyper@linkstart{footnote}{\Hy@footnote@currentHref}%
    \@makefnmark
    \hyper@linkend
    \ifhmode\spacefactor\@x@sf\fi
    \relax
  }%
\makeatother

\hypersetup{footnotecolor=blue}
\makeatletter
\patchcmd{\@startsection}
  {\@afterindenttrue}
  {\@afterindentfalse}
  {}{}
\makeatother
\title[Quotients of Buildings by Groups Acting Freely on Chambers]{Quotients of Buildings by Groups Acting\\ Freely on Chambers}
\author{William Norledge}
\address[William Norledge]{Pennsylvania State University}
\email{wxn39@psu.edu}
\begin{document}

% % % % % % % % % % % % % % % % % % % % % % % % %
% Autoref names
\renewcommand{\chapterautorefname}{Chapter}
\renewcommand{\sectionautorefname}{Section}
\renewcommand{\subsectionautorefname}{Section}
% % % % % % % % % % % % % % % % % % % % % % % % % 

\begin{abstract}
We introduce certain directed multigraphs with extra structure, called Weyl graphs, which model quotients of Tits buildings by type-preserving chamber-free group actions. Their advantage over complexes of groups, which are often used for the CAT(0) Davis realization of buildings, is that Weyl graphs exploit the ultimate combinatorial W-metric structure of buildings. Weyl graphs generalize Tits's chambers systems of type M by allowing 2-residues to be quotients of generalized polygons by flag-free group actions, and Weyl graphs are easily constructed by amalgamating such quotients. We develop covering theory of Weyl graphs, which can be used to construct buildings as universal covers. We describe a method of obtaining a group presentation of the fundamental group of a Weyl graph, which acts chamber-freely on the covering building. The theory developed here is part of a fully general theory which deals with not necessarily chamber-free actions and the stacky version of buildings.
\end{abstract}

\maketitle

% % % % % % % % % % % % % % % % % % % % % % % % % % % % % 
\setcounter{tocdepth}{1} % what appears
\hypertarget{foo}{ }
\tableofcontents
%\setlength{\cftbeforesecskip}{1pt} % spacing
% % % % % % % % % % % % % % % % % % % % % % % % % % % % % 

%%%%%%%%%%%%%%%%%%%%%%%%%%%%%%%%%%%%%%%%%%%%%%%%%%%%%%%%%%%%%%%%%%%%%%%%
%%%%%%%%%%%%%%%%%%%%%%%%%%%%%%%%%%%%%%%%%%%%%%%%%%%%%%%%%%%%%%%%%%%%%%%%
\section*{Introduction}\label{intro}
%%%%%%%%%%%%%%%%%%%%%%%%%%%%%%%%%%%%%%%%%%%%%%%%%%%%%%%%%%%%%%%%%%%%%%%%
%%%%%%%%%%%%%%%%%%%%%%%%%%%%%%%%%%%%%%%%%%%%%%%%%%%%%%%%%%%%%%%%%%%%%%%%

The theory of lattices in Lie groups was extended by Bruhat-Tits, Ihara, Serre, and others to algebraic groups over discretely valued fields by equipping these groups with actions on associated Bruhat-Tits buildings \cite{MR185016}, \cite{MR327923}. Similarly, lattices in Kac-Moody groups have been studied by constructing actions on Kac-Moody buildings \cite{carbone2003existence}, \cite{remy2006topological}. Therefore, in the study of (lattices in) locally compact groups, automorphism groups of buildings are natural examples to study after Lie groups and algebraic groups. See \cite{anneprobsonpolycomp}. Indeed, many aspects of lattice theory for algebraic groups have been extended to locally finite trees by Bass-Lubotzky \cite{BL}, which are the non-classical buildings of type $W=\widetilde{A}_1$ generalizing the Bruhat-Tits building (tree) for $\text{SL}_2$. One would like to generalize tree lattices to higher rank buildings, which are naturally CAT(0) cell complex by taking the Davis realization \cite{davis94buildingscat0}, \cite[Section 12.4]{ab08}. Crucial to the study of tree lattices is the (covering) theory of graphs of groups, which are the stacky\footnote{ stacky in the sense of stacks, see e.g. \cite{noohi05}} quotients trees \cite{serre1980trees}, \cite{Bass}. However, higher dimensional analogs of graphs of groups, called complexes of groups \cite{Hae}, \cite[Chapter III.$\cC$]{BH}, \cite{MR2400734}, do not take full advantage of the fundamental combinatorial aspect of buildings. 

In \cite{tits81local}, Tits developed such combinatorial covering theory of buildings in the restricted setting of type-preserving group actions which are free on the set of $2$-residues. See also \cite[ Chapter 4]{ronanlectures}, \cite{kantorscabs}. In this paper we go beyond $2$-residues, and develop covering theory which accounts for all chamber-free actions. In this more general setting, we require the generalization of chamber systems obtained by replacing equivalence relations with groupoids. Also, extra data we call `suites' are needed, which tell you when a cycle is covered by an apartment in a $2$-residue. Equivalently, they tell you what strict homotopies of galleries are permitted. The resulting structure is a directed multigraph with edges labeled over the generators of a Coxeter group, such that edges of the same label have the structure of a groupoid, equipped with a collection of distinguished cycles- the generating suites. 

Previously in \cite{nor1} we introduced $W$-groupoids, which model quotients of buildings by generic \hbox{type-preserving} group actions. Thus, $W$-groupoids are the full stacky generalization of buildings, e.g. they are to buildings what orbifolds are to manifolds. They generalize Bruhat decompositions by allowing non chamber-transitive actions \cite[Proposition 7.1]{nor1}. We call a $W$-groupoid strict if it corresponds to a chamber-free action. In \autoref{sec:wgroup}, we show that Weyl graphs are essentially presentations\footnote{ in the sense of group presentations} of strict $W$-groupoids.

A general strategy which uses the theory of Weyl graphs to construct chamber-free actions of groups on buildings is the following. Take a class of quotients of generalized polygons and amalgamate them along common $1$-residues to form higher rank $2$-Weyl graphs $\cW$. Associate `gluing data' to each choice of amalgamation (examples of gluing data include the four diagrams in \cite[p.~48]{ronanlectures} which determine the quotients of the four chamber-regular lattices of type $\wt{A}_2$ and order $2$, and the `based difference matrices' in \cite{essert2013geometric} which determine the quotients of the so-called Singer lattices of type $\wt{A}_2$). Then, by first obtaining presentations of the fundamental groups of the quotient generalized polygons, one uses the gluing data to encode a presentation of the fundamental group of $\cW$. We use this strategy in \cite{nor3} to construct the Singer cyclic lattices of type $M=(m_{st})_{s,t\in S}$, where $m_{st}\in \{1,2,3,\infty\}$, which generalizes a construction in \cite{essert2013geometric} (done there using complexes of groups for the Davis realization). 

%%%%%%%%%%%%%%%%%%%%%%%%%%%%%%%%%%%%%%%%%%%%%%%%%%%%%%%%%%%%%%%%%%%%%%%%
\subsection*{Structure} 
%%%%%%%%%%%%%%%%%%%%%%%%%%%%%%%%%%%%%%%%%%%%%%%%%%%%%%%%%%%%%%%%%%%%%%%%

We begin in \autoref{sec:prelim} by recalling some basic aspects of Coxeter groups. In \autoref{sec:welydata} we define Weyl data, which is the data of a Weyl graph that does not satisfy any axioms. We describe several important notions in this general setting. In \autoref{sec:2weylgraphs} we define $2$-Weyl graphs, which is Weyl data that is `almost' the quotient of a building. Precisely, $2$-Weyl graphs are quotients of chamber systems of type $M$ by chamber-free group actions. In \autoref{sec:coverofweyl} we describe covering theory of $2$-Weyl graphs, which borrows heavily from covering theory of groupoids, and establish a Galois correspondence. (We give a self-contained exposition of covering theory of groupoids in \autoref{apen}.) In \autoref{sec:weylgraphs} we define Weyl graphs, and show that a $2$-Weyl graph is a Weyl graph if and only if its universal cover is a building. Then, by a generalization of the local-to-global result of Tits \cite{tits81local}, a $2$-Weyl graph is a Weyl graph if and only if its residues of type $C_3$ and $H_3$ are Weyl graphs. We show that our theory provides a simple construction of $2$-Weyl graphs by amalgamating quotients of generalized polygons along groupoids (the thick analog of constructing Coxeter groups by amalgamating dihedral groups along $\bZ/2\bZ$). Finally, in \autoref{sec:groupres}, we give a canonical presentation of the fundamental group of a Weyl graph, modulo the choice of a spanning tree. 

%%%%%%%%%%%%%%%%%%%%%%%%%%%%%%%%%%%%%%%%%%%%%%%%%%%%%%%%%%%%%%%%%%%%%%%%
\subsection*{Acknowledgments} 
%%%%%%%%%%%%%%%%%%%%%%%%%%%%%%%%%%%%%%%%%%%%%%%%%%%%%%%%%%%%%%%%%%%%%%%%

We thank Alina Vdovina, Corneliu Hoffman, and David Stewart for useful suggestions. We thank Anne Thomas for useful discussions during an early stage of this project. This research was partly supported by a grant from Templeton Religion Trust as part of the mathematical picture language project at Harvard University. We also thank Newcastle University for their support. 
 
%%%%%%%%%%%%%%%%%%%%%%%%%%%%%%%%%%%%%%%%%%%%%%%%%%%%%%%%%%%%%%%%%%%%%%%%
%%%%%%%%%%%%%%%%%%%%%%%%%%%%%%%%%%%%%%%%%%%%%%%%%%%%%%%%%%%%%%%%%%%%%%%%
\section{Preliminaries} \label{sec:prelim} 
%%%%%%%%%%%%%%%%%%%%%%%%%%%%%%%%%%%%%%%%%%%%%%%%%%%%%%%%%%%%%%%%%%%%%%%%
%%%%%%%%%%%%%%%%%%%%%%%%%%%%%%%%%%%%%%%%%%%%%%%%%%%%%%%%%%%%%%%%%%%%%%%%

In this section, we fix notation and terminology, and recall basic results on Coxeter groups. Most of this material is standard, see e.g. \cite{bourbaki2008lie}, \cite{bjorner06combinatorics}, \cite{davisbook}, although our terminology concerning homotopy of words is slightly different (\autoref{section:homotopyofwords}). %We also give an exposition of covering theory of groupoids, from which covering theory of Weyl graphs will borrow heavily. %We change the classical theory slightly by making use of `outer isomorphisms', which allow for a basepoint-free definition of the fundamental group. We give a full proof of the Galois correspondence for groupoids in this basepoint-free setting.
 
%%%%%%%%%%%%%%%%%%%%%%%%%%%%%%%%%%%%%%%%%%%%%%%%%%%%%%%%%%%%%%%%%%%%%%%%
\subsection{Graphs and Galleries} \label{sec:defofgraph} \index{graph}
%%%%%%%%%%%%%%%%%%%%%%%%%%%%%%%%%%%%%%%%%%%%%%%%%%%%%%%%%%%%%%%%%%%%%%%%

Let a \emph{graph} $\cW=(\cW_0,\cW_1)$ be a set of \emph{vertices} $\cW_0$ and a set of \emph{edges} $\cW_1$ which is equipped with a function
 \[     
\cW_1\to \cW_0\times \cW_0
,\qquad  
i\mapsto (\iota(i),\tau(i)) 
.\footnote{ so our notion of a graph is that of a directed graph with loops and multiple edges}\] 
Then $\iota(i)$ and $\tau(i)$ are called the \emph{initial} and \emph{terminal} vertices (or `extremities') of $i$ respectively. Provided there is no ambiguity, we let $\cW$ denote both $\cW_0$ and $\cW_1$. A \emph{morphism} of graphs $\go=(\go_0,\go_1)$ consists of a function of the vertices $\go_0$ and a function of the edges $\go_1$ which preserves extremities. Provided there is no ambiguity, we let $\go$ denote both $\go_0$ and $\go_1$. 

Let $S$ be a set of labels. A \emph{graph labeled over $S$} is a graph $\cW=(\cW_0,\cW_1)$ equipped with a \emph{type function} $\upsilon$ on its edges into $S$,
\[
\upsilon:\cW_1\to S
,\qquad   
i\mapsto \upsilon(i) 
.\]   
The label $\upsilon(i)\in S$ is called the \emph{type} of $i$. A \emph{morphism} of graphs labeled over $S$ is a morphism of the underlying graphs which additionally preserves types of edges.  
 
For $a,b\in \bZ$ with $a\leq b$, let $\lfloor a,b \rfloor$ denote the graph whose set of vertices is the interval $[a,b]\subset \bZ$, with a single edge $i_k$ traveling from $k-1$ to $k$ for each $k\in \{a+1,\dots,b\}$, thus
\[
\lfloor a,b \rfloor \ = \ a \xrightarrow{i_{a+1}}  a+1 \xrightarrow{i_{a+2}} \quad  \cdots \quad   \xrightarrow{i_{b-1}} b-1  \xrightarrow{i_{b}} b  
.\]
Let $\cW$ be a graph labeled over $S$, and let $\lfloor a,b \rfloor$ also be labeled over $S$. A \emph{gallery} $\gb$ in $\cW$ is a (labeled graph) morphism
\[
\gb:\lfloor a,b \rfloor \to \cW
.\] 
Then $\iota(\gb):=\gb(a)$ and $\tau(\gb):=\gb(b)$ are called the \emph{initial} and \emph{terminal} vertices of $\gb$ respectively. The \emph{length} $|\gb|$ of $\gb$ is the number of edges of $\lfloor a,b \rfloor$. Putting $s_k=\upsilon(i_k)$ for $k\in \{a+1,\dots,b\}$, then the \emph{type} $\gb_S$ of $\gb$ is the sequence of labels
\[
\gb_S:=s_{a+1},\dots, s_{b}
.\] 
A \emph{trivial} gallery is a gallery $\gb$ with $|\gb|=0$, in which case $\gb_S$ is empty. A \emph{cycle} is a gallery $\gb$ with $\iota(\gb)=\tau(\gb)$. A \emph{minimal gallery} is a gallery $\gb$ whose length is minimal amongst all the galleries from $\iota(\gb)$ to $\tau(\gb)$. The \emph{sequence of edges of} $\gb$ is the sequence of edges
\[
\gb(i_{a+1}),\dots, \gb(i_{b})
.\]
Conversely, a finite sequence of adjacent edges of $\cW$ determines a gallery. A \emph{subgallery} of $\gb$ is a gallery whose sequence of edges is a consecutive subsequence of the sequence of edges of $\gb$. Let $\gb$ and $\gb'$ be galleries in $\cW$ such that $\tau(\gb)=\iota(\gb')$. The \emph{concatenation} \index{gallery!concatenation} $\gb \gb'$ of $\gb$ with $\gb'$ is the gallery whose sequence of edges is the sequence of edges of $\gb$ followed by the sequence of edges of $\gb'$.    

%%%%%%%%%%%%%%%%%%%%%%%%%%%%%%%%%%%%%%%%%%%%%%%%%%%%%%%%%%%%%%%%%%%%%%%%
\subsection{Coxeter Groups}
%%%%%%%%%%%%%%%%%%%%%%%%%%%%%%%%%%%%%%%%%%%%%%%%%%%%%%%%%%%%%%%%%%%%%%%%

A \emph{marked group} $(G,S)$ is a group $G$ equipped with a finite generating set $S\subseteq G$. A \emph{group generated by involutions} $(W,S)$ is a marked group whose generators $S\subseteq W$ are involutions $s^2=1$. We let $W$ denote both $(W,S)$ and the underlying group of $(W,S)$, since the meaning will always be clear from the context. For $(W,S)$ a group generated by involutions and $s,t\in S$, let $m_{st}$ denote the order\footnote{ smallest positive power which equals the identity; we may have $m_{st}=\infty$} of $st\in W$. Notice that $m_{st}=m_{ts}$.  

Let $W=(W,S)$ be a group generated by involutions, and let $f=s_1\dots s_n$ be a word over $S$ (i.e. a sequence of elements, `letters', of $S$). We denote the length of $f$ by $|f|=n$, and we let $f^{-1}:=s_n\dots s_1$. A \emph{subword} of $f$ is a consecutive subsequence of $f$. A \emph{substring} of $f$ is a (perhaps non-consecutive) subsequence of $f$. For example, if $f=sttustu$, then $sttus$ and $tust$ are subwords, whereas $stsu$ is the substring obtained by skipping every second letter. 

Let $w(f)$ denote the element of $W$ obtained by treating $f$ as a product of generators of $W$. If $w=w(f)$, then $w^{-1}=w(f^{-1})$. We say two words $f$, $f'$ are \emph{equivalent} (or `homotopic', see below) if $w(f)=w(f')$. Provided there is no ambiguity, we sometimes let $f$ denote $w(f)$. The word $f$ is called a \emph{decomposition} of $w(f)$. We call a word $f$ \emph{reduced} if there are no words equivalent to $f$ which have a length strictly less than $f$. If $f$ is reduced, then $f$ is called a \emph{reduced decomposition} of $w(f)$. For $w\in W$, the \emph{word length} $|w|:=|f|$ of $w$ is the length of the reduced decomposition(s) $f$ of $w$.   

A \emph{Coxeter group} $W=(W,S)$ of \emph{rank} $|S|$ is a group generated by involutions such that for any group $G$ and any function $F:S\to G$ such that $(F(s)F(t))^{m_{st}}=1$ for all $s,t\in S$, we have that $F$ extends to a unique homomorphism $W\to G$. It follows that $W$ has the following \emph{canonical presentation},
\[
W=\big \la S\ |\ (st)^{m_{st}}=1 : s,t\in S  \big  \ra 
.\footnote{ the data of a Coxeter group together with its choice of generators is sometimes called a Coxeter system, but this term is redundant for us since by Coxeter group we mean a certain marked group}\]
For $J\subseteq S$ we denote by $W_J=(W_J,J)$ the \emph{standard subgroup} of $W$ which is generated by $J$. Standard subgroups are themselves Coxeter groups in the obvious way. A subset $J\subseteq S$ is called \emph{spherical} if $W_J$ is a finite group.  

A \emph{Coxeter matrix} $M$ on a set $S$ is a symmetric matrix
\[
M:S\times S\to \bZ_{\geq 1}\cup \{\infty\}
,\qquad
(s,t)\mapsto M_{st}
\] 
such that $M_{st}=1$ if and only if $s=t$. For $J\subseteq S$, we denote by $M_J$ the Coxeter matrix which is the restriction of $M$ to $J\times J$. A Coxeter group $W$ determines a Coxeter matrix $M=M(W)$ by putting $M_{st}=m_{st}$. Conversely, for any Coxeter matrix $M$, the \emph{Coxeter group of type $M$} is the group
\[W(M)=\big \la S\ |\ (st)^{M_{st}}=1 : s,t\in S \big  \ra.  \]
By constructing a faithful linear representation (see e.g. \cite[Section 2.5]{ab08}), one can prove that the order of $st \in W(M)$ is equal to $M_{st}$. Therefore we also denote $M_{st}$ by $m_{st}$. %This shows that Coxeter groups are essentially in bijection with Coxeter matrices (up to suitable notions of equivalence), however inequivalent Coxeter groups may have isomorphic underlying groups.  

A \emph{simplicial graph} $L=(V(L),E(L))$ is an undirected graph without loops or multiple edges in which the \emph{edges} $E(L)$ are modeled as $2$-element subsets of the \emph{vertices} $V(L)$. \index{Coxeter group!defining graph} Let $M$ be a Coxeter matrix on $S$. The \emph{defining graph} $L=L(M)$ of $M$ is the edge labeled simplicial graph with
\[V(L)=S,\qquad E(L)=\big\{\{s,t\}: s,t\in S,\  s\neq t,\ m_{st}< \infty \big\} \]
where the edge $\{s,t\}\in E(L)$ is labeled by $m_{st}$. Note that $L$ is defined differently to Coxeter-Dynkin diagrams. 

%%%%%%%%%%%%%%%%%%%%%%%%%%%%%%%%%%%%%%%%%%%%%%%%%%%%%%%%%%%%%%%%%%%%%%%%
\subsection{Homotopy of Words} \label{section:homotopyofwords} \index{words!homotopy}
%%%%%%%%%%%%%%%%%%%%%%%%%%%%%%%%%%%%%%%%%%%%%%%%%%%%%%%%%%%%%%%%%%%%%%%%

Let $W=(W,S)$ be a Coxeter group with Coxeter matrix $M$, let $f$, $f'$, $f''$ be words over $S$, and let $s,t\in S$ with $s\neq t$. Let an \emph{$(s,t)$-word}, or \emph{alternating word}, be a word over $\{s,t\}$ which begins with the letter $s$ and contains no consecutive letters. For $m\in \bN$, we denote by $p_m(s,t)$ the unique $(s,t)$-word which has length $m$ ($m$ may be odd). If $m_{st}< \infty$, we denote by $p(s,t)$ the word $p_{m_{st}}(s,t)$. Thus
\[   
p_m(s,t):=\underbrace{stst\dots\, }_{\text{$m$ letters}}  \qquad \text{and} \qquad p(s,t):=\underbrace{stst\dots\, }_{\text{$m_{st}$ letters}}   
.\]
We denote by $p^{-1}(s,t)$ the word obtained from $p(s,t)$ by reversing the order of letters. %If $m_{st}$ is odd, then $p(s,t)=p^{-1}(t,s)$. 

A \emph{contraction} is an alteration from a word of the form $f ss f'$$\, $\footnote{ juxtaposition of words denotes their concatenation} to the word $f f'$. An \emph{expansion} is an alteration from a word of the form $f f'$ to the word $f ss f'$. An \emph{elementary strict homotopy} is an alteration from a word of the form $f p(s,t) f'$ to the word $f p(t,s) f'$. A \emph{strict homotopy} is an alternation of a word which is a composition of elementary strict homotopies. If a word $f$ can be altered via a strict homotopy to give the word $\hat {f}$, we say $f$ is \emph{strictly homotopic} to $\hat {f}$, and we write $f\simeq \hat {f}$. The relation `$\simeq$' is an equivalence relation on words over $S$.\footnote{ our notion of strict homotopy of words is what Tits and Ronan call homotopy of words \cite{tits81local}, \cite{ronanlectures}.} %The reason for this change in terminology will become clear when we introduce homotopy of galleries.  

Let a \emph{homotopy} of words be any composition of contractions, expansions, and elementary strict homotopies. If a word $f$ can be altered via a homotopy to give the word $\hat{f}$, we say $f$ is \emph{homotopic} to $\hat{f}$, and write $f\sim \hat{f}$. The relation `$\sim$' is an equivalence relation on words over $S$. Clearly, if two words are strictly homotopic, then they are homotopic. An important property of Coxeter groups is that a partial converse to this holds (\autoref{thm:funthmofcox} (MT2)). Clearly, two words are equivalent if and only if they are homotopic. 

We say a word $f$ is \emph{$M$-reduced} if there are no words strictly homotopic to $f$ which are of the form $f' ssf''$. Clearly, reduced implies $M$-reduced, with the converse also holding (\autoref{thm:funthmofcox} (MT1)). 

\begin{prop} \label{prop:+-1}
Let $W=(W,S)$ be a Coxeter group. For all $w\in W$ and $s\in S$, we have the following dichotomies,
\[|ws|=|w|+1\qquad   \text{or}\qquad  |ws|=|w|-1\] 
and 
\[|sw|=|w|+1\qquad   \text{or}\qquad   |sw|=|w|-1.\]
\end{prop}  
\begin{proof}
This follows by the triangle inequality and the fact that homotopic words have the same length modulo $2$.
\end{proof}

These dichotomies may be viewed in terms of the Bruhat order, see \autoref{remondicot}. 

\begin{prop} \label{prop:endins}
Let $W=(W,S)$ be a Coxeter group. For all $w\in W$ and $s\in S$, if $|ws|=|w|-1$ (alternatively $|sw|=|w|-1$), then there exists a reduced decomposition $f$ of $w$ which ends (alternatively starts) with $s$. 
\end{prop}

\begin{proof}
	Let $f'$ be a reduced decomposition of $ws$. Then $|f'|=|ws|=|w|-1$. Put $f=f's$. Firstly, $f$ is a decomposition of $w$ since
	\[w=(ws)s=w(f')s=w(f's)=w(f).\] 
	\noindent Secondly, $f$ is reduced since 
	\[|f|=|f's|=|f'|+1=|w|.\] 
	\noindent The result for when $|sw|=|w|-1$ follows by a symmetric argument. 
\end{proof}

The following is a key result from \cite{tits1969wordproblem}, which gives a solution to the word problem in Coxeter groups.

\begin{thm} \label{thm:funthmofcox} \index{Coxeter group!Main Theorem}
For any Coxeter group $W$, we have the following:
	
\begin{enumerate}  [itemindent=0cm, leftmargin=1.9cm]
\item [\textbf{(MT1)}]
$M$-reduced words are reduced
\item [\textbf{(MT2)}]
homotopic reduced words are strictly homotopic. 
\end{enumerate}
\end{thm}

We also have the \emph{deletion condition} and the \emph{exchange condition} of Coxeter groups, which are easily proven as consequences of \autoref{thm:funthmofcox}. 

\begin{cor}[Deletion Condition]  \index{Coxeter group!Deletion condition}
	If a word $f$ over $S$ is not reduced, then there exists a substring of $f$ obtained by deleting two letters which is homotopic to $f$.   
\end{cor}

\begin{cor}[Exchange Condition] \index{Coxeter group!Exchange condition}
	Let $f$ be a reduced word. If $fs$ (alternatively $sf$) is not reduced, then there exists a substring $f'$ of $f$, obtained by deleting one letter, which is homotopic to $fs$ (alternatively $sf$).  
\end{cor}

We also have the following easy result.

\begin{cor} \label{cor:cancalation}
	Let $f$ be a word, and let $f'$ and $f''$ be reduced words. If $f' f\sim f'' f$ (or $f f' \sim f f'' $), then $f'\simeq f''$.
\end{cor}

\begin{proof}
	We have $w(f' f)=w(f'' f)$, so $w(f') w(f)=w(f'') w(f)$, which implies that $w(f')=w(f'')$. Thus $f'\sim f''$, and so $f'\simeq f''$ by (MT$2$). The result for when $f f' \sim f f'' $ follows by a symmetric argument. 
\end{proof}

\begin{prop} \label{prop:subgroupofcox}
	Let $W=(W,S)$ be a Coxeter group with Coxeter matrix $M$, and let $J\subseteq S$. The extension $W(M_J)\to W$ of the embedding $J\hookrightarrow S$ is an embedding of groups. In particular, the standard subgroup $W_J\leq W$ is naturally the Coxeter group $W(M_J)$.
\end{prop}

\begin{proof}
	Let $f$ and $f'$ be words over $J$ which are homotopic with respect to $M$. It suffices to show that this homotopy is a composition of contractions, expansions, and elementary strict homotopies between words over $J$, since this shows that $f$ and $f'$ are also homotopic with respect to $M_J$. Indeed, (MT$1$) tells us we can homotope $f$ and $f'$ to reduced words, say $\hat{f}$ and $\hat{f'}$ respectively, using only strict homotopies and contractions. Then, by (MT$2$), we have $\hat{f}\simeq \hat{f'}$. 
\end{proof}

Let us briefly describe the Bruhat order, which is a way of ordering the elements of a Coxeter group. To give a quick definition of the Bruhat order we need the following result, which is proved in \cite{bjorner06combinatorics}.

\begin{prop}\label{prop:bracket}
	Let $W$ be a Coxeter group, and let $w,w'\in W$. If a decomposition of $w'$ is a substring of a reduced decomposition of $w$, then every (reduced) decomposition of $w$ contains a substring which is a decomposition of $w'$.
\end{prop}

The \emph{Bruhat order} of a Coxeter group $W$ is the binary relation `$\leq$' on the elements of $W$ such that $w'\leq w$ if a reduced decomposition of $w$ contains a substring which is a decomposition of $w'$. By \autoref{prop:bracket}, we have $w'\leq w$ if and only if every reduced decomposition of $w$ contains a substring which is a decomposition of $w'$. It follows that $\leq$' is a partial ordering of the elements of $W$.

\begin{remark} \label{remondicot}
	Let $w\in W$ and $s\in S$. Notice that $ws>w$ if and only if $|ws|=|w|+1$, and, by the exchange condition, $ws<w$ if and only if $|ws|=|w|-1$. Similarly, $sw>w$ if and only if $|sw|=|w|+1$, and $sw<w$ if and only if $|sw|=|w|-1$.
\end{remark}  

%%%%%%%%%%%%%%%%%%%%%%%%%%%%%%%%%%%%%%%%%%%%%%%%%%%%%%%%%%%%%%%%%%%%%%%%
%%%%%%%%%%%%%%%%%%%%%%%%%%%%%%%%%%%%%%%%%%%%%%%%%%%%%%%%%%%%%%%%%%%%%%%%
\section{Weyl Data} \label{sec:welydata}
%%%%%%%%%%%%%%%%%%%%%%%%%%%%%%%%%%%%%%%%%%%%%%%%%%%%%%%%%%%%%%%%%%%%%%%%
%%%%%%%%%%%%%%%%%%%%%%%%%%%%%%%%%%%%%%%%%%%%%%%%%%%%%%%%%%%%%%%%%%%%%%%%

We introduce Weyl data and related notions, in particular homotopy of galleries, and fundamental groupoids. Throughout this paper, we shall point to \autoref{apen} for relevant details on groupoids. See also \cite{brown06topology}.

%%%%%%%%%%%%%%%%%%%%%%%%%%%%%%%%%%%%%%%%%%%%%%%%%%%%%%%%%%%%%%%%%%%%%%%%
\subsection{Graphs of Type $M$} \label{section:Graphs of Type $M$}
%%%%%%%%%%%%%%%%%%%%%%%%%%%%%%%%%%%%%%%%%%%%%%%%%%%%%%%%%%%%%%%%%%%%%%%%

Let $M=(m_{st})$ be a Coxeter matrix on $S$ with associated Coxeter group $W=W(M)$. Let us call a graph $\cW$ labeled over $S$ a \emph{graph of type $M$}. From now on, we say \emph{chambers} instead of vertices. For $\gb$ a gallery in a graph of type $M$, recall that $\gb_S$ denotes the word over $S$ which is the sequence of types of edges of $\gb$. The element 
\[
\gb_W:=w(\gb_S)\in W
\] 
for which $\gb_S$ is a decomposition is called the \emph{$W$-length} of $\gb$. Let a \emph{geodesic} $\gamma$ (often called a gallery of reduced type) be a gallery whose type $\gamma_S$ is a reduced word. 

Let $s,t\in S$ with $s\neq t$ and $m_{st}< \infty$. An \emph{$(s,t)$-geodesic}, or \emph{alternating geodesic}, is a geodesic with type an alternating word of the form
\[ 
p_m(s,t)= \underbrace{stst\dots}_{\text{$m$ letters}} \qquad \qquad  \text{(it follows that $m\leq m_{st}$)}
.\] 
A \emph{maximal $(s,t)$-geodesic}, or \emph{maximal alternating geodesic}, is a gallery with type 
\[p(s,t)=\underbrace{stst\dots}_{\text{$m_{st}$ letters}}.\] 
An \emph{$(s,t)$-cycle} is cycle with type $p_{2m_{st}}(s,t)$.

For $J\subseteq S$, the \emph{$J$-restriction} $\cW_J$ of $\cW$ is the graph of type $M_J$ with chambers $\cW_0$, and edges
\[(\cW_J)_1=\big\{i\in \cW_1:\upsilon(i)\in J\big \}\subseteq \cW_1.\] 
\noindent The extremities and type function of $\cW_J$ are the restrictions to $(\cW_J)_1$ of the corresponding functions of $\cW$. For $J=\{s\}$, we abbreviate $\cW_s:=\cW_{\{s\}}$.

%%%%%%%%%%%%%%%%%%%%%%%%%%%%%%%%%%%%%%%%%%%%%%%%%%%%%%%%%%%%%%%%%%%%%%%%
\subsection{Generalized Chamber Systems and Weyl Data} \label{section:Generalized Chamber Systems and Weyl Data}
%%%%%%%%%%%%%%%%%%%%%%%%%%%%%%%%%%%%%%%%%%%%%%%%%%%%%%%%%%%%%%%%%%%%%%%%

In the following, see \autoref{groupoids} for our definition of a groupoid. Let a \emph{generalized chamber system} $\cW=(\cW_0,\cW_1,\cW_s)$ of type $M$ be a graph $\cW=(\cW_0,\cW_1)$ of type $M$, with additional data
\begin{enumerate} [label=(\arabic*),start=1]
	\item
	for each $s\in S$, a (small) groupoid whose vertices and non-trivial edges are those of $\cW_s$, called the \emph{panel groupoid} of type $s$.
\end{enumerate}
If in addition $\cW$ has data
\begin{enumerate} [resume*]
	\item
	for each pair $(s,t)\in S\times S$ such that $s\neq t$ and $m_{st}< \infty$, a set $\cW(s,t)$ of $(s,t)$-cycles in $\cW$ called \emph{defining $(s,t)$-suites}, or \emph{defining suites}\footnote{ the defining suites may be viewed as groupoid-theoretic analogs of defining relators for groups; they will tell us what strict homotopies of galleries are permitted}
\end{enumerate}
then $\cW=(\cW_0,\cW_1,\cW_s,\cW(s,t))$ is called \emph{Weyl data}. Weyl data is called \emph{simple} if every \hbox{$(s,t)$-cycle} of the underlying graph of $\cW$ is a defining suite. An \emph{$s$-panel} is a connected component of the panel groupoid of type $s$, and we let \emph{panel} refer to an $s$-panel for some $s\in S$. From now on, we let $\cW_s$ denote the panel groupoid of type $s$, i.e. $\cW_s$ has an extra set of trivial edges which are not part of the graph $\cW$.   

A generalized chamber system of type $M$ is equivalent to a set of chambers $\cW_0$ which is equipped with an indexed family of groupoids $(\cW_s)_{s\in S}$, where the set of chambers of each groupoid $\cW_s$ is $\cW_0$. If all the indexed groupoids $\cW_s$ are setoids (equivalent to equivalence relations), then we recover Tits's notion of a chamber system \cite{tits81local}. 

%The \emph{rank} of a generalized chamber system is the cardinality of $S$. A generalized chamber system is called \emph{locally finite} \index{Weyl data!locally finite} if each of its panels is a finite groupoid. 

Given an edge $i\in \cW$ in a generalized chamber system with $\upsilon(i)=s$, we denote by $i^{-1}$ the inverse of $i$ in the panel groupoid $\cW_{s}$. For edges $i,i'\in \cW$ with $\upsilon(i)=\upsilon(i')=s$ and $\tau(i)=\iota(i')$, we denote by $i;i'$ their composition in $\cW_s$ (if $i'\neq i^{-1}$, then $i;i'$ is an edge of $\cW$). We let $ii'$ denote the gallery which is the concatenation of $i$ with $i'$. Let a \emph{backtrack} $\gb$ be a gallery which consists of an edge followed by its inverse; $\gb=i i^{-1}$. Let a \emph{detour} be a gallery $\gb$ which consists of two edges of the same type which are not mutually inverse; $\gb=ii'$, where $\upsilon(i)=\upsilon(i')$ and $i'\neq i ^{-1}$. Given a gallery $\gb$ in a generalized chamber system $\cW$, with sequence of edges $i_1,\dots,i_n$, let the \emph{inverse} $\gb^{-1}$ of $\gb$ be the gallery in $\cW$ whose sequence of edges is $i_n^{-1},\dots,i_1^{-1}$. 

A \emph{morphism} $\go:\cW\to \cW'$ of generalized chamber systems of type $M$ is a labeled graph morphism which satisfies the following two properties,
\begin{enumerate}[label=(\roman*)]
	\item
	for all edges $i\in \cW$, we have 
	\[\go(i^{-1})=(\go(i))^{-1}\]
	\item
	for all detours $ii'$ in $\cW$, we have 
	\[\go(i;i')=\go(i);\go(i').\]
\end{enumerate}
If we think of a generalized chamber system as a collection of groupoids indexed over $S$, then a morphism of generalized chamber systems is equivalently a collection of groupoid morphisms indexed over $S$, each of which has a trivial kernel,\footnote{ i.e. for each edge $g\in \cW_s$, $\omega(g)=1$ implies that $g$ is trivial} and consists of the same function on the chambers. 

\begin{remark} \label{remark:preservationofss}
	The composition of a backtrack with a morphism is a backtrack, and the composition of a detour with a morphism is a detour. %The first statement follows from the first property of morphisms. To see the second statement, suppose that the composition of a detour $ii'$ with a morphism $\go$ is a backtrack. Then $\go(i')=\go(i)^{-1}$. But $\go(i;i')=\go(i);\go(i')$ must be non-trivial, a contradiction.       
\end{remark} 

%%%%%%%%%%%%%%%%%%%%%%%%%%%%%%%%%%%%%%%%%%%%%%%%%%%%%%%%%%%%%%%%%%%%%%%%
\subsection{Homotopy of Galleries} \label{sec:homoofgal} 
%%%%%%%%%%%%%%%%%%%%%%%%%%%%%%%%%%%%%%%%%%%%%%%%%%%%%%%%%%%%%%%%%%%%%%%%

Let $\cW=(\cW_0,\cW_1,\cW_s,\cW(s,t))$ be Weyl data. Let $i\in \cW$ be any edge. Let $j,j'\in \cW$ be edges such that $jj'$ is a detour, and put $k=j;j'$. Let $\theta(s,t)$ be a defining suite of $\cW$. Then a \emph{contraction} of a gallery in $\cW$ is any of the following:

\begin{enumerate} [label=(\roman*)]
	\item
	delete a backtrack; an alternation from a gallery of the form $\gb ii^{-1} \gb'$ to the gallery $\gb \gb'$ 
	\item
	take a shortcut; an alternation from a gallery of the form $\gb jj' \gb'$ to the gallery $\gb k \gb'$ 
	\item
	delete a defining suite; an alternation from a gallery of the form $\gb\theta(s,t) \gb'$ to the gallery $\gb \gb'$.
\end{enumerate}

\noindent An \emph{expansion} is an alteration of a gallery which is the inverse of a contraction. An \emph{elementary homotopy} is an expansion or a contraction. A \emph{$1$-elementary homotopy} is an elementary homotopy of type (i) or (ii). A \emph{$2$-elementary homotopy} is an elementary homotopy of type (iii). Notice that $1$-elementary homotopies of type (i) and $2$-elementary homotopies preserve $W$-length, whereas $1$-elementary homotopies of type (ii) do not.

%\begin{remark} \label{remark:nopreservation}
%	A $1$-elementary homotopy of type (i) does not preserve length, but it does preserve $W$-length. An $1$-elementary homotopy of type (ii) does not preserve length or $W$-length. A $1$-elementary homotopy of either kind won't preserve the property of being a minimal gallery or a geodesic. A $2$-elementary homotopy does not preserve length, but it does preserve $W$-length. A $2$-elementary homotopy also won't preserve the property of being a minimal gallery or a geodesic.
%\end{remark}

A \emph{homotopy} is an alteration of a gallery which is a composition of elementary homotopies. If a gallery $\gb$ can be altered via a homotopy to give the gallery $\hat{\gb}$, we say $\gb$ is \emph{homotopic} to $\hat{\gb}$, and write $\gb \sim \hat{\gb}$. Then `$\sim$' is an equivalence relation on galleries. We denote by $[\gb]$ the homotopy equivalence class of the gallery $\gb$. We say a gallery $\gb$ is \emph{null-homotopic} if $\gb$ is homotopic to a trivial gallery (a gallery of length $0$). 

%\begin{remark} \label{lemma:invofgal}
%	Let $\gb$ be a gallery in Weyl data $\cW$. Then the concatenations $\gb \gb^{-1}$ and $\gb^{-1} \gb$ are null-homotopic since these galleries can be altered to give a trivial gallery via a composition of contractions of type (i).
%\end{remark}

%\begin{remark} \label{remark:homotopyofsubgalleries}
%	If we have some homotopy altering $\gb$ to give $\hat{\gb}$, then this homotopy can be applied to any gallery of the form $\gb' \gb \gb''$, to give $\gb' \hat{\gb} \gb''$. Thus, if galleries differ only by homotopic subgalleries, then they're homotopic. 
%\end{remark}

We now define certain homotopies which preserves both length and $W$-length. Let $\rho(s,t)$ be a maximal $(s,t)$-geodesic, and let $\rho(t,s)$ be a maximal $(t,s)$-geodesic such that $\rho(s,t)\sim \rho(t,s)$. An \emph{elementary strict homotopy} is an alteration from a gallery of the form $\gb\rho(s,t)  \gb'$ to the gallery $\gb\rho(t,s)  \gb'$. Since galleries which differ only by homotopic subgalleries are necessarily homotopic, it follows that elementary strict homotopies are homotopies. Elementary strict homotopies change the type of a gallery by an elementary strict homotopy of words.

A \emph{strict homotopy} is an alteration of a gallery which is a composition of elementary strict homotopies. If a gallery $\gb$ can be altered via a strict homotopy to give the gallery $\hat{\gb}$, we say $\gb$ is \emph{strictly homotopic} to $\hat{\gb}$, and write $\gb \simeq \hat{\gb}$.  Then `$\simeq$' is an equivalence relation on galleries. We denote by $[\gb]_{\simeq}$ the strict homotopy equivalence class of the gallery $\gb$. If two galleries are strictly homotopic, then they are homotopic. Strict homotopies change the type of a gallery by a strict homotopy of words.

\begin{remark} \label{rem:stricthomopresgeo}
	Strict homotopies preserve both the length and the $W$-length of a gallery, and therefore the property of being a minimal gallery or a geodesic. Thus, if $\gamma$ is geodesic, then so is every gallery in $[\gamma]_{\simeq}$.  
\end{remark}

We define the \emph{$W$-length} of a strict homotopy class of galleries $[\gb]_{\simeq}$ (but usually geodesics) to be $\gb_W$. For $\gamma$ a geodesic in $\cW$, let the \emph{$\gamma$-gallery map} $F_{\gamma}$ be the function given by
\[       F_{\gamma}: [\gamma]_{\simeq}\to M(S),\footnote{ $M(S)$ denotes the set of words over $S$ (the free monoid on $S$)}  \qquad  \hat{\gamma} \mapsto \hat{\gamma}_S   .     \]
\noindent By \autoref{rem:stricthomopresgeo}, each $\hat{\gamma}\in [\gamma]_{\simeq}$ is a geodesic, and so the image of $F_{\gamma}$ is a subset of the words which are reduced decompositions of $\gamma_W$. 

Let an \emph{$(s,t)$-suite}, or just \emph{suite}, be an $(s,t)$-cycle which is null-homotopic. In particular, defining suites are suites. An $(s,t)$-cycle $\theta(s,t)$ can be presented as the concatenation of a maximal $(s,t)$-geodesic $\rho(s,t)$ with the inverse of a maximal $(t,s)$-geodesic $\rho(t,s)$,
\[  \theta(s,t)=\rho(s,t) \rho(t,s)^{-1}.\]
\noindent The geodesics $\rho(s,t)$ and $\rho(t,s)$ are uniquely determined. %Notice that if $m_{st}$ is even, then $\rho(t,s)^{-1}$ is a maximal $(s,t)$-geodesic, and if $m_{st}$ is odd, then $\rho(t,s)^{-1}$ is a maximal $(t,s)$-geodesic. 
We have that $\theta(s,t)$ is a suite of $\cW$ if and only if $\rho(s,t)\sim \rho(t,s)$, since 
\begin{align*}
&\theta(s,t)\  \text{is a suite} \\
\iff &\rho(s,t) \rho(t,s)^{-1}\ \text{is null-homotopic} \\
\iff &\rho(s,t) \rho(t,s)^{-1} \rho(t,s)\sim \rho(t,s)\\
\iff &\rho(s,t)\sim \rho(t,s).
\end{align*} 
It is straightforward to check that if $\theta$ is a suite of $\cW$, then so is $\theta^{-1}$, and so is any cyclic permutation of $\theta$.   

\subsection{Morphisms of Weyl Data}
%%%%%%%%%%%%%%%%%%%%%%%%%%%%%%%%%%%%%%%%%%%%%%%%%%%%%%%%%%%%%%%%%%%%%%%%

A \emph{morphism} $\go:\cW\to \cW'$ of Weyl data of type $M$ is a morphism of the underlying generalized chamber systems (as defined in \autoref{section:Generalized Chamber Systems and Weyl Data}) which additionally satisfies the following property,
\begin{enumerate} [label=(\roman*),start=3]
	\item
	for each pair $(s,t)\in S\times S$ with $s\neq t$ and $m_{st}< \infty$, and each defining suite $\theta(s,t)\in \cW(s,t)$, we have that $\go\circ \theta(s,t)$ is a suite of $\cW'$. 
\end{enumerate}

\begin{lem}  \label{lemma:decentofhomotopy}
	Let $\go:\cW \to\cW'$ be a morphism of Weyl data. If $\gb$ and $\hat{\gb}$ are homotopic galleries in $\cW$, then $\go \circ \gb$ and  $\go \circ \hat{\gb}$ are homotopic galleries in $\cW'$.   
\end{lem}

\begin{proof}
	Suppose that $\gb \sim \hat{\gb}$ via an elementary homotopy. If it is a $1$-elementary homotopy, then $\go \circ \gb \sim \go \circ \hat{\gb}$ by \autoref{remark:preservationofss}. If it is a $2$-elementary homotopy, then $\go \circ \gb \sim \go \circ \hat{\gb}$ by the fact that morphisms send defining suites to suites. The result then follows, since a homotopy is a composition of elementary homotopies.
\end{proof}

\begin{cor} \label{cor:decentofhomotopy} 
	Let $\go:\cW\to \cW'$ be a morphism of Weyl data and let $\theta$ be a suite of $\cW$, then $\go \circ \theta$ is a suite of $\cW'$.
\end{cor}

\begin{proof}
	If $\theta$ is a suite, then $\theta\sim \gb_0$, where $\gb_0$ is a trivial gallery. Then, by \autoref{lemma:decentofhomotopy}, $\go \circ \theta\sim \go \circ \gb_0$. But $\go \circ \gb_0$ is also a trivial gallery. Thus, $\go \circ \theta$ is null-homotopic, and therefore is a suite. 
\end{proof}

The \emph{composition} of morphisms of Weyl data is just their composition as morphisms of generalized chamber systems. By \autoref{cor:decentofhomotopy}, this composition is again a morphism of Weyl data. 

For $\cW$ and $\cW'$ Weyl data of type $M$, an \emph{isomorphism} $\go:\cW\to \cW'$ is a morphism which has an inverse. An \emph{automorphism} of Weyl data $\cW$ is an isomorphism from $\cW$ to itself. 

\begin{prop} \label{prop:charisoofweyl}
	Let $\cW$ and $\cW'$ be Weyl data of type $M$. A morphism $\go:\cW\to \cW'$ is an isomorphism if and only if $\go$ is bijective on chambers and edges, and for all galleries $\gb$ in $\cW$ such that $\go \circ\gb$ is a defining suite of $\cW'$, we have that $\go$ is a suite of $\cW$.
\end{prop}

\begin{proof}
	Suppose that $\go:\cW\to \cW'$ is an isomorphism, and let $\go^{-1}$ be the inverse of $\go$. Then $\go_0$ and $\go_1$ must be bijective since they have inverses $\go_0^{-1}$ and $\go_1^{-1}$ respectively as functions of sets. Let $\gb$ be a gallery in $\cW$ such that $\go \circ\gb$ is a defining suite of $\cW'$. But as a morphism, $\go^{-1}$ sends defining suites to suites, and so $\go^{-1} \circ  \go \circ\gb=\gb$ must be a suite.
	
	Now suppose that $\go:\cW\to \cW'$ is a morphism, $\go_0$ and $\go_1$ are bijective, and for all galleries $\gb$ in $\cW$ such that $\go \circ\gb$ is a defining suite of $\cW'$, we have that $\go$ is a suite of $\cW$. Let $\go^{-1}=(\go^{-1}_0,\go^{-1}_1)$. This is clearly a morphism of generalized chamber systems. Moreover, $\go^{-1}$ is a morphism of Weyl data by the hypothesis on $\go$, and so is an inverse for $\go$. 
\end{proof}

%%%%%%%%%%%%%%%%%%%%%%%%%%%%%%%%%%%%%%%%%%%%%%%%%%%%%%%%%%%%%%%%%%%%%%%%
\subsection{Restrictions and Residues} \label{section:subsatarestrictres}
%%%%%%%%%%%%%%%%%%%%%%%%%%%%%%%%%%%%%%%%%%%%%%%%%%%%%%%%%%%%%%%%%%%%%%%%

Let $M$ be a Coxeter matrix on $S$, and let $\cW$ be Weyl data of type $M$. Let $\cC$ be a subset of the chambers of $\cW$. The \emph{full subdata} $\cW_{\cC}$ of $\cW$ on $\cC$ is the Weyl data with chambers $\cC$, and edges,
\[(\cW_{\cC})_1=\big \{i\in \cW_1:e_i\cap \cC=e_i\big \}\subseteq \cW_1.\]
\noindent The extremities and type function of $\cW_{\cC}$ are the restrictions to $(\cW_{\cC})_1$ of the corresponding functions of $\cW$. The panel groupoids are the obvious restrictions of the panel groupoids of $\cW$, and the defining suites are the defining suites of $\cW$ whose images are contained in $\cW_\cC$. This gives $\cW_{\cC}$ the structure of Weyl data of type $M$.

Let $J\subseteq S$. The \emph{$J$-restriction} $\cW_J$ of $\cW$ is the Weyl data of type $M_J$ whose underlying graph of type $M_J$ is the $J$-restriction of $\cW$ (in the sense of \autoref{section:Graphs of Type $M$}), whose panel groupoid of type $s$, for $s\in J$, is the panel groupoid of type $s$ of $\cW$, and whose set of defining $(s,t)$-suites, for $(s,t)\in J\times J$, is the set of defining $(s,t)$-suites $\cW(s,t)$ of $\cW$.  

For $J\subseteq J'\subseteq S$, there is a natural embedding $\varepsilon_{JJ'}:\cW_J\hookrightarrow \cW_{J'}$ over the inclusion $J\hookrightarrow J'$, called the \emph{internal embedding} from $J$ to $J'$. We denote $\varepsilon_{JS}$ by $\varepsilon_{J}$.  

In the same way that the data which defines a Coxeter group can be encoded in an edge labeled simplicial graph, Weyl data can be encoded in a vertex and edge labeled simplicial graph, whose flags (adjacent vertex-edges pairs) are associated to embeddings of rank $1$ Weyl data into rank $2$ Weyl data. \index{Weyl data!defining graph}  

Let $\cW$ be Weyl data of type $W$. Let $L$ be the defining graph of $W$. The \emph{defining graph} $\mathcal{L}$ of $\cW$ is the graph $L$ whose vertex $s\in S$ is labeled by $\cW_s$, and whose edge $J=\{s,t\}\in E(L)$ is labeled by $\cW_J$. Each ordered pair, or \emph{flag}, $( s, J  )\in V(L)\times E(L)$ such that $s\in J$, is equipped with the embedding $\varepsilon_{sJ}: \cW_s\hookrightarrow \cW_J$. The defining graph of $\cW$ essentially reconstructs $\cW$ as an amalgam of rank $2$ Weyl data along the panel groupoids of $\cW$.

Let $C\in \cW$ be a chamber. The \emph{$J$-residue} $R_J(C)$ at $C$ is the connected component of $\cW_J$ which contains $C$. Formally, $R_J(C)$ is the full subdata of $\cW_J$ on the subset of $\cW_0$ containing those vertices which are connected by galleries in $\cW_J$ to $C$. Thus, $R_J(C)$ is connected Weyl data of type $M_J$. For $J=\{s\}$, we write $R_s(C)$. If $|J|=n$, then we call $R_J(C)$ an \emph{$n$-residue}. The $1$-resides are the panels, and the $0$-residues are the chambers. A $J$-residue is called \emph{spherical} if $J$ is a spherical subset. 

\begin{lem}\label{lemma:automorphisms}
	Let $\cW$ be Weyl data and let $R$ be a residue of $\cW$. Let $g\in \Aut(\cW)$ be an automorphism of $\cW$. If there exists a chamber $C\in R$ with $g\cdot C\in R$, then $g\cdot R= R$. 
\end{lem} 

\begin{proof}
	Let $J$ be the type of $R$. The automorphism $g$ is an automorphism of $\cW_J$, preserving the connectivity of $J$-residues. Therefore $g\cdot R$ is contained in $R$. But $g\cdot R$ must be a $J$-residue of $\cW$, and therefore $g\cdot R= R$. 
\end{proof}

\begin{remark} \label{rem:actiononres}
	Groups act by automorphisms on Weyl data. By \autoref{lemma:automorphisms}, the action of a group $G$ on Weyl data $\cW$ induces an action of $G$ on the set of $J$-residues of $\cW$, for each $J\subseteq S$. The action on $\emptyset$-residues is the restriction of the action to chambers, and the action on $S$-residues is the action induced on the connected components of $\cW$.
\end{remark}

%%%%%%%%%%%%%%%%%%%%%%%%%%%%%%%%%%%%%%%%%%%%%%%%%%%%%%%%%%%%%%%%%%%%%%%%
\subsection{The Fundamental Groupoid of Weyl Data}
%%%%%%%%%%%%%%%%%%%%%%%%%%%%%%%%%%%%%%%%%%%%%%%%%%%%%%%%%%%%%%%%%%%%%%%%

The \emph{fundamental groupoid} $\overline{\cW}$ of Weyl data $\cW$ is the groupoid whose set of vertices is the set of chambers of $\cW$, and whose set of edges is given by
\[
\overline{\cW}_1:=\big\{ [\gb]: \gb\ \text{is a gallery in}\ \cW  \big\}
.\] 
Recall that $[\gb]$ denotes the set of galleries which are homotopic to $\gb$. The extremities of edges are given by
\[  
\iota([\gb]):=\iota(\gb)\qquad \text{and} \qquad \tau([\gb])=\tau(\gb)
.\]
Let $\id:\cW_0\to \overline{\cW}_1$ be the map which sends a chamber $C\in \cW$ to the class of the trivial gallery at $C$, let $\inv:\overline{\cW}_1\to \overline{\cW}_1$ be the map $[\gb]\mapsto [\gb^{-1}]$, and let the composition be 
\begin{equation} \label{eq:1}      \tag{$\varheart$}
[\gb];[\gb']:=[\gb \gb'].
\end{equation}
\noindent First, we need to check that these functions are well-defined.

\begin{prop} \label{proposition:binaryoperation}
	Let $\cW$ be Weyl data. The extremities, inverses, and composition of $\overline{\cW}$ are well-defined.   
\end{prop}

\begin{proof}
	For the extremities, just notice that an elementary homotopy of a gallery preserve its extremities. For the inverses, suppose that $\gb\sim \hat{\gb}$ via an elementary homotopy, then we need to show that $\gb^{-1}\sim \hat{\gb}^{-1}$. In the case of a $1$-elementary homotopy of type (i), just notice that the inverse of a backtrack is a backtrack. In the case of a $1$-elementary homotopy of type (ii), notice that if $jj'$ is a detour with $k=j;j'$, then $j'^{-1}j^{-1}$ is a detour with $k^{-1}=j'^{-1};j^{-1}$. The case of a $2$-elementary homotopy is clear. For the composition, if $\gb \sim \hat{\gb}$, then clearly $\gb \gb' \sim \hat{\gb}\gb'$. Similarly, if $\gb \sim \hat{\gb}$, then $\gb' \gb  \sim \gb' \hat{\gb}$. 
\end{proof}

It is now clear that the fundamental groupoid $\overline{\cW}$ of $\cW$ is indeed a groupoid. Notice that $\overline{\cW}$ is connected if and only if $\cW$ is connected. Weyl data $\cW$ is called \emph{simply connected} if $\overline{\cW}$ is a setoid. The (basepoint-free) \emph{fundamental group} $\pi_1(\cW)$ of connected Weyl data $\cW$ is the (basepoint-free) fundamental group of $\overline{\cW}$ (see \autoref{fun group}). Thus, connected Weyl data is simply connected if and only if its fundamental group is trivial.  

Given a morphism of Weyl data $\go:\cW\to \cW'$, let $\bar{\go}$ denote the following homomorphism of groupoids,
\[
\bar{\go}:\overline{\cW} \to \overline{\cW}',\qquad [\gb] \mapsto [\go \circ \gb]
.\] 
This is well-defined by \autoref{lemma:decentofhomotopy}. To see that $\bar{\go}$ is a groupoid homomorphism, we have 
\[
[\gb][\gb']=[\gb\gb']\mapsto [\go\circ \gb\gb']=[\go\circ \gb] [   \go\circ \gb']
.\] 
We call $\bar{\go}$ the groupoid homomorphism \emph{induced} by $\go$. The map $\go\mapsto \bar{\go}$ is functorial; for morphisms $\go:\cW\to \cW'$ and $\go':\cW'\to \cW''$, putting $\go''=\go'\circ \go$, we have $\bar{\go}''=\bar{\go}' \circ \bar{\go}$. This follows directly from associativity $(\go'\circ \go) \circ \gb=\go'\circ (\go \circ \gb)$. 

We call the fundamental groupoid $\overline{\cW}_J$ of $\cW_J$ the \emph{$J$-groupoid} of $\cW$. For $J\subseteq J'\subseteq S$, we have a homomorphism $\overline{\varepsilon}_{JJ'}:  \overline{\cW}_J\to \overline{\cW}_{J'}$, called the \emph{internal homomorphism} from $J$ to $J'$. Later on, we will see that $\overline{\varepsilon}_{JJ'}$ is an embedding of groupoids when $\cW$ is the quotient of a building.

%%%%%%%%%%%%%%%%%%%%%%%%%%%%%%%%%%%%%%%%%%%%%%%%%%%%%%%%%%%%%%%%%%%%%%%%
\section{$2$-Weyl Graphs} \label{sec:2weylgraphs} 
%%%%%%%%%%%%%%%%%%%%%%%%%%%%%%%%%%%%%%%%%%%%%%%%%%%%%%%%%%%%%%%%%%%%%%%%

We introduce $2$-Weyl graphs and collect some of their basic properties. We will see that $2$-Weyl graphs are exactly the quotients of chamber systems of type $M$ by chamber-free group actions.

\subsection{Pre-Weyl Graphs and $2$-Weyl Graphs} \label{section:preweyl} 
Let $M$ be a Coxeter matrix. A \emph{pre-Weyl graph} $\cW$ of type $M$ is Weyl data of type $M$ which satisfies the following two properties, 
	\begin{enumerate}  [itemindent=0cm, leftmargin=1.9cm]
		\item [\textbf{(PW0)}]
		no panel is isomorphic to the trivial groupoid\footnote{ i.e. the unique groupoid with one vertex and one edge} (equivalently, for any chamber $C\in \cW$, and any word $f$ over $S$, there exists a gallery $\gb$ in $\cW$ with $\iota(\gb)=C$ and $\gb_S=f$)
		\item [\textbf{(PW1)}]
		each maximal $(s,t)$-geodesic is homotopic to a maximal $(t,s)$-geodesic.
	\end{enumerate}
	If in addition $\cW$ satisfies the following property
	
	\begin{enumerate}  [itemindent=0cm, leftmargin=1.9cm]
		\item [\textbf{($2$W)}]
		homotopic alternating geodesics have the same $W$-length
	\end{enumerate}
	then $\cW$ is called a \emph{$2$-Weyl graph}, or we say that $\cW$ is $2$-Weyl. A \emph{morphism} of pre-Weyl or $2$-Weyl graphs is a morphism of the underlying Weyl data (as defined in \autoref{sec:homoofgal}). These axioms should be compared with those in \cite[Section 3.2]{tits81local}. The axiom playing the role of ($2$W) is denoted (CS$_M$2) by Tits.

Since a gallery whose type is a one letter word is an alternating geodesic, property ($2$W) implies that
\[
\overline{\varepsilon}_s:\cW_s\to \overline{\cW}, \qquad  i\mapsto [i]     
\] 
is injective for each $s\in S$. For $2$-Weyl graphs, we make the convention of identifying $\cW_s$ with the subgroupoid $\overline{\varepsilon}_s(\cW_s)\leq \overline{\cW}$. Thus, we think of the edges of a $2$-Weyl graph $\cW$ as also being edges of its fundamental groupoid $\overline{\cW}$.

\begin{prop} \label{PW2}
	Let $\cW$ be Weyl data. Property (PW$1$) is equivalent to the property that for all geodesics $\gamma$ in $\cW$, the $\gamma$-gallery map $F_{\gamma}$ is surjective into the words which are reduced decompositions of $\gamma_W$.
\end{prop}

\begin{proof}
	By (MT2), every reduced decomposition of $\gamma_W$ is strictly homotopic to $\gamma_S$. By definition, this strict homotopy of words is a composition of elementary strict homotopies, which can be done at the gallery level by (PW$1$). The converse is clear. 
\end{proof}

\begin{prop}  \label{allgalgeo}
	Let $\cW$ be a pre-Weyl graph. Then every gallery $\gb$ of $\cW$ is homotopic to a geodesic. Moreover, this homotopy can be chosen to be a composition of $1$-elementary contractions and strict homotopies.
\end{prop}

\begin{proof}
	If $\gb_S$ is reduced, we are done. If not, then by (MT1), $\gb_S$ is strictly homotopic to a word which repeats a letter. Therefore, by property (PW$1$), $\gb$ is strictly homotopic to a gallery $\gb'$ which contains either a backtrack or a detour. Then, $\gb'$ is homotopic via a $1$-elementary contraction to a shorter gallery. We can keep carrying out this process of applying strict homotopies and then $1$-elementary contractions until we obtain a geodesic, which will be homotopic to $\gb$.   
\end{proof}

It follows that if $\cW$ is a pre-Weyl graph, then the edges of its fundamental groupoid are homotopy classes of \emph{geodesics},
\[\overline{\cW}_1=  \underbrace{ \big\{ [\gb]: \gb\ \text{is a gallery in}\ \cW  \big\}=\big\{ [\gamma]: \gamma\ \text{is a geodesic in}\ \cW  \big\}}_{\text{\autoref{allgalgeo}}}.\] 

\begin{cor}
	Let $\cW$ be a pre-Weyl graph. Then the minimal galleries of $\cW$ are geodesics. 
\end{cor}

\begin{proof}
	Let $\gb$ be a minimal gallery. If $\gb_S$ is not reduced, then we can obtain a gallery from $\gb$ via a strict homotopy and a $1$-elementary contraction which is homotopic to $\gb$, and yet is shorter than $\gb$. This contradicts the minimality of $\gb$.  
\end{proof} 

%In buildings the converse holds; minimal galleries and geodesics coincide.

%%%%%%%%%%%%%%%%%%%%%%%%%%%%%%%%%%%%%%%%%%%%%%%%%%%%%%%%%%%%%%%%%%%%%%%%
\subsection{The $2$-Weyl Properties} \label{sec:2weyl}
%%%%%%%%%%%%%%%%%%%%%%%%%%%%%%%%%%%%%%%%%%%%%%%%%%%%%%%%%%%%%%%%%%%%%%%%

Let us introduce three more properties which arbitrary Weyl data may satisfy. These properties, together with ($2$W), are the rank $2$ cases of what we will call the Weyl properties (see \autoref{section:weylprop}).

\begin{enumerate}  [itemindent=0cm, leftmargin=1.9cm] \index{$2$-Weyl graph!$2$-Weyl properties}
	\item [\textbf{($2$C)}]
	homotopic alternating geodesics are strictly homotopic
	\item [\textbf{($2$SH)}]
	strictly homotopic alternating geodesics of the same type are equal (equivalently $F_{\gamma}$ is injective for all alternating geodesics $\gamma$)
	\item [\textbf{($2$H)}]
	homotopic alternating geodesics of the same type are equal. 
\end{enumerate}

\noindent We now show that for pre-Weyl graphs $\cW$, we have the following implications,
\[  (2\text{C})\iff(2\text{W}) \implies (2\text{H})\iff (2\text{SH}) .  \]
We begin with a simple observation.

\begin{lem}  \label{prop:two(st)stricthomo}
	Let $\cW$ be a pre-Weyl graph. If two homotopic maximal alternating geodesics of $\cW$ have the same type, then they are strictly homotopic.
\end{lem}

\begin{proof}
	Let $\rho(s,t)$ and $\rho'(s,t)$ be two homotopic maximal $(s,t)$-geodesics of $\cW$. By (PW$1$), $\rho(s,t)$ is homotopic to a maximal $(t,s)$-geodesic $\rho(t,s)$. Then $\rho'(s,t)\sim \rho(s,t)\simeq \rho(t,s)$, thus $\rho'(s,t)\simeq \rho(t,s)$. Then $\rho(s,t)\simeq  \rho(t,s)   \simeq \rho'(s,t)$, and so $\rho(s,t)\simeq \rho'(s,t)$.   
\end{proof}

We now prove $(2\text{H})\iff (2\text{SH})$.

\begin{prop} \label{prop:quiv2weyl}
	Let $\cW$ be a pre-Weyl graph. Then the following are equivalent,
	
	\begin{enumerate}[label=(\roman*)]
		\item
		$\cW$ has property ($2$SH)
		\item
		$\cW$ has property ($2$H)
		\item
		each maximal $(s,t)$-geodesic is homotopic to at most one maximal $(t,s)$-geodesic.
	\end{enumerate}
\end{prop}

\begin{proof}
	We have ($2$SH)$\implies$($2$H) since two unequal homotopic alternating geodesics of the same type can be extended by property (PW$0$) to give two unequal homotopic maximal alternating geodesics of the same type, which will be strictly homotopic by \autoref{prop:two(st)stricthomo}. The converse ($2$H)$\implies$($2$SH) is clear. 
	
	To see ($2$SH)$\implies$(iii), let $\rho(s,t)$ be a maximal $(s,t)$-geodesic which is homotopic to maximal $(t,s)$-geodesics $\rho$ and $\rho'$. Then $\rho$ and $\rho'$ are strictly homotopic by \autoref{prop:two(st)stricthomo}, and so $\rho=\rho'$ by ($2$SH). For the converse (iii)$\implies$($2$SH), suppose that $\cW$ does not have property ($2$SH). Then two unequal maximal $(t,s)$-geodesics $\rho$ and $\rho'$ are strictly homotopic.  By (PW$1$), $\rho$ is homotopic to a maximal $(s,t)$-geodesic $\rho(s,t)$, which by transitivity will also be homotopic to $\rho'$. Thus, $\cW$ does not have property (iii).
\end{proof}

We now prove $(2\text{W}) \implies (2\text{H})$.

\begin{prop}  \label{prop:WimpliesH}
	Let $\cW$ be a $2$-Weyl graph. Then $\cW$ has property ($2$H).  
\end{prop}

\begin{proof}
	Let $\rho$ and $\rho'$ be homotopic alternating geodesics in $\cW$ with the same type. Let $i$ be the last edge of $\rho$, and let $i'$ be the last edge of $\rho'$. Towards a contradiction, suppose that $i\neq i'$. Let $j=i'i^{-1}$. Let $\ga$ and $\ga'$ be the subgalleries such that $\ga i=\rho$ and $\ga 'i'=\rho'$. Then $\ga$ and $\ga'j$ are homotopic alternating geodesics with different $W$-lengths, a contradiction. Thus, $i=i'$. Then $\ga$ and $\ga'$ are also homotopic alternating geodesics, and so can apply the same argument to the penultimate edges of $\rho$ and $\rho'$. Therefore, by induction, we may conclude that $\rho=\rho'$.  
\end{proof}

\begin{remark}  \label{rem:wordhomouniquegalhomo}
	In particular, $2$-Weyl graphs have property (iii)  of \autoref{prop:quiv2weyl}. In the presence of (PW$1$), property (iii) implies that each maximal $(s,t)$-geodesic is homotopic to exactly one maximal $(t,s)$-geodesic. This just says that every maximal alternating geodesic is contained within exactly one suite. Therefore in $2$-Weyl graphs, elementary strict homotopies of words induce unique elementary strict homotopies of galleries. 
\end{remark} 

Finally, we prove $(2\text{C})\iff(2\text{W})$.

\begin{prop}  \label{prop:equiv2wand2c}
	Let $\cW$ be a pre-Weyl graph. Then the following are equivalent,
	
	\begin{enumerate}[label=(\roman*)]
		\item
		$\cW$ is $2$-Weyl
		\item
		$\cW$ has property ($2$C).
	\end{enumerate}
\end{prop}

\begin{proof}
	To see that ($2$W)$\implies$($2$C), just notice that if ($2$C) does not hold, then we must have distinct homotopic alternating geodesics of the same type, which contradicts ($2$H). The converse ($2$C)$\implies$($2$W) is clear.
\end{proof}

%%%%%%%%%%%%%%%%%%%%%%%%%%%%%%%%%%%%%%%%%%%%%%%%%%%%%%%%%%%%%%%%%%%%%%%%
%%%%%%%%%%%%%%%%%%%%%%%%%%%%%%%%%%%%%%%%%%%%%%%%%%%%%%%%%%%%%%%%%%%%%%%%
\section{Covering Theory of $2$-Weyl Graphs} \label{sec:coverofweyl}
%%%%%%%%%%%%%%%%%%%%%%%%%%%%%%%%%%%%%%%%%%%%%%%%%%%%%%%%%%%%%%%%%%%%%%%%
%%%%%%%%%%%%%%%%%%%%%%%%%%%%%%%%%%%%%%%%%%%%%%%%%%%%%%%%%%%%%%%%%%%%%%%%

We now develop covering theory of $2$-Weyl graphs. We show that coverings of $2$-Weyl graphs can be modeled with coverings of groupoids, which in turn shows that the (isomorphism classes of) coverings of a connected $2$-Weyl graph $\cW$ are naturally in bijection with the conjugacy classes of subgroups of the fundamental group of $\cW$. For details on coverings of groupoids, see \autoref{covgroup}.

%%%%%%%%%%%%%%%%%%%%%%%%%%%%%%%%%%%%%%%%%%%%%%%%%%%%%%%%%%%%%%%%%%%%%%%%
\subsection{Coverings of Weyl Data} 
%%%%%%%%%%%%%%%%%%%%%%%%%%%%%%%%%%%%%%%%%%%%%%%%%%%%%%%%%%%%%%%%%%%%%%%%

Given a chamber $C\in \cW$, let $\cW(C,-)$ denote the set of edges $i\in \cW$ such that $\iota(i)=C$. A morphism $\go:\wt{\cW}\to \cW$ of generalized chamber systems of type $M$ is called \emph{\'etale} if for each chamber $C\in \wt{\cW}$, the restriction $\go|_{\wt{\cW}(C,-)}$ of $\go$ to $\wt{\cW}(C,-)$ is a bijection into $\cW(\go(C),-)$. If $\go$ is additionally surjective on chambers, then $\go$ is called \emph{surjective-\'etale}. 

Given a morphism of generalized chamber systems $\go:\wt{\cW}\to \cW$, let $\go_s: \wt{\cW}_s\to \cW_s$ denote the induced homomorphism of panel groupoids of type $s$ (see \autoref{section:Generalized Chamber Systems and Weyl Data}). Notice that $\go$ is surjective-\'etale if and only if $\go_s$ is a covering of groupoids for each $s\in S$. We say a gallery $\gb$ in $\cW$ \emph{lifts} with respect to $\go$ to the gallery $\gb'$ if $\gb=\go \circ \gb'$.    

\begin{prop} \label{prop:Etale=Unique Gallery Lifting}
	Let $\go:\wt{\cW} \to \cW$ be an \'etale morphism of generalized chamber systems. Then for each chamber $\wt{C} \in \wt{\cW}$, every gallery $\gb$ in $\cW$ which issues from $\go(\wt{C})$ has a unique lifting with respect to $\go$ to a gallery $\til{\gb}$ which issues from $\wt{C}$. 
\end{prop} 

\begin{proof}
	Let $C=\go(\wt{C})$. First we prove existence. Let $i_1,\dots,i_n$ be the sequence of edges of $\gb$. For $k\in \{1,\dots,n\}$, let $\tilde{i}_k\in \wt{\cW}$ be an edge with $\go(\tilde{i}_k)=i_k$, $\iota(\tilde{i}_1)=C$, and $\tau(\tilde{i}_k)=\iota(\tilde{i}_{k+1})$. Notice that such $\tilde{i}_k$ exist by the fact that $\gb$ is \'etale. Then, letting $\til{\gb}$ be the gallery whose sequence of edges is $\tilde{i}_1,\dots,\tilde{i}_n$, we have $\gb=\go \circ \til{\gb}$ as required. For uniqueness, let $\til{\gb}=\tilde{i}_1,\dots ,\tilde{i}_n$ and $\til{\gb}'=\tilde{i}'_1,\dots ,\tilde{i}'_n$ be galleries in $\wt{\cW}$ issuing from $\wt{C}$, with $\go \circ \til{\gb}=\go \circ \til{\gb}'=\gb$. In particular $\go(\tilde{i}_1)=\go(\tilde{i}'_1)$, with $\tilde{i}_1$ and $\tilde{i}'_1$ issuing from the same chamber $\wt{C}$. Thus, $\tilde{i}_1=\tilde{i}'_1$ since $\go$ is \'etale. The same argument shows that $\tilde{i}_2=\tilde{i}'_2$, and so on along the sequences of edges. We conclude that $\tilde{i}_1\dots \tilde{i}_n=\tilde{i}'_1\dots \tilde{i}'_n$.  
\end{proof} 

\begin{prop} \label{propsemietiset}
	Let $\go:\wt{\cW} \to \cW$ be an \'etale morphism of generalized chamber systems. If $\cW$ is connected, then $\go$ is surjective-\'etale.
\end{prop}

\begin{proof}
	Suppose that $\cW$ is connected. To see that $\go$ is surjective on chambers, let $C\in \cW$ be any chamber. Pick a chamber $D\in \wt{\cW}$, and let $\gb$ be a gallery in $\cW$ which goes from $\go(D)$ to $C$. Lift $\gb$ to a gallery $\tilde{\gb}$ of $\wt{\cW}$ which issues from $D$. Then we have $\go(\tau(\tilde{\gb}))=C$.
\end{proof}

A \emph{pre-covering} $p:\wt{\cW} \to \cW$ from a generalized chamber system $\wt{\cW}$ to Weyl data $\cW$ is a surjective-\'etale morphism such that for all galleries $\gb$ in $\wt{\cW}$, if $p\circ \gb$ is a defining suite of $\cW$, then $\gb$ is a cycle. A \emph{covering} $p:\wt{\cW} \to \cW$ of Weyl data is a surjective-\'etale morphism such that for all galleries $\gb$ in $\wt{\cW}$, if $p\circ \gb$ is a defining suite of $\cW$, then $\gb$ is a suite of $\wt{\cW}$. A covering $p:\wt{\cW} \to \cW$ is called \emph{connected} if $\wt{\cW}$ is connected.  

Let $p:\wt{\cW} \to \cW$ be a pre-covering. The \emph{completion} of $\wt{\cW}$ with respect to $p$ is the Weyl data whose underlying generalized chamber system is $\wt{\cW}$, and whose suites are defined as follows; for any gallery $\gb$ of $\wt{\cW}$, if $p \circ \gb$ is a defining suite of $\cW$, then let $\gb$ be a defining suite of $\wt{\cW}$. With $\wt{\cW}$ redefined to be its completion, then $p:\wt{\cW} \to \cW$ is a covering of Weyl data. 

\begin{lem}  \label{lemma:homotopylift}
	Let $p:\wt{\cW} \to \cW$ be a covering of Weyl data. Let $\gb$ and $\gb'$ be galleries in $\cW$ with $\gb\sim \gb'$. Let $C=\iota(\gb)=\iota(\gb')$, and pick any $\wt{C}\in \wt{\cW}$ such that $p(\wt{C})=C$. Let $\til{\gb}$ and $\til{\gb}'$ be the unique lifts of $\gb$ and $\gb'$, respectively, issuing from $\wt{C}$. Then $\til{\gb}\sim \til{\gb}'$. 
\end{lem} 

\begin{proof}
	By hypothesis, there exists a sequence of galleries
	\[\gb=\gb_1,\dots,\gb_n=\gb'\] 
	such that consecutive galleries differ only by an elementary homotopy. Let
	\[\til{\gb}=\til{\gb}_1,\dots,\til{\gb}_n=\til{\gb}'\] 
	be the sequence of galleries obtained by lifting each $\gb_m$, for $m\in \{1,\dots,n\}$, to a gallery issuing from $\wt{C}$. Suppose that $\gb_m$ and $\gb_{m+1}$ differ by a $1$-elementary homotopy of type (i). Since $p$ is surjective-\'etale, for edges $i,i'\in \wt{\cW}$ with $\upsilon(i)=\upsilon(i')$, if $p(i;i')=p(i);p(i')=1$ then $i;i'=1$. Thus, backtracks lift to backtracks, and so $\til{\gb}_m\sim \til{\gb}_{m+1}$. Suppose that $\gb_m$ and $\gb_{m+1}$ differ by a $1$-elementary homotopy of type (ii). But for edges $j,j',k\in \wt{\cW}$ with $\upsilon(j)=\upsilon(j')=\upsilon(k)$, if $p(j);p(j')=p(k)$ then $j;j'=k$ by the fact that $p$ is surjective-\'etale. Thus, detours lift to detours, and so $\til{\gb}_m\sim \til{\gb}_{m+1}$. If $\gb_m$ and $\gb_{m+1}$ differ by a $2$-elementary homotopy, then $\til{\gb}_m\sim \til{\gb}_{m+1}$ by the fact that $p$ lifts defining suites to suites. Therefore, by transitivity, we have $\til{\gb}\sim \til{\gb}'$.   
\end{proof}   

\begin{cor}  \label{corlift}
	Let $p:\wt{\cW} \to \cW$ be a covering of Weyl data. Then for all galleries $\gb$ in $\wt{\cW}$, if $p\circ \gb$ is a suite of $\cW$, then $\gb$ is a suite of $\wt{\cW}$.
\end{cor}

\begin{proof}
	We have $p\circ \gb\sim \gb_0$, where $\gb_0$ is a trivial gallery. Then, by \autoref{lemma:homotopylift}, $\gb$ is homotopic to the lifting of $\gb_0$, which is trivial. Thus, $\gb$ is null-homotopic. 
\end{proof}

We have the following characterization of isomorphisms amongst coverings.

\begin{prop} \label{prop:charofisoaomgcovweyl}
	Let $p:\wt{\cW} \to \cW$ be a covering of Weyl data. Then $p$ is an isomorphism if and only if $p$ is injective on chambers.
\end{prop}

\begin{proof}
	Clearly, if $p$ is an isomorphism then $p$ is injective on chambers. Conversely, suppose that $p$ is injective on chambers. By \autoref{prop:charisoofweyl}, it suffices to show that $p$ is bijective on chambers and edges. This follows from that fact that each $p_s$ is an isomorphism by \autoref{prop:injecvert}. 
\end{proof}

We now show that coverings preserve and reflect the property of being $2$-Weyl amongst Weyl data.

\begin{lem}
	Let $p:\wt{\cW}\to \cW$ be a surjective-\'etale morphism of Weyl data. Then $\wt{\cW}$ has property (PW$0$) if and only if $\cW$ has property (PW$0$).
\end{lem}

\begin{proof}
	This follows from the fact that each $s$-homomorphism $p_s:\wt{\cW}_s\to \cW_s$ is a covering of groupoids.
\end{proof}

\begin{lem}
	Let $p:\wt{\cW}\to \cW$ be a surjective-\'etale morphism of Weyl data. If $\wt{\cW}$ has property (PW$1$), then $\cW$ has property (PW$1$).
\end{lem}

\begin{proof}
	Let $\rho$ be a maximal $(s,t)$-geodesic of $\cW$. Lift $\rho$ to a maximal $(s,t)$-geodesic $\til{\rho}$ of $\wt{\cW}$. Let $\til{\rho}'$ be a maximal $(t,s)$-geodesic which is homotopic to $\til{\rho}$. Then $p\circ \til{\rho}'$ is a maximal $(t,s)$-geodesic which is homotopic to $\rho$.
\end{proof}

\begin{lem}
	Let $p:\wt{\cW}\to \cW$ be a covering of Weyl data. If $\cW$ has property (PW$1$), then $\wt{\cW}$ has property (PW$1$).
\end{lem}

\begin{proof}
	Let $\tilde{\rho}$ be a maximal $(s,t)$-geodesic of $\wt{\cW}$. Let $\rho=p \circ \tilde{\rho}$, and let $\rho'$ be a maximal $(t,s)$-geodesic which is homotopic to $\rho$. Lift $\rho'$ to the gallery $\tilde{\rho}'$ which issues from the same chamber as $\tilde{\rho}$. Then, since homotopies lift with respect to coverings, $\tilde{\rho}$ and $\tilde{\rho}'$ are homotopic.
\end{proof}

\begin{lem}
	Let $p:\wt{\cW}\to \cW$ be a morphism of Weyl data. If $\cW$ has property ($2$W), then $\wt{\cW}$ has property ($2$W).
\end{lem}

\begin{proof}
	Let $\til{\gamma}$ and $\til{\gamma}'$ be two homotopic alternating geodesics in $\wt{\cW}$. Let $\gamma=  p  \circ   \til{\gamma}$ and $\gamma'=  p  \circ   \til{\gamma}'$. Then $\gamma$ and $\gamma'$ are homotopic alternating geodesics in $\cW$, and $\til{\gamma}_W=\gamma_W=\gamma'_W=\til{\gamma}'_W$, since $\cW$ has property ($2$W).  
\end{proof}

\begin{lem}
	Let $p:\wt{\cW}\to \cW$ be an covering of Weyl data. If $\wt{\cW}$ has property ($2$W), then $\cW$ has property ($2$W).
\end{lem}

\begin{proof}
	Let $\gamma$ and $\gamma'$ be two homotopic alternating geodesics in $\cW$. Lift these geodesics to homotopic alternating geodesics $\til{\gamma}$ and $\til{\gamma}'$ in $\wt{\cW}$. Then $\gamma_W=\til{\gamma}_W=\til{\gamma}'_W=\gamma'_W$, since $\wt{\cW}$ has property ($2$W). 
\end{proof} 

Together these lemmas prove the following.

\begin{thm} \label{prop:presandreflect}
	Let $p:\wt{\cW}\to \cW$ be a covering of Weyl data. Then $\wt{\cW}$ is $2$-Weyl if and only if $\cW$ is $2$-Weyl.
\end{thm}

We now show that a covering of Weyl data induces local coverings of the residues.

\begin{prop}
	Let $p:\wt{\cW}\to \cW$ be a covering of Weyl data. Let $\wt{R}$ be a $J$-residue of $\wt{\cW}$, and let $R$ be the $J$-residue of $\cW$ which contains the $p$-image of $\wt{R}$. Then the restriction of $p$ to $\wt{R}$ is a covering of $R$.    
\end{prop}

\begin{proof}
	Let $p_{\wt{R}}:\wt{R} \to R$ be the morphism of Weyl data which is the restriction of $p$ to $\wt{R}$. Clearly $p_{\wt{R}}:\wt{R} \to R$ is \'etale, and so is surjective-\'etale by \autoref{propsemietiset}. The fact that $p_{\wt{R}}$ is a covering then follows directly from the fact that $p:\wt{\cW}\to \cW$ is a covering.    
\end{proof} 

We call $p_{\wt{R}}:\wt{R} \to R$ the \emph{local covering at $\wt{R}$}. We finish this section with a result which shows that, as long as one establishes the $2$-Weyl property, surjective-\'etale morphisms suffice when it comes to coverings.

\begin{prop} \label{lemma:2weyletale}
	Let $\cW$ be a $2$-Weyl graph, and let $\wt{\cW}$ be a pre-Weyl graph. Then a surjective-\'etale morphism $\go:\wt{\cW}\to \cW$ is a covering.  
\end{prop}

\begin{proof}
	Let $\gb$ be a gallery in $\wt{\cW}$ which descends to a defining suite. Towards a contradiction, suppose that $\gb$ is not a suite. Since $\gb$ has the type of a suite, we may write $\gb=\rho(s,t) \rho(t,s)^{-1}$. Let $\hat{\rho}$ be a gallery which is strictly homotopic to $\rho(s,t)$, which must exist by (PW$1$).  We have $\hat{\rho}\neq \rho(t,s)$ by hypothesis. Since $\go$ is surjective-\'etale, $\go$ preserves the distinction of galleries, thus $\go \circ \hat{\rho}\neq \go \circ \rho(t,s)$. But $\go \circ \hat{\rho} \sim \go \circ  \rho(s,t)$ since $\hat{\rho} \sim  \rho(s,t)$ by our choice of $\hat{\rho}$, and homotopies descend through morphisms. Also
	\[\go \circ \rho(s,t) \sim \go \circ \rho(t,s)\] 
	\noindent since $\go \circ \gb$ is a suite, and
	\[(\go\circ \rho(s,t))^{-1} (\go\circ \rho(t,s))=\go\circ\gb.\] 
	\noindent Therefore $\go\circ \hat{\rho}\sim \go \circ \rho(t,s)$ by transitivity. This is a contradiction of the fact that $\cW$ is $2$-Weyl, since both $\go \circ \hat{\rho}$ and $\go \circ \rho(t,s) $ have the same type (see \autoref{prop:WimpliesH}).  
\end{proof}

\begin{cor} 
	A surjective-\'etale morphism between $2$-Weyl graphs is a covering. 
\end{cor}

%%%%%%%%%%%%%%%%%%%%%%%%%%%%%%%%%%%%%%%%%%%%%%%%%%%%%%%%%%%%%%%%%%%%%%%%
\subsection{Classification of Coverings of $2$-Weyl Graphs}
%%%%%%%%%%%%%%%%%%%%%%%%%%%%%%%%%%%%%%%%%%%%%%%%%%%%%%%%%%%%%%%%%%%%%%%%

We now move to the setting where Weyl data $\cW$ is $2$-Weyl. Recall that in this setting, the panel groupoids of $\cW$ are naturally subgroupoids of the fundamental groupoid of $\cW$. We construct a bijective correspondence between connected coverings of $2$-Weyl graphs and conjugacy classes of subgroups of their fundamental groups.

\begin{prop} \label{lem:detbys}
	Let $\cW$ be a $2$-Weyl graph. Then the edges of the panel groupoids of $\cW$ generate $\overline{\cW}$.
\end{prop}

\begin{proof}
	Let $[\gb] \in \overline{\cW}$ be a homotopy class of galleries, and let $i_1,\dots,i_n$ be the sequence of edges of $\gb$. Then in $\overline{\cW}$, we have, 
	\[[\gb]=[i_1\dots i_n]=   i_1;\dots ;i_n. \qedhere  \]
\end{proof}

\begin{prop} \label{prop:rightdefofcov}
	Let $p:\wt{\cW} \to \cW$ be a covering of $2$-Weyl graphs. Then the induced homomorphism $\bar{p}$ of the fundamental groupoids is a covering of groupoids.
\end{prop} 

\begin{proof}
	Firstly, $\bar{p}$ is surjective on vertices since $p$ is surjective on chambers. Let $\wt{C}\in \wt{\cW}$ be a chamber and let $C=p(\wt{C})$. Let $\gb$ be a gallery of $\cW$ which issues from $C$, and let $\tilde{\gb}$ be the lifting of $\gb$ to a gallery which issues from $\wt{C}$. Then
	\[\bar{p} \big  ([\tilde{\gb}]  \big )=  [p\circ \til{\gb}]  =  [\gb].\] 
	Therefore the restriction of $\bar{p}$ to the homotopy classes which issue from $\wt{C}$ is surjective into the homotopy classes which issue from $C$. Finally, towards a contradiction, suppose that the restriction of $\bar{p}$ to the homotopy classes which issue from $\wt{C}$ is not injective. This implies that there exist non-homotopic galleries in $\wt{\cW}$ issuing from $\wt{C}$ whose $p$-images are homotopic. This contradicts \autoref{lemma:homotopylift}.  
\end{proof} 

Let $p:\wt{\cW}\to \cW$ and $p':\wt{\cW}'\to \cW$ be coverings of a $2$-Weyl graph $\cW$. A \emph{morphism} of coverings $\mu:p\to p'$ is a morphism $\mu:\wt{\cW}\to \wt{\cW}'$ such that $p=p'\circ \mu$. We call two coverings \emph{isomorphic} if there exists
a morphism between them which is an isomorphism of $2$-Weyl graphs. The \emph{composition} of morphisms of coverings is just their composition as morphisms of $2$-Weyl graphs. 

\begin{prop} \label{prop:importantresult}
	Let $p:\wt{\cW}\to \cW$, $p':\wt{\cW}'\to \cW$, and $\mu:\wt{\cW}\to \wt{\cW}'$ be morphisms of connected $2$-Weyl graphs with $p = p' \circ \mu$. If $p$ and $p'$ are coverings, then $\mu$ is a covering, and if $p$ and $\mu$ are coverings, then $p'$ is a covering.
\end{prop}

\begin{proof}
	Suppose that $p$ and $p'$ are coverings. For any chamber $C\in \wt{\cW}$, we have, 
	\[p |_{\cW(C,-)}=p' |_{\cW(\mu(C),-)}  \circ\  \mu |_{\cW(C,-)} .\]
	Then, since $p |_{\cW(C,-)}$ and $p' |_{\cW(\mu(C),-)}$ are bijections, we must have that $\mu |_{\cW(C,-)}$ is a bijection for all  $C\in \wt{\cW}$. Therefore $\mu$ is \'etale, and so is surjective-\'etale because $\wt{\cW}'$ is connected. To see that $\mu$ is a covering, let $\gb$ be a gallery in $\wt{\cW}$ such that $\mu\circ \gb$ is a defining suite of $\wt{\cW}'$. Therefore $p' \circ(\mu\circ \gb)$ is a suite of $\cW$, since $p'$ is a morphism. But $p \circ \gb =p' \circ\mu\circ \gb$, and so $\gb$ is a suite since $p$ is a covering and coverings lift suites to suites by \autoref{corlift}. 
	
	Now suppose that $p$ and $\mu$ are coverings. Let $D\in \wt{\cW}'$ be any chamber. Pick a chamber $C\in \wt{\cW}$ such that $\mu(C)=D$ (here we use the fact that $\mu$ is surjective on chambers). Then we have, 
	\[p |_{\cW(C,-)}=p' |_{\cW(D,-)}  \circ\  \mu |_{\cW(C,-)} .\]
	Then, since $p |_{\cW(C,-)}$ and $\mu |_{\cW(C,-)}$ are bijections, we must have that $p' |_{\cW(D,-)}$ is a bijection. Therefore $p'$ is \'etale, and so is surjective-\'etale because $\cW$ is connected. To see that $p'$ is a covering,  let $\gb$ be a gallery in $\wt{\cW}'$ such that $p'\circ \gb$ is a defining suite of $\cW$. Lift $\gb$ with respect to $\mu$ to a gallery $\tilde{\gb}$ in $\wt{\cW}$. Then $p \circ \tilde{\gb}=p' \circ \gb$. Since $p$ is a covering, $p \circ \tilde{\gb}$ must be a suite of $\cW$. Then, since $\mu$ is a morphism, $\gb$ must be a suite of $\wt{\cW}'$.     
\end{proof}

\begin{prop} \label{prop:faithful}
	Let $\go,\go':\cW\to \cW'$ be two morphisms between the same $2$-Weyl graphs. If $\bar{\go}=\bar{\go}'$, then $\go=\go'$.
\end{prop}

\begin{proof}
	Towards a contradiction, suppose that $\go\neq\go'$. Then there exists an edge $i\in \cW$ with $\go(i)\neq\go'(i)$. We cannot have $\go(i)\sim\go'(i)$ as galleries because $\cW'_{\upsilon(i)}\leq \overline{\cW}'$. Therefore
	\[\bar{\go}([i])=     [\go(i)]\neq     [\go'(i)]         =\bar{\go}'([i])\] 
	and so $\bar{\go}\neq \bar{\go}'$, a contradiction.
\end{proof}

\begin{prop} \label{prop:genliftweyl}
	Let $p:\wt{\cW}\to \cW$ and $p':\wt{\cW}'\to \cW$ be coverings of a $2$-Weyl graph $\cW$. Let $\lambda:\bar{p}\to \bar{p}'$ be a covering morphism of the groupoid coverings $\bar{p}$ and $\bar{p}'$. Then there exists a unique covering morphism $\mu:p\to p'$ such that $\bar{\mu}=\lambda$.  
\end{prop}

\begin{proof}
	Let $\mu:\wt{\cW}\to \wt{\cW}'$ be the generalized chamber system morphism whose $s$-homomorphism $\mu_s$ is the restriction of $\lambda$ to $\wt{\cW}_s$. This is well-defined because the restriction of $\lambda$ preserves types by the fact its a covering morphism. Then $\mu$ is a morphism of Weyl data since a suite, as a sequence of edges whose composition in the fundamental groupoid of $\wt{\cW}$ is trivial, must get mapped by $\lambda$ to a sequence of edges whose composition is also trivial. Then $\bar{\mu}=\lambda$ since they agree on a generating set (see \autoref{lem:detbys}). Finally, $\mu$ is unique by \autoref{prop:faithful}.
\end{proof}

\begin{cor} \label{cor:isocovers}
	Let $\lambda:\bar{p}\to \bar{p}'$ be an isomorphism, and let $\mu:p\to p'$ be the unique morphism such that $\bar{\mu}=\lambda$. Then $\mu$ is an isomorphism. Thus, if two coverings of a $2$-Weyl graph induce isomorphic groupoid coverings, then they are isomorphic. 
\end{cor}

\begin{proof}
	Let $\mu':p'\to p$ be the unique morphism with $\bar{\mu}'=\lambda^{-1}$. Put $\mu''=\mu'\circ \mu$. Then $\bar{\mu}''=\lambda^{-1} \circ\lambda=1$. Thus, $\mu''=1$ by \autoref{prop:faithful}. By a symmetric argument, $\mu\circ \mu'=1$, and so $\mu$ is an isomorphism $\mu:p\to p'$.   
\end{proof}

This shows that coverings of $2$-Weyl graphs can be modeled in a non-forgetful way by coverings of groupoids. We now show that the injective correspondence we've just established is bijective.

\begin{thm} \label{existenceofweylcover}
	Let $\cW$ be a $2$-Weyl graph. For every groupoid covering $\gf:\cG\to\overline{\cW}$, there exists a covering $p:\wt{\cW}\to\cW$ with $\bar{p}\cong \gf$.
\end{thm} 

\begin{proof}
	Let the chambers of $\wt{\cW}$ be the vertices of $\cG$. Let the panel groupoid $\wt{\cW}_s$ be the subgroupoid $\cG_s\leq \cG$ which is the $\gf$-preimage of $\cW_s\leq \overline{\cW}$. This gives $\wt{\cW}$ the structure of a generalized chamber system. Let $p:\wt{\cW}\to \cW$ be the generalized chamber system morphism such that the $s$-homomorphism $p_s$ is the restriction of $\gf$ to $\cG_s$. The map $p:\wt{\cW}\to \cW$ is surjective-\'etale because $\gf$ is a groupoid covering. It is a pre-covering since a suite in $\cW$ lifts to a loop in $\cG$ because the composition of the sequence of edges of a suite is trivial. Finally, redefine $\wt{\cW}$ to be its completion with respect to $p$. Thus, we obtain a covering $p:\wt{\cW}\to \cW$. We have $\bar{p}\cong \gf$ by \autoref{lem:detbys}.%Let $[K]$ be any conjugacy class of subgroups of $\pi_1(\cW)$. Pick a subgroup $H$ of $\overline{\cW}$ whose image in $\pi_1(\cW)=\pi_1(\overline{\cW})$ is $[K]$. Then the covering $p:\wt{\cW}^H\to \cW$ based at $H$ is the required covering.  
\end{proof}

Let $p:\wt{\cW}\to \cW$ be a connected covering of $2$-Weyl graphs. Then $p$ is called a \emph{universal cover} if for any connected covering $p':\wt{\cW}'\to \cW$, there exists a covering morphism $\mu:p\to p'$. 

\begin{prop} \label{prop:charuni}
	A connected covering $p:\wt{\cW}\to \cW$ of $2$-Weyl graphs is a universal cover if and only if $\wt{\cW}$ is simply connected. 
\end{prop}

\begin{proof}
	Suppose that $p:\wt{\cW}\to \cW$ is a universal cover. Then $\bar{p}$ is a universal cover of groupoids. Therefore, the universal groupoid of $\wt{\cW}$ is a setoid, and so $\wt{\cW}$ is simply connected. Conversely, suppose that $\wt{\cW}$ is simply connected. Again, this means that $\bar{p}$ is a universal cover of groupoids. Then $p$ is a universal cover by \autoref{prop:genliftweyl}. 
\end{proof}

We often denote simply connected Weyl data by $\Delta$.

\begin{cor}
	Universal covers are unique up to isomorphism, and every connected $2$-Weyl graph $\cW$ has a universal cover $p:\Delta\to \cW$.
\end{cor}

\begin{proof}
	Universal covers are unique up to isomorphism by \autoref{prop:charuni} and the 1-1 correspondence between $2$-Weyl graph coverings and groupoid coverings. The existence of universal covers follows from \autoref{prop:charuni} and \autoref{existenceofweylcover}.
\end{proof}

\begin{cor} \label{cor:isoofunicovers}
	Let $p:\Delta\to \cW$ and $p':\Delta'\to \cW$ be universal covers of a connected $2$-Weyl graph $\cW$, and let $\mu:p\to p'$ be a covering morphism. Then $\mu$ is an isomorphism. 
\end{cor}

\begin{proof}
	Notice that $\pi_1(\bar{\mu})$ is an outer isomorphism between trivial groups. Then $\bar{\mu}$ is an isomorphism by \autoref{prop:conservative}, and $\mu$ is an isomorphism by \autoref{cor:isocovers}.
\end{proof}

%%%%%%%%%%%%%%%%%%%%%%%%%%%%%%%%%%%%%%%%%%%%%%%%%%%%%%%%%%%%%%%%%%%%%%%%
\subsection{Coverings and Group Actions}
%%%%%%%%%%%%%%%%%%%%%%%%%%%%%%%%%%%%%%%%%%%%%%%%%%%%%%%%%%%%%%%%%%%%%%%%

We now describe the relationship between groups acting on $2$-Weyl graphs and coverings. Groups act on $2$-Weyl graphs by automorphisms. We say a group $G$ acts \emph{chamber-freely} on a $2$-Weyl graph $\cW$ if the restriction of the action of $G$ to the set of chambers of $\cW$ is free. 

Let $p:\wt{\cW}\to \cW$ be a covering of $2$-Weyl graphs. The \emph{deck transformation group} $\Aut(p)$ of $p$ is the group of automorphisms of $\wt{\cW}$ which commute with $p$. Since $\Aut(p)\leq \Aut(\wt{\cW})$, a covering $p:\wt{\cW}\to \cW$ determines a faithful action of $\Aut(p)$ on the left of $\wt{\cW}$. Notice that the homomorphism $\Aut(p)\to \Aut(\bar{p})$, $g\mapsto \bar{g}$, is injective by \autoref{prop:faithful}, and surjective by \autoref{prop:genliftweyl}. 

\begin{prop} 
	Let $p:\wt{\cW}\to \cW$ be a connected covering. Then $\Aut(p)$ acts chamber-freely on $\wt{\cW}$.   
\end{prop}

\begin{proof}
	Given the natural isomorphism $\Aut(p)\to \Aut(\bar{p})$, $g\mapsto \bar{g}$, this follows from \autoref{prop:freeaction}.
\end{proof}

We will see that conversely, if a group acts chamber-freely on a connected $2$-Weyl graph $\cW'$, then it is naturally the deck transformation group of a covering $\cW'\to \cW$. A connected covering $p:\wt{\cW}\to \cW$ of $2$-Weyl graphs is called \emph{regular} if the induced groupoid covering $\bar{p}$ is regular. Thus, if $p$ is regular, then we can identify $\pi_1(\wt{\cW})$ with its $\pi_1(\bar{p})$-image in $\pi_1(\cW)$. 

\begin{prop} \label{prop:trans}
	Let $p:\wt{\cW}\to \cW$ be a regular covering of $2$-Weyl graphs. Then the action of $\Aut(p)$ restricted to the $p$-preimage of a chamber or an edge is regular.
\end{prop}

\begin{proof}
	Given the natural isomorphism $\Aut(p)\to \Aut(\bar{p})$, $g\mapsto \bar{g}$, the case of a chamber follows from \autoref{prop:fib}. In the case of an edge, let $i,i'\in \wt{\cW}$ be edges with $p(g)=p(g')$. Let $g\in \Aut(p)$ such that $g \cdot \iota(i)=\iota(i')$. Then $g \cdot i= i'$, since $p$ is a covering.
\end{proof}

\begin{prop}
	Let $p:\wt{\cW}\to \cW$ be a regular covering of $2$-Weyl graphs. Then there exists a natural outer isomorphism, 
	\[  \pi_1(\cW)/\pi_1(\wt{\cW}) \to \Aut(p).\] 
\end{prop}

\begin{proof}
	Given the natural isomorphism $\Aut(p)\to \Aut(\bar{p})$, $g\mapsto \bar{g}$, this follows from \autoref{prop:natouteriso}.
\end{proof}

We now show that, conversely, if a group $G$ acts chamber-freely on a connected $2$-Weyl graph $\cW$, then there exists a $2$-Weyl graph $\cW'$ and a regular covering $\cW \to \cW'$ of which $G$ is naturally the automorphism group. We associate to the chamber-free action of a group $G$ on a connected $2$-Weyl graph $\cW$ the \emph{quotient} $2$-Weyl graph $G\backslash \cW$, which is the $2$-Weyl graph defined as follows. First, we define a generalized chamber system $G\backslash \cW$ by letting the set of chambers be the set of orbits $G\backslash \cW_0$, and letting the panel groupoid of type $s$ be the quotient groupoid $G\backslash \cW_s$. The \emph{quotient map} 
\[
\pi:\cW\to G\backslash \cW
\] 
is the morphism of generalized chamber systems such that the $s$-homomorphism $\pi_s$ is the groupoid quotient map $\pi_s:\cW_s\to G\backslash \cW_s$. Notice that $\pi:\cW\to G\backslash \cW$ is surjective-\'etale since the $\pi_s$ are coverings by \autoref{prop:quot}.  

We give $G\backslash \cW$ the structure of Weyl data by letting the suites be the $\pi$-images of suites in $\cW$, i.e. $\theta$ is a defining suite of $G\backslash \cW$ if there exists a defining suite $\theta'$ of $\cW$ with $\theta=\pi\circ\theta'$. This ensures that $\pi$ is a morphism of Weyl data.  

\begin{prop}
	If a group $G$ acts chamber-freely on a $2$-Weyl graph $\cW$, then $G\backslash \cW$ is a $2$-Weyl graph, and the quotient map $\pi:\cW\to G\backslash \cW$ is a covering.    
\end{prop} 

\begin{proof}
	Notice that $\pi:\cW\to G\backslash \cW$ is a covering since all the defining suites of $G\backslash \cW$ are $\pi$-images of defining suites of $\cW$. Then $G\backslash \cW$ is $2$-Weyl by \autoref{prop:presandreflect}. 
\end{proof}

If a group $G$ acts chamber-freely on a $2$-Weyl graph $\cW$, then this induces a free action of $G$ on $\overline{\cW}$. The following result shows that our $2$-Weyl graph quotient construction is compatible with the groupoid quotient construction.

\begin{prop}\label{prop:compat}
	Let $G$ be a group which acts chamber-freely on a $2$-Weyl graph $\cW$ with quotient map $\pi: \cW\to G\backslash \cW$. Let $\gf:\overline{\cW}\to G\backslash\overline{\cW}$ be the quotient map of the associated action of $G$ on $\overline{\cW}$.  Then there exists a unique isomorphism $\psi:\overline{G\backslash \cW}\to G\backslash \overline{\cW}$ such that $\gf=\psi\circ \bar{\pi}$.
\end{prop} 

\begin{proof}
	Since $\gf=\psi\circ \bar{\pi}$, then $\psi:\overline{G\backslash \cW}\to G\backslash \overline{\cW}$ must be the homomorphism whose map on edges is 
	\[\bar{\pi}(g)\mapsto \gf(g), \qquad  \text{for } g\in \overline{\cW}_1 .\] 
	This is well-defined because $\gf$ is constant on $G$-orbits. Checking that $\psi$ is a homomorphism, we have
	\[\bar{\pi}(g)\bar{\pi}(g')=\bar{\pi}(gg')\mapsto \gf(gg') =\gf(g)\gf(g')  .    \]
	\noindent Then $\psi$ is a covering by \autoref{prop:groupoidcover}. Notice that $\psi$ is injective on chambers since both $\gf$ and $\bar{\pi}$ identify chambers if and only if they are in the same $G$-orbit. Therefore $\psi$ is an isomorphism by \autoref{prop:injecvert}.
\end{proof}

\begin{cor}
	Let $G$ be a group which acts chamber-freely on a connected $2$-Weyl graph $\cW$, and let $\pi:\cW\to G\backslash \cW$ be the associated quotient map. Then $G$ is naturally $\Aut(\pi)$. 
\end{cor}

\begin{proof}
	We have a natural embedding $G\hookrightarrow \Aut(\pi)$. The composition of this embedding with the isomorphism $\Aut(\pi)\to\Aut(\bar{\pi})$ is an isomorphism by \autoref{prop:quot} and \autoref{prop:compat}. Therefore $G\hookrightarrow \Aut(\pi)$ is an isomorphism.
\end{proof}

\begin{prop} \label{prop:campatcoverings}
	Let $p:\wt{\cW}\to \cW$ be a regular covering, and let $\pi:\wt{\cW}\to \Aut(p)\backslash \wt{\cW}$ be the quotient associated to the action of $\Aut(p)$. Then there exists a unique isomorphism $\go: \Aut(p)\backslash \wt{\cW}\to \cW$ such that $p = \go \circ \pi$.
\end{prop}

\begin{proof}
	Since $p = \go \circ \pi$, then $\go:\Aut(p)\backslash\wt{\cW}\rightarrow  \cW$ must be the map such that for chambers, 
	\[\pi(C)\mapsto p(C), \qquad   \text{for } C\in \wt{\cW}_0\] 
	and for edges, 
	\[\pi(i)\mapsto p(i), \qquad \text{for } i\in \wt{\cW}_1.\] 
	This is well-defined since $p$ is constant on $\Aut(p)$-orbits. Checking that $\go$ is a graph morphism, for the extremities we have 
	\[          \iota (\pi(i))=\pi(\iota(i)) \mapsto  p(\iota(i))=\iota(p(i))  ,\qquad \tau (\pi(i))=\pi(\tau(i)) \mapsto  p(\tau(i))=\tau(p(i))         \]
	and for the type function we have
	\[          \upsilon (\pi(i))=\upsilon(i) =  \upsilon (p(i)) . \]
	Checking that $\go$ is a morphism of generalized chamber systems, we have
	\[             \pi(i^{-1})\mapsto p(i^{-1})=p(i)^{-1}             ,\qquad  \pi(i);\pi(i')=\pi(i;i')\mapsto   p(i;i')=p(i);p(i') .          \]
	Then $\go$ is a bijection on chambers since it has the inverse $C\mapsto \pi(p^{-1}(C))$, for $C\in \cW_0$, and $\go$ is a bijection on edges since it has the inverse $i\mapsto \pi(p^{-1}(i))$, for $i\in \cW_1$. Notice that these inverses are well-defined by \autoref{prop:trans}. Therefore $\go$ is an isomorphism of generalized chamber systems. 
	
	Let $\theta$ be a suite of $\Aut(p)\backslash\wt{\cW}$. Lift $\theta$ to a gallery $\tilde{\theta}$ in $\wt{\cW}$. Then $\tilde{\theta}$ is a suite since $\pi$ is a covering. Then $p \circ \tilde{\theta}=\theta$ is also a suite since $p$ is a morphism of Weyl data. Therefore $\go$ is an morphism of Weyl data. Then $\go$ is a covering by \autoref{prop:importantresult}, and so $\go$ is an isomorphism of Weyl data.   
\end{proof}

Let $G$ be a group which acts chamber-freely on a $2$-Weyl graph $\cW$. Let $R$ be a residue of $\cW$. Then the \emph{isotropy} $H_R$ of $R$ is the subgroup,
\[H_R=\{  g\in G:g\cdot R=R    \}\leq G.\]
We now show that the local covering at $R$ of the quotient map $\cW\to G\backslash \cW$ is naturally the quotient map $R\to H_R\backslash R$.

\begin{prop}\label{quoofres}
	Let $G$ be a group which acts chamber-freely on a $2$-Weyl graph $\cW$, and let $\pi: \cW\to G\backslash \cW$ be the associated quotient. Let $R$ be a residue of $\cW$, and let $\pi_R$ denote the local covering at $R$. Let $H=H_R\leq G$ be the isotropy of $R$. Let $\pi_H: R\to H\backslash R$ be the quotient map associated to the action of $H$ on $R$. Then there exists a unique isomorphism $\go$ such that $\pi_R=\go \circ \pi_H$. 
\end{prop}

\begin{proof}
	Since $\pi_R=\go \circ \pi_H$, then $\go$ must be the map such that for chambers, 
	\[\pi_H(C)\mapsto \pi_R(C),\qquad   \text{for } C\in R_0\]
	and for edges,
	\[       \pi_H(i)\mapsto \pi_R(i),\qquad  \text{for } i\in R_1.         \]
	This is well-defined because $\pi_R$ is constant on $H$-orbits. Also, $\go$ is bijective on chambers and edges by the definition of $H$. Then $\go$ is an isomorphism of Weyl data by the same arguments as those in the proof of \autoref{prop:campatcoverings}.   
\end{proof}

%%%%%%%%%%%%%%%%%%%%%%%%%%%%%%%%%%%%%%%%%%%%%%%%%%%%%%%%%%%%%%%%%%%%%%%%
%%%%%%%%%%%%%%%%%%%%%%%%%%%%%%%%%%%%%%%%%%%%%%%%%%%%%%%%%%%%%%%%%%%%%%%%
\section{Weyl Graphs} \label{sec:weylgraphs} 
%%%%%%%%%%%%%%%%%%%%%%%%%%%%%%%%%%%%%%%%%%%%%%%%%%%%%%%%%%%%%%%%%%%%%%%%
%%%%%%%%%%%%%%%%%%%%%%%%%%%%%%%%%%%%%%%%%%%%%%%%%%%%%%%%%%%%%%%%%%%%%%%%

%%%%%%%%%%%%%%%%%%%%%%%%%%%%%%%%%%%%%%%%%%%%%%%%%%%%%%%%%%%%%%%%%%%%%%%%
\subsection{The Weyl Properties}  \label{section:weylprop}  
%%%%%%%%%%%%%%%%%%%%%%%%%%%%%%%%%%%%%%%%%%%%%%%%%%%%%%%%%%%%%%%%%%%%%%%%

We have the following stronger versions of the $2$-Weyl properties from \autoref{sec:2weyl},\footnote{ the $2$-Weyl properties are just the restriction of the Weyl properties to $2$-residues}
\begin{enumerate}  [itemindent=0cm, leftmargin=1.9cm]
	\item [\textbf{(W)}]
	homotopic geodesics have the same $W$-length
	\item [\textbf{(C)}]
	homotopic geodesics are strictly homotopic
	\item [\textbf{(SH)}]
	strictly homotopic geodesics of the same type are equal
	\item [\textbf{(H)}]
	homotopic geodesics of the same type are equal.
\end{enumerate}
We call these four properties the \emph{Weyl properties}. Properties similar to these feature in \cite{tits81local}, \cite{ronanlectures}. %Notice that (SH) says that the $\gamma$-gallery map $F_{\gamma}$ is injective for all geodesics $\gamma$.   
We will show that for $2$-Weyl graphs, we have the following implications,
\[\text{(C)} \implies \text{(W)} \implies \text{(H)}  \implies \text{(SH)} \implies \text{(C)}.\] 
In fact, the only implication which is not straightforward is $\text{(SH)} \implies \text{(C)}$. Finally, we have a property which will characterize buildings amongst connected $2$-Weyl graphs,  
\begin{enumerate}  [itemindent=0cm, leftmargin=1.9cm]
	\item [\textbf{(B)}]
	geodesics with the same extremities have the same $W$-length.
\end{enumerate}
If Weyl data has property (B), then geodesics can be used to give a well-defined notion of `distance' between chambers, whose value is an element of $W$.  

%%%%%%%%%%%%%%%%%%%%%%%%%%%%%%%%%%%%%%%%%%%%%%%%%%%%%%%%%%%%%%%%%%%%%%%%
\subsection{Universal Cover of a $2$-Weyl Graph with Property (SH)} \label{constructingunicov}
%%%%%%%%%%%%%%%%%%%%%%%%%%%%%%%%%%%%%%%%%%%%%%%%%%%%%%%%%%%%%%%%%%%%%%%%

In this section, we show that the universal cover of a connected $2$-Weyl graph $\cW$ with property (SH) can be constructed by representing the chambers as strict homotopy classes of geodesics issuing from a fixed chamber $C\in \cW$. By using a method similar to the proof of \cite[Proposition 4.8]{ronanlectures}, this construction will be used to prove $\text{(SH)} \implies \text{(C)}$ in the setting of $2$-Weyl graphs. We begin with two important consequences of property (SH).

\begin{figure}[t] 
	\centering
	\includegraphics[scale=1]{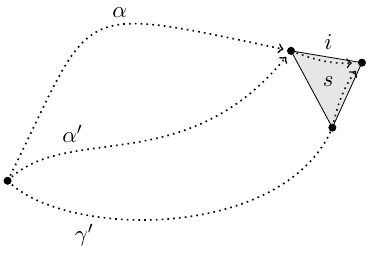}
	\caption{\autoref{endonsameedge}}
	\label{fig:SHLemma}
\end{figure} 

\begin{lem} \label{endonsameedge}
	Let $\cW$ be a pre-Weyl graph with property (SH), and let $\gamma$, $\gamma'$ be strictly homotopic geodesics in $\cW$ whose types end (begin) with the same letter $s\in S$. Then the final (initial) edges of $\gamma$ and $\gamma'$ must be equal.   
\end{lem}

\begin{proof}
	Let $\gamma_S=fs$, and $\gamma'_S=f's$. Let $i$ be the final edge of $\gamma$, and let $\ga$ be the subgallery such that $\gamma=\alpha i$. In particular, $\alpha_S=f$. Since $\gamma\simeq \gamma'$ as galleries, we have $fs\simeq f's$ as words, and so $f\simeq f'$ by \autoref{cor:cancalation}. Therefore there exists a gallery $\alpha'$ with $\alpha'_S=f'$ and $\ga\simeq \alpha'$. By transitivity, $\alpha' i$ is strictly homotopic to $\gamma'$, but they also have the same type. Thus, $\alpha' i=\gamma'$ by (SH), and so the final edge of $\gamma'$ is also $i$. The case of beginning with the same letter follows by a symmetric argument. 
\end{proof}

\begin{lem}  \label{subgalleriesaresh}
	Let $\cW$ be a pre-Weyl graph with property (SH), and let $\gamma$, $\gamma'$, $\gamma''$ be geodesics in $\cW$. If $\gamma' \gamma \simeq \gamma'' \gamma$ (or $\gamma \gamma' \simeq \gamma  \gamma''$), then $\gamma' \simeq \gamma''$.
\end{lem}

\begin{proof}
	Suppose that $\gamma' \gamma \simeq \gamma'' \gamma$. Since $\gamma'_S \gamma_S \simeq \gamma''_S \gamma_S$, we have $\gamma'_S\simeq \gamma''_S$ by \autoref{cor:cancalation}. Therefore there exists a gallery $\hat{\gamma}$ with $\hat{\gamma}_S=\gamma''_S$ and $\hat{\gamma}\simeq \gamma'$. Then $\hat{\gamma} \gamma \simeq \gamma'' \gamma$, and so $\hat{\gamma}=\gamma''$ by (SH). Thus $\gamma' \simeq \hat{\gamma}=  \gamma''$ as required. The case where $\gamma \gamma' \simeq \gamma  \gamma''$ follows by a symmetric argument.
\end{proof}

\begin{center}
	\emph{For the remainder of this section, $\cW$ is a connected $2$-Weyl graph with property (SH), $C\in \cW$ is a fixed chamber, and for geodesics $\gamma$ in $\cW$ we let $[\gamma]$ denote the strict homotopy class $[\gamma]_{\simeq}$.}
\end{center}

We now describe a certain construction $\wt{\cW}^C\to \cW$ of the universal cover of $\cW$. Let the chambers of $\wt{\cW}^C$ be the strict homotopy classes of the geodesics which issue from $C$. Thus
\[     \wt{\cW}_0^C=\big\{ [\gamma] : \gamma\ \text{is geodesic in}\ \cW\ \text{such that}\ \iota(\gamma)=C  \big  \} .   \]
We denote the class of the trivial geodesic at $C$ by $\wt{C}$. We let the edges be pairs $([\gamma],i)$, where $i$ is an edge of $\cW$ which issues from the chamber at which $[\gamma]$ terminates. Thus
\[    
\wt{\cW}_1^C=\big\{ ([\gamma],i) :   [\gamma]\in \wt{\cW}_0^C,\   \iota(i)=\tau(\gamma)           \big  \} 
.\]   
We now define the extremities of $\wt{\cW}^C$. Fix an edge $([\gamma],i)$. We always have
\[\iota([\gamma],i)=[\gamma]  .\]
\noindent Let $s=\upsilon(i)$ and $w=\gamma_W$. If $ws>w$ (in the Bruhat order), then $\gamma i$ is a geodesic, and we put
\[\tau([\gamma ],i)=[\gamma i].  \]

\begin{figure}[t] 
	\centering
	\includegraphics[scale=1]{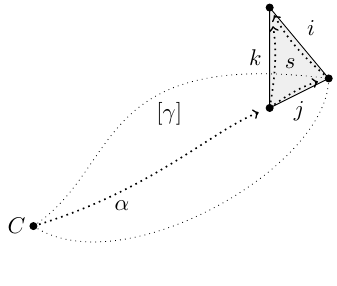}
	\caption{Defining $\tau([\gamma],i)$}
	\label{fig:SH}
\end{figure}

\noindent This is clearly well-defined. If $w s<w$, then $\gamma i$ is not a geodesic, however by the exchange condition, there exist geodesics in $[\gamma]$ whose types end with $s$. By \autoref{endonsameedge}, all these geodesics end with same $s$-labeled edge. Call this edge $j$, and let $k=j;i$. We may have $j=i^{-1}$, and $k=1$, in which case we let $k$ as a gallery denote the corresponding trivial gallery. Pick any geodesic $\alpha$ such that $\alpha j\in[\gamma]$ (see \autoref{fig:SH}). We put  
\[\tau([\gamma],i)=[\alpha k].  \]
\noindent To see that this is well-defined, suppose that $\alpha'$ is another geodesic used instead of $\alpha$. Then, by \autoref{subgalleriesaresh}, we have $\alpha\simeq \alpha'$, thus $[\alpha k]=[\alpha' k]$. Finally, we put $\upsilon ([\gamma],i)=\upsilon(i)$. This gives $\wt{\cW}^C$ the structure of a graph of type $M$.

We now show that $\wt{\cW}^C$ is naturally a chamber system. Let $s$, $w$, $j$, $k$, and $\ga$ be as before for a fixed edge $([\gamma],i)$. Firstly, notice that $\wt{\cW}^C$ does not have loops, for this would imply either:
\begin{itemize}
	\item
	$\gamma\simeq \gamma i$ in the case where $w s<w$, which is not possible since these galleries have different lengths
	\item
	$\gamma \simeq \alpha k$ in the case where $w s>w$, which also is not possible; if $k$ is trivial, then these galleries have different lengths, and if $k$ is not trivial, then $k=j$ by \autoref{endonsameedge}, which contradicts the fact that $k=j;i$. 
\end{itemize}
\noindent Secondly, edges of $\wt{\cW}^C$ with the same type and the same extremities are equal. To see this, suppose that $([\gamma],i)$ and $([\gamma],i')$ are edges which issue from $[\gamma]$ and terminate at the same chamber $D$, with $\upsilon(i)=\upsilon(i')$. If $ws>w$, then $[\gamma i]=[\gamma i']=D$. But by \autoref{endonsameedge}, we must have $i=i'$, and so $([\gamma],i)=([\gamma],i')$. If $ws<w$, then $[\ga k]=[\ga k']=D$, where $k'=j;i'$, and so:
\begin{itemize}
	\item
	if both $k$ and $k'$ are trivial, then $i=i'=j^{-1}$, and $([\gamma],i)=([\gamma],i')$
	\item
	if exactly one of $\{k,k'\}$ is trivial, then $\ga k$ cannot be strictly homotopic to $\ga k'$ since they have different lengths, a contradiction
	\item
	if neither are trivial, then $k=k'$ by \autoref{endonsameedge}, so again we must have $i=i'$, and so $([\gamma],i)=([\gamma],i')$.
\end{itemize}

We now claim that if $[\gamma] \xrightarrow{ ([\gamma],i)}  [\gamma']$ and $[\gamma'] \xrightarrow{ ([\gamma'],i')}  D$ are two edges with $\upsilon(i)=\upsilon(i')$ and $i \neq i'$, then $([\gamma],i;i')$ is an edge which terminates at $D$. If $ws>w$, then $[\gamma']=[\gamma i]$. Therefore $D=[\gamma i'']$, where $i''=i; i'$. The result follows. If $ws<w$, then $([\gamma],i;i')$ terminates at $[\ga i''']$, where $i'''=j;i;i'$. But $[\gamma']=[\ga i'']$, and so $D= [\ga i''']$. Therefore, for edges $i,i'\in \cW$ with $\upsilon(i)=\upsilon(i')$ and $i \neq i'$, we can define the composition,
\[([\gamma],i);([\gamma'],i')=([\gamma],i;i') .            \]
\noindent If $ws>w$, the inverse of an edge $([\gamma],i)$ is $([\gamma i],i^{-1})$. If $ws<w$, the inverse of an edge $([\gamma],i)$ is $([\ga k],i^{-1})$. This gives $\wt{\cW}^C$ the structure of a chamber system.

Define a map $p:\wt{\cW}^C\to \cW$ by putting $p([\gamma] )=\tau (\gamma )$ for chambers, and $p([\gamma],i)=i$ for edges. In particular we have $p(\wt{C})= C$. This map preserves extremities. For $\iota$, this follows from the fact that $\tau(\gamma)=\iota(i)$, and for $\tau$, this follows from the fact that $p([\gamma i])=\tau(i)$ if $ws>w$, and $p([\alpha k])=\tau(k)=\tau(i)$ if $ws<w$. Moreover, $p$ is an \'etale morphism of generalized chamber systems. This follows directly from the definition of composition in $\wt{\cW}^C$ and the definition of edges in $\wt{\cW}^C$. Then $p$ is surjective-\'etale since $\cW$ is connected. We now show that $p:\wt{\cW}^C\to \cW$ is also a pre-covering. First, we need the following observation.

\begin{lem} \label{lem:byconstruction}
	Let $\gb$ be a gallery in $\cW$ which issues from $C$, and let $\til{\gb}$ be the lifting of $\gb$ to a gallery which issues from $\wt{C}$. Then $\gb$ is homotopic to the geodesics in $\tau(\til{\gb})$.
\end{lem}

\begin{proof}
	We prove by induction on the length of $\gb$. If $\gb$ consists of one edge, then the result is trivial since $\tau(\til{\gb})$ is the set containing $\gb$. Suppose that the result holds for galleries of length $n$, and that $|\gb|=n+1$. Let $i$ be the last edge of $\gb$, and let $\gb'$ be the subgallery such that $\gb'i=\gb$. Let $\til{\gb}'$ be the lifting of $\gb'$ to a gallery which issues from $\wt{C}$, and let $\gamma\in \tau(\til{\gb}')$. Notice that $\gb'\sim \gamma$ by the induction hypothesis. Now, either $\tau(\til{\gb})=[\gamma i]$, or $\tau(\til{\gb})=[\ga k]$. In either case, $\gamma i$ is homotopic to the geodesics in $\tau(\til{\gb})$ since $\ga k\sim \gamma i$. Then $\gb=\gb'i\sim \gamma i$ by the induction hypothesis, and so $\gb$ is homotopic to the geodesics in $\tau(\til{\gb})$.
\end{proof}

\begin{lem}
	The surjective-\'etale morphism $p:\wt{\cW}^C\to \cW$ is a pre-covering.
\end{lem}

\begin{proof}
	Let $\gb$ be a gallery in $\wt{\cW}^C$ such that $p \circ \gb$ is a suite of $\cW$. Let $\rho$ and $\hat{\rho}$ be a maximal $(s,t)$-geodesic and a maximal $(t,s)$-geodesic respectively such that $\gb=\rho^{-1} \hat{\rho}$. We now prove that $\tau(\rho)=\tau(\hat{\rho})$, which implies that $\gb$ is a cycle. Let $[\gamma]=\iota(\rho)=\iota(\hat{\rho})$, and put $w=\gamma_W$. Let $\rho_{\cW}=p\circ \rho$, and  $\hat{\rho}_{\cW}=p\circ \hat{\rho}$. 
	
	Let $J=\{s,t\}$, and let $w'$ be the unique representative of the coset $wW_J$ with minimal word length. Then $w=w' w_J$, for some $w_J\in W_J$. We may assume that $\gamma$ is of the form $\gamma=\gamma'\gamma_J$, where $\gamma'_W=w'$ and ${\gamma_J}_W=w_J$. Let $\rho'$ and $\hat{\rho}'$ be geodesics which are homotopic to $\gamma_J\rho_{\cW}$ and $\gamma_J\hat{\rho}_{\cW}$ respectively. It follows that $\gamma' \rho'$ and $\gamma'\hat{\rho}'$ are geodesics. Since $p \circ \gb$ is a suite of $\cW$, we have $ \rho_{\cW} \sim \hat{\rho}_{\cW}$, and so, 
	\[\rho'\sim \gamma_J\rho_{\cW}\sim \gamma_J\hat{\rho}_{\cW}\sim \hat{\rho}'.     \]
	Then $\rho' \simeq \hat{\rho}'$ since $\cW$ has property ($2$C), and so $\gamma' \rho'\simeq\gamma'\hat{\rho}'$. The geodesics in $\tau(\rho)$ are homotopic to $\gamma \rho_{\cW}$, and therefore to $\gamma' \rho'$, by \autoref{lem:byconstruction}. Similarly, the geodesics in $\tau(\hat{\rho})$ are homotopic to $\gamma \hat{\rho}_{\cW}$, and therefore to $\gamma'\hat{\rho}'$. Thus, $\gamma' \rho' \in \tau(\rho)$ and $\gamma'\hat{\rho}'\in\tau(\hat{\rho})$, and so $\tau(\rho)= \tau(\hat{\rho})$ as required. 
\end{proof}

Redefine $\wt{\cW}^C$ to be its completion with respect to $p$. Thus, we obtain a covering $p:\wt{\cW}^C\to \cW$.

\begin{thm} \label{thm:stricthomosunicov}
	The covering $p:\wt{\cW}^C\to \cW$ is the universal cover of $\cW$.
\end{thm}

\begin{proof}
	First, notice that $p$ is a connected covering since for chambers $[\gamma],[\gamma'] \in \wt{\cW}^C$, the lifting of $\gamma$ and $\gamma'$ to galleries in $\wt{\cW}^C$ which issue from the class of the trivial gallery connect $[\gamma]$ and $[\gamma']$ respectively to the class of the trivial gallery.  
	
	Let $p':\wt{\cW}\to \cW$ be the universal cover of $\cW$. We now construct a morphism of coverings $\mu:p\to p'$, which proves that $\wt{\cW}^C$ is also simply connected. Pick a chamber $\wt{C}\in \wt{\cW}$ such that $p'(\wt{C})=C$. Let $\mu:\wt{\cW}^C\to \wt{\cW}$ be the morphism whose map on chambers is $[\gamma]\mapsto \tau(\tilde{\gamma})$, where $\tilde{\gamma}$ is the lifting of $\gamma$ to a gallery which issues from $\wt{C}$. To see that this is a morphism of chamber systems, let $[\gamma] \xrightarrow{ ([\gamma],i)}  [\gamma']$ be an edge of $\wt{\cW}^C$, and let $\tilde{i}$ be the unique edge of $\wt{\cW}$ such that $p'(\tilde{i})=i$ and $\iota(\tilde{i})=\tau({\tilde{\gamma}})$. Then $\tau(\tilde{i})=\tau({\tilde{\gamma}'})$, and so $[\gamma] \xrightarrow{ ([\gamma],i)}  [\gamma']$ is mapped to $\tilde{i}$. We have $p=p' \circ \mu$ since $p'(\tau(\tilde{\gamma}))=\tau(\gamma)$. Finally, to see that $\mu$ is a morphism of Weyl data, let $\theta$ be a suite of $\wt{\cW}^C$. Then $p\circ \theta$ must be a suite of $\wt{\cW}$ because $\wt{\cW}$ is simply connected. 
\end{proof}

%%%%%%%%%%%%%%%%%%%%%%%%%%%%%%%%%%%%%%%%%%%%%%%%%%%%%%%%%%%%%%%%%%%%%%%%
\subsection{Weyl Graphs and Buildings}
%%%%%%%%%%%%%%%%%%%%%%%%%%%%%%%%%%%%%%%%%%%%%%%%%%%%%%%%%%%%%%%%%%%%%%%%

We now define Weyl graphs and collect some of their basic properties. We show that the notion of a building is equivalent to that of the universal cover of a connected Weyl graph. We also prove the equivalence of the Weyl properties for $2$-Weyl graphs, which will give us several characterizations of Weyl graphs.  

Let $M$ be a Coxeter matrix. Let a \emph{Weyl graph} of type $M$ be a pre-Weyl graph of type $M$ which additionally satisfies the following Weyl property,
\begin{enumerate}  [itemindent=0cm, leftmargin=1.9cm]
	\item [\textbf{(W)}]
	homotopic geodesics have the same $W$-length.
\end{enumerate}
A Weyl graph $\cW$ is called a \emph{building} if $\cW$ is connected and simply connected. For all $s,t\in S$ with $s\neq t$ and $m_{st}<\infty$, every $(s,t)$-cycle in a building is a suite. Therefore the Weyl data of a building is just the simple Weyl data induced by its chamber system. Notice that $2$-Weyl graphs are exactly pre-Weyl graphs whose $2$-residues are Weyl graphs. Notice also that the residues of Weyl graphs are again Weyl graphs, thus Weyl graphs are $2$-Weyl. See \autoref{rem:weylisbuilding} for a demonstration that our chosen definition of a building is equivalent to a classical definition which appears in \cite{ronanlectures}. We will denote Weyl graphs which are buildings by $\Delta$. 

Property (W) allows us to define the \emph{$W$-length} ${[\gamma]}_W$ of a homotopy class of galleries $[\gamma]$ in a Weyl graph to be the $W$-length of the geodesics which it contains. Thus, ${[\gamma]}_W=\gamma_W$. The function
\[    
\overline{\cW}_1\to W
,\qquad 
[\gamma]\mapsto  {[\gamma]}_W 
\]
is called the \emph{metrization} of the fundamental groupoid $\overline{\cW}$ of $\cW$.   

\begin{prop} \label{prop:weylisbuilding}
	A connected pre-Weyl graph $\cW$ is a building if and only if $\cW$ satisfies property (B). 
\end{prop}

\begin{proof}
	If a Weyl graph is simply connected, then all the geodesics with the same extremities will be homotopic, and therefore have the same $W$-length. Conversely, suppose that a connected pre-Weyl graph $\cW$ satisfies property (B). Let $\gamma$ be a geodesic in $\cW$ with $\iota(\gamma)=\tau(\gamma)=C$. 
	Then (B) implies that $\gamma_W=1$, since the trivial gallery at $C$ is a geodesic. Thus, $\gamma$ is trivial, and so there is only one homotopy class of loops at $C$. Therefore $\cW$ is simply connected, and so $\cW$ is a building.  
\end{proof}

\begin{remark} \label{rem:weylisbuilding}
	In \cite{ronanlectures}, a building of type $M$ is defined to be a weak chamber system $\cW$ of type $M$ which admits a function $\cW_0\times \cW_0\to W$ such that if $f$ is a reduced word over $S$, then $(C,D)\mapsto w(f)$ if and only if there exists a geodesic $\gamma$ which travels from $C$ to $D$ with $\gamma_S=f$. Clearly, the underlying chamber system of a connected and simply connected Weyl graph admits such a function. Conversely, given a building $\cW$ in the sense of \cite{ronanlectures}, the simple Weyl data associated to the chamber system $\cW$ will be a connected pre-Weyl graph with property (B), and so is a connected and simply connected Weyl graph by \autoref{prop:weylisbuilding}.
\end{remark}

We now show that the image of a Weyl graph under a covering is a Weyl graph, and that a covering of a Weyl graph is again a Weyl graph.

	\begin{lem}
		Let $\go:\cW'\to \cW$ be a morphism of Weyl data. If $\cW$ has property (W), then $\cW'$ has property (W).
	\end{lem}
	
	\begin{proof}
		Let $\til{\gamma}$ and $\til{\gamma}'$ be two homotopic geodesics in $\cW'$. Let $\gamma=  \go  \circ   \til{\gamma}$ and $\gamma'=  \go  \circ   \til{\gamma}'$. Then $\gamma$ and $\gamma'$ are homotopic geodesics in $\cW$, and so $\til{\gamma}_W=\gamma_W=\gamma'_W=\til{\gamma}'_W$ since $\cW$ has property (W).  
	\end{proof}
	
	\begin{lem}
		Let $p:\wt{\cW}\to \cW$ be an covering of Weyl data. If $\wt{\cW}$ has property (W), then $\cW$ has property (W).
	\end{lem}
	
	\begin{proof}
		Let $\gamma$ and $\gamma'$ be two homotopic geodesics in $\cW$. Lift these geodesics to homotopic geodesics $\til{\gamma}$ and $\til{\gamma}'$ in $\wt{\cW}$. Then $\gamma_W=\til{\gamma}_W=\til{\gamma}'_W=\gamma'_W$ since $\wt{\cW}$ has property (W). 
	\end{proof} 
	
\begin{thm} \label{prop:presreflectW}
	Let $p:\wt{\cW}\to \cW$ be a covering of Weyl data. Then $\wt{\cW}$ is a Weyl graph if and only if $\cW$ is a Weyl graph.
\end{thm}

\begin{proof}
	Recall from \autoref{prop:presandreflect} that coverings preserve and reflect the properties (PW$0$) and (PW$1$). For (W), we have the following two lemmas, whose proofs are copies of those from \autoref{prop:presandreflect}, but without the assumption that geodesics are alternating. The result then follows from the previous lemmas.
\end{proof}

\begin{cor}
	The universal cover of a connected Weyl graph is a building, and the image of a building under a covering is a connected Weyl graph.
\end{cor}

Therefore connected Weyl graphs are exactly the quotients of buildings by chamber-free group actions, and buildings are exactly the universal covers of connected Weyl graphs.

\begin{thm} \label{thm:eqivofweylprop}
	Let $\cW$ be a $2$-Weyl graph. If $\cW$ has any of the Weyl properties, then $\cW$ is a Weyl graph. Conversely, if $\cW$ is a Weyl graph, then $\cW$ satisfies all the Weyl properties. 
\end{thm}

\begin{proof}
The implications which we will prove are
\[
\text{(C)} \implies \text{(W)} \implies \text{(H)}  \implies \text{(SH)} \implies \text{(C)}
.\] 
The implications (H)$\implies$(SH) and (C)$\implies$ (W) are clear. The proof of (W)$\implies$ (H) is essentially the same as the proof of \autoref{prop:WimpliesH}, which was the rank $2$ case. (Let $\rho$ and $\rho'$ be homotopic geodesics in $\cW$ with the same type. Let $i$ be the last edge of $\rho$, and let $i'$ be the last edge of $\rho'$. Towards a contradiction, suppose that $i\neq i'$. Let $j=i'i^{-1}$. Let $\ga$ and $\ga'$ be the subgalleries such that $\ga i=\rho$ and $\ga 'i'=\rho'$. Then $\ga$ and $\ga'j$ are homotopic geodesics with different $W$-lengths, a contradiction. Thus, $i=i'$. Then $\ga$ and $\ga'$ are also homotopic geodesics, and so can apply the same argument to the penultimate edges of $\rho$ and $\rho'$. Therefore, by induction, we may conclude that $\rho=\rho'$.)
Finally, we prove (SH)$\implies$ (C) using the construction of \autoref{constructingunicov}. Let $\gamma$ and $\gamma'$ be homotopic geodesics in $\cW$. Put $C=\iota(\gamma)=\iota(\gamma')$, and let $p:\wt{\cW}_C\to \cW$ be the covering constructed in \autoref{constructingunicov}. The map $\mu$ constructed in the proof of \autoref{thm:stricthomosunicov} is bijective on chambers by \autoref{cor:isoofunicovers}. Since $\gamma\sim \gamma'$, we have $\mu([\gamma])=\mu([\gamma'])$, thus $[\gamma]=[\gamma']$, and so $\gamma\simeq \gamma'$.  
\end{proof}

\begin{cor} \label{prop:embedgroupoid}
	Let $\cW$ be a Weyl graph which is labeled over $S$. For each $J\subseteq S$, the groupoid homomorphism, 
	\[\overline{\varepsilon}_J: \overline{\cW}_J  \to   \overline{\cW}   \] 
	is an embedding. Thus, $\overline{\cW}_J$ is naturally a subgroupoid of $\overline{\cW}$.
\end{cor}

\begin{proof}
	Let $\gamma$ and $\gamma'$ be homotopic geodesics of $\cW$ which are contained in a $J$-residue $R$. Then $\gamma$ and $\gamma'$ are strictly homotopic, and so a homotopy from $\gamma$ to $\gamma'$ can take place within $R$. Thus, $\gamma$ and $\gamma'$ are also homotopic in $\cW_J$.  
\end{proof}

%%%%%%%%%%%%%%%%%%%%%%%%%%%%%%%%%%%%%%%%%%%%%%%%%%%%%%%%%%%%%%%%%%%%%%%%
\subsection{Tits's Local-to-Global Results}
%%%%%%%%%%%%%%%%%%%%%%%%%%%%%%%%%%%%%%%%%%%%%%%%%%%%%%%%%%%%%%%%%%%%%%%%

We now show that $2$-Weyl graphs are often Weyl graphs. In particular, the only thing which stops a $2$-Weyl graph from being a Weyl graph are the spherical $3$-residues of type $C_3$ and $H_3$. Given the preservation and reflection of property (W) under covering maps, and the fact that buildings are exactly the universal covers of connected Weyl graphs, this is an easy consequence of results in \cite{tits81local}. However, we provide a more direct proof of the first part of Tits's result, along the lines of \cite[Theorem 4.9]{ronanlectures}.  

\begin{lem} \label{lem1}  \index{Weyl graph!local to global}
	Let $\cW$ be a $2$-Weyl graph. Then $\cW$ is a Weyl graph if (and only if) the spherical $3$-residues of $\cW$ are Weyl graphs. 
\end{lem}

\begin{proof}
	We claim that $\cW$ has property (SH), which suffices by \autoref{thm:eqivofweylprop}. To see this, let $\gamma$ and $\gamma'$ be strictly homotopic geodesics of $\cW$ with the same type. A strict homotopy between $\gamma$ and $\gamma'$ induces a strict homotopy of words, which, by \cite[Theorem 2.17]{ronanlectures}, decomposes into self strict homotopies which are either inessential, or else only alter a subword over $J$, where $J$ is a $3$-element spherical subset of $S$. A strict homotopy of galleries which induces an inessential self strict homotopy of words cannot change a gallery. It then follows that $\gamma=\gamma'$ by the hypothesis on $\cW$. Thus, $\cW$ has property (SH), and so is a Weyl graph. 
\end{proof}

\begin{lem} \label{lem2} 
	If a $2$-Weyl graph $\cW$ of type $M$ and of rank $3$ is not a Weyl graph, then either $M = C_3$ or $M=H_3$.
\end{lem}

\begin{proof}
	Since $\cW$ is not a Weyl graph, it must have a connected component which is not a Weyl graph. Let $\Delta$ be the universal cover of this connected component. By \autoref{prop:presreflectW}, $\Delta$ is a connected simply connected chamber system of type $M$ which is not a building. Therefore $M = C_3$ or $M=H_3$ by the discussion in \cite[Section 2.2]{tits81local} and \cite[Theorem 1]{tits81local}.
\end{proof}

Then, combining \autoref{lem1} and \autoref{lem2} gives the following.

\begin{thm}  \label{thm:noh3orc3}
	Let $\cW$ be a $2$-Weyl graph. Then $\cW$ is a Weyl graph if (and only if) the residues of $\cW$ of type $C_3$ and $H_3$ are Weyl graphs.
\end{thm}

%%%%%%%%%%%%%%%%%%%%%%%%%%%%%%%%%%%%%%%%%%%%%%%%%%%%%%%%%%%%%%%%%%%%%%%%
\subsection{Weyl Graphs and Strict $W$-Groupoids} \label{sec:wgroup}
%%%%%%%%%%%%%%%%%%%%%%%%%%%%%%%%%%%%%%%%%%%%%%%%%%%%%%%%%%%%%%%%%%%%%%%%

Throughout this section, $\cW$ denotes a Weyl graph of type $M$, $\overline{\cW}$ denotes the fundamental groupoid of $\cW$, and $\gd:\overline{\cW}\to W$ denotes the metrization $[\gamma]\mapsto \gamma_W$. 

Let us recall the definition of $W$-groupoids from \cite{nor1}, which are the full stacky generalization of buildings. Given a groupoid $\cG$, a \emph{$W$-groupoid} $\delta:\cG\to W$ on $\cG$ is a function on the edges of $\cG$ into $W$ such that
\begin{enumerate}[itemindent=0cm, leftmargin=1.9cm]
\item [\textbf{(weak)}]
for all chambers (vertices) $C\in \cG$ and all $s\in S$, there exists a edge $g\in \cG$ with $\iota(g)=C$ and $\gd(g)=s$
\item [\textbf{(WG1)}]
for all identity edges $1\in \cG$, we have  $\gd(1)=1$ 
\item [\textbf{(WG2)}]
for all edges $g,h\in \cG$ such that $gh$ is defined in $\cG$, we have (in the Bruhat order)
\[
\delta(gh)\geq \delta(g)\delta(h)
\]
\item [\textbf{(WG3)}]
for all edges $g\in \cG$ and for each $s\in S$ such that $\gd(g)s<\gd(g)$, there exists an edge $h\in \cG$ with $\iota(h)=\iota(g)$ such that
\[
\gd(h^{-1}g)=s\qquad \text{and}\qquad \gd(h)=\gd(g)s
.\]
\end{enumerate} 
If in addition $\gd(g)=1$ implies that $g$ is a trivial edge, then $\delta:\cG\to W$ is called a \hbox{\emph{strict $W$-groupoid}}. We have the following analogies with classical metric spaces in the strict case and when $\cG$ is a setoid (i.e. buildings); (WG1) is identity of indiscernibles, (WG2) is triangle inequality, and (WG3) is `geodesic metric space'. We have the following local version of (WG2),
\begin{enumerate}  [itemindent=0cm, leftmargin=1.9cm] \index{$W$-groupoid!properties!local triangle inequality}
\item [\textbf{(WG2$'$)}]
For all edges $g,h\in \cG$ such that $g^{-1}h$ is defined in $\cG$ and $\gd(g^{-1}h)\in S$, putting $w=\gd(g)$ and $s=\gd(g^{-1}h)$, we have
\begin{enumerate}[label=(\roman*)]
\item
if $ws<w$, then $\gd(h)\in \{w, ws\}$
\item
if $ws>w$, then $\gd(h)=ws$.
\end{enumerate}
\end{enumerate} 
We now show that the metrization of the fundamental groupoid of a Weyl graph is a strict $W$-groupoid. 

	\begin{lem} 
	The metrization $\gd:\overline{\cW}\to W$ satisfies property (WG1). Moreover, for all classes $g\in \overline{\cW}$, if $\gd(g)=1$, then $g$ is trivial.  
\end{lem}

\begin{proof}
	The metrization has property (WG1) because every trivial class $1\in \overline{\cW}$ contains a geodesic which is a trivial gallery, and classes have the $W$-length of geodesic(s) they contain. For $g\in \overline{\cW}$, if $\gd(g)=1$ then $g$ contains a geodesic which is a trivial gallery, and so $g$ must be trivial.
\end{proof}

\begin{lem} \label{lemma:localtriinequal}
	The metrization $\gd:\overline{\cW}\to W$ satisfies property (WG2$'$).  
\end{lem}

\begin{proof}
	Let $g=[\gamma]\in \overline{\cW}$ be a class with $\gd(g)=w$, and let $h=[\gamma']\in \overline{\cW}$ be a class with $\gd(g^{-1} h )=s$. Since $g^{-1}h$ has $W$-length a generator, $g^{-1}h$ will contain a single geodesic which is just an edge $i$ with $\upsilon(i)=s$. We have $h=g (g^{-1}h)=[\gamma][i]=[\gamma i]$. If $ws>w$, then $\gamma i$ is a geodesic with $W$-length $ws$, and so $\gd(h)=ws$. If $ws<w$, then by the exchange condition, there exists a geodesic $\hat{\gamma} \in g$ which ends with an edge $j$ of type $s$. Let $\gamma''$ be the subgallery such that $\hat{\gamma}=\gamma'' j$. Let $k=j;i$. Then $\gamma'' k$ is a geodesic with $\gamma'' k\sim \gamma'$ (if $i=j^{-1}$, then here $k$ denotes the corresponding trivial gallery, and so $\gamma'' k=\gamma''$). Therefore the $W$-length of $h$ is equal to the $W$-length of $\gamma'' k$. If $i=j^{-1}$, then $\gd(h)=ws$, and otherwise $\gd(h)=w$. Thus $\gd(h)\in \{w, ws \}$, as required.  
\end{proof}  

\begin{lem} 
	The metrization $\gd:\overline{\cW}\to W$ satisfies property (WG3).  
\end{lem}

\begin{proof}
	Let $g\in \overline{\cW}$ be a class, and let $s\in S$ such that $\gd(g)s<\gd(g)$. Then by the exchange condition, there exists a geodesic $\gamma\in g$ which ends with an edge $i$ of type $s$. Let $\gamma'$ be the subgallery such that $\gamma=\gamma' i$. Let $h=[\gamma']$. Notice that $g=h;i$. Then $\gd(h^{-1}g)=\gd(i)=s$, and so $\gd(h)=\gd(g)s$ as required. 
\end{proof}

\begin{thm}
Let $\cW$ be a Weyl graph of type $M$. Then the metrization, 
\[\gd:\overline{\cW}\to W
,\qquad
 [\gamma]\mapsto \gamma_W\] 
is a strict $W$-groupoid of type $M$. 
\end{thm}

\begin{proof}
The metrization $\gd:\overline{\cW}\to W$ is weak by property (PW0) of Weyl graphs. The result then follows by the previous lemmas and \cite[Proposition 4.3]{nor1}.
\end{proof}

We now show that strict $W$-groupoids are equivalent to Weyl graphs. Let $\gd:\cG\to W$ be a strict $W$-groupoid of type $M$. For $s\in S$, let $\cG_s$ denote the subgroupoid of $\cG$ whose set of edges is 
\[(\cG_s)_1=\big \{ g\in \cG_1 : \gd(g)\in \{ 1,s \}  \big   \}.\] 
This defines a subgroupoid of $\cG$. We call $\cG_s$ the \emph{parabolic groupoid} of type $s$. Then $(\cG_s)_{s\in S}$ is a generalized chamber system of type $M$. We give $(\cG_s)_{s\in S}$ the structure of Weyl data, denoted $\cW(\cG)$, as follows; for each pair $(s,t) \in S\times S$ with $s \neq t$ and $m_{st}<\infty$, let an $(s,t)$-cycle $\theta$ of $(\cG_s)_{s\in S}$ be a defining suite of $\cW(\cG)$ if the sequence of edges of $\theta$ is a decomposition of a trivial class in $\cG$. We call $\cW(\cG)$ the Weyl data \emph{associated} to $\gd:\cG\to W$. 

\begin{lem}
	Let $\gd:\cG\to W$ be a strict $W$-groupoid, and let $\cW(\cG)$ be the Weyl data associated to $\gd:\cG\to W$. Then $\cW(\cG)$ is a Weyl graph.
\end{lem}

\begin{proof}
	The Weyl data $\cW(\cG)$ satisfies property (PW0) because $\gd:\cG\to W$ satisfies property (WG0). For property (PW1), let $\rho$ be a maximal $(s,t)$-geodesic of $\cW(\cG)$. Let $g\in \cG$ be the class for which the sequence of edges of $\rho$ is reduced decomposition. Then $g$ has a reduced decomposition $g=g_1\dots g_{m_{st}}$ of type $p(t,s)$. Let $\hat{\rho}$ be the gallery in $\cW(\cG)$ whose sequence of edges is $g_1,\dots, g_{m_{st}}$. Then $\rho \hat{\rho}^{-1}$ is a defining suite of $\cW(\cG)$, and so $\rho\sim \hat{\rho}$. For property (W), let $\gamma$ and $\gamma'$ be (strictly) homotopic geodesics of $\cW(\cG)$. Then the sequences of edges of $\gamma$ and $\gamma'$ are reduced decompositions of the same class $g\in \cW$. Therefore we have $\gamma_W=\gd(g)=\gamma'_W$. 
\end{proof}

The following shows that the association $\gd:\cG\to W \mapsto \cW(\cG)$ is the inverse of the metrization of the fundamental groupoid of a Weyl graph.

\begin{prop}
	Let $\cW$ be a Weyl graph. Then the Weyl data $\cW(\overline{\cW})$ associated to the metrization $\gd:\overline{\cW}\to W$ of fundamental groupoid of $\cW$ is naturally isomorphic to $\cW$. 
\end{prop}

\begin{proof}
	The parabolic groupoids of $\overline{\cW}$ are just the panel groupoids of $\cW$, and so we have an isomorphism of generalized chamber systems $\go:\cW(\overline{\cW})\to \cW$ such that $\go_s:\cW_s\to \cW_s$ is the identity for all $s\in S$. Let $\theta$ be a defining suite of $\cW(\overline{\cW})$. Then it follows directly from the definitions that $\go\circ \theta$ is a suite of $\cW$. Let $\go^{-1}$ be the inverse of $\go$, and let $\theta$ be a defining suite of $\cW$. Again, it follows directly from the definitions that $\go^{-1}\circ \theta$ is a suite of $\cW(\overline{\cW})$. Therefore $\go$ has an inverse as a morphism of Weyl data, and so is an isomorphism of Weyl data. 
\end{proof}

Finally, we show that the metrization of the fundamental groupoid of a Weyl graph is the inverse of the association $\gd:\cG\to W \mapsto \cW(\cG)$.

\begin{prop}
	Let $\gd:\cG\to W$ be a strict $W$-groupoid, and let $\gd':\overline{\cW}(\cG)\to W$ be the metrization of the fundamental groupoid of $\cW(\cG)$. Then there exists an isomorphism $\psi:\cG\to\overline{\cW}(\cG)$ such that $\gd=\gd' \circ \psi$.
\end{prop}

\begin{proof}
	Let $\psi:\overline{\cW}(\cG)\to  \cG $ be the homomorphism which sends the class of a gallery $[\gb]\in \overline{\cW}(\cG)$ to the class of $\cG$ for which the sequence of edges of $\gb$ is a decomposition. This is well-defined because $1$-elementary homotopies respect the composition of classes in the parabolic groupoids of $\cG$, and the suites of $\overline{\cW}(\cG)$ consist of sequences of edges whose composition in $\cG$ is trivial. To see that $\psi$ is a homomorphism, recall that $[\gb][\gb']=[\gb\gb']$, and if the sequence of edges of $\gb$ is a decomposition for $g$, and the sequence of edges of $\gb'$ is a decomposition for $g'$, then the sequence of edges of $\gb\gb'$ is clearly a decomposition for $gg'$. To see that $\psi$ is an isomorphism, consider its inverse, which sends a class $g\in \cG$ to the class of a geodesic whose sequence of edges is a reduced decomposition of $g$. This is well-defined because there will exist strict homotopies between geodesics obtained from reduced decompositions of $g$. Therefore $\psi$ is bijective, and so $\psi$ is an isomorphism. Finally, the fact that $\gd=\gd' \circ \psi$ follows from the properties of $W$-groupoids.
\end{proof}

%%%%%%%%%%%%%%%%%%%%%%%%%%%%%%%%%%%%%%%%%%%%%%%%%%%%%%%%%%%%%%%%%%%%%%%%
%%%%%%%%%%%%%%%%%%%%%%%%%%%%%%%%%%%%%%%%%%%%%%%%%%%%%%%%%%%%%%%%%%%%%%%%
\section{The Fundamental Group of a Weyl Graph} \label{sec:groupres}
%%%%%%%%%%%%%%%%%%%%%%%%%%%%%%%%%%%%%%%%%%%%%%%%%%%%%%%%%%%%%%%%%%%%%%%%
%%%%%%%%%%%%%%%%%%%%%%%%%%%%%%%%%%%%%%%%%%%%%%%%%%%%%%%%%%%%%%%%%%%%%%%%

We now describe a method for obtaining a group presentation of the fundamental group of a connected Weyl graph. 

%%%%%%%%%%%%%%%%%%%%%%%%%%%%%%%%%%%%%%%%%%%%%%%%%%%%%%%%%%%%%%%%%%%%%%%%
\subsection{Generating Sets}
%%%%%%%%%%%%%%%%%%%%%%%%%%%%%%%%%%%%%%%%%%%%%%%%%%%%%%%%%%%%%%%%%%%%%%%%

Let $M$ be a Coxeter matrix on $S$, and let $\cW$ be a connected Weyl graph of type $M$. For each $s\in S$, pick a base chamber in each connected component of $\cW_s$. Let $B_s=\{B_{s,0},B_{s,1},\dots \}$ be the set of base chambers in $\cW_s$, and let $\cB_s$ be the set of non-trivial loops in $\cW_s$ at the base chambers, thus
\[
\cB_s:=  \big\{ i\in \cW_1: \upsilon(i)=s,\ \iota(i)=\tau(i)\in B_s \big\}
.\]
For each $s\in S$ and for each chamber $C\in \cW_s$ with $C\notin B_s$, pick an edge $i_C\in \cW_s$ with $\iota(i_C)\in B_s$ and $\tau(i_C)=C$. Let
\[
\cI^{+}_s:=\{ i_C : C\in \cW_0\setminus B_s   \}, \qquad \cI^{-}_s:=\{ i_C^{-1} : C\in \cW_0\setminus B_s   \},\qquad  \cI_s:= \cI^{+}_s \sqcup \cI^{-}_s
.\] 
\noindent Also put $\cS_s=\cB_s \sqcup  \cI_s$, and
\[        
\cB:=\bigsqcup_{s\in S} \cB_s ,\qquad              \cI:=\bigsqcup_{s\in S} \cI_s ,\qquad  \cI^{+}:=\bigsqcup_{s\in S} \cI^{+}_s ,\qquad \cI^{-}:=\bigsqcup_{s\in S} \cI^{-}_s ,\qquad      \cS:=  \bigsqcup_{s\in S} \cS_s= \cB\sqcup \cI
.\]
We call $\cS$ a \emph{generating set} of $\cW$. Notice that $\cS_s$ generates $\cW_s$, and $\cS$ generates $\overline{\cW}$ (as groupoids). For each edge $i\in \cW$, putting $s=\upsilon(i)$, then there exists a unique $3$-tuple $(i^-, i_{\cB} ,i^+)$ of edges of $\cW_s$ such that the following hold,
\begin{enumerate}[label=(\roman*)]
	\item
	$i=i^- ; i_{\cB} ; i^+$
	\item
	either $i^- \in  \cI_s^{-}$ or $i^-$ is trivial
	\item
	either $i_{\cB} \in  \cB_s$ or $i_{\cB}$ is trivial
	\item
	either $i^+ \in  \cI_s^{+}$ or $i^+$ is trivial. 
\end{enumerate}
The \emph{expression} of $i$ with respect to $\cS$ is the word obtained from the sequence $i^-, i_{\cB} ,i^+$ by deleting any trivial edges.

%%%%%%%%%%%%%%%%%%%%%%%%%%%%%%%%%%%%%%%%%%%%%%%%%%%%%%%%%%%%%%%%%%%%%%%%
\subsection{The Universal Group}
%%%%%%%%%%%%%%%%%%%%%%%%%%%%%%%%%%%%%%%%%%%%%%%%%%%%%%%%%%%%%%%%%%%%%%%%

Let $\cW$ be a connected Weyl graph and let $\cS$ be a generating set of $\cW$. Let $r=  i_1,\dots,i_n     $ be a sequence of edges of $\cW$. Let the \emph{$\cS$-sequence} of $r$ be the sequence obtained from $r$ by replacing each edge $i_k$, $1 \leq k\leq n$, with the expression of $i_k$ with respect to $\cS$. Notice that if $r$ is the sequence of edges of a gallery $\gb$, then the $\cS$-sequence of $r$ is the sequence of edges of a gallery which can be obtained from $\gb$ by a composition of expansions (defined in \autoref{sec:homoofgal}). Let $\cR$ denote the set of sequences of edges of $\cW$ which are obtained as the $\cS$-sequences of the defining suites of $\cW$. The elements of $\cR$ are called (defining) \emph{$\cS$-suites}, and we think of them as words over $\cS$.

The \emph{universal group} $\FG(\cW)$ of $\cW$ with respect to $\cS$ is the group generated by $\cS$, subject to the relations of the local groups at each base chamber, and $\cR$ (treated as a set of relators). Explicitly, let $\cR_{s,k}$ be a set of defining relations for the local group of $\cW_s$ at $B_{s,k}\in B_s$, and let $\cR_{\cB}=\bigsqcup_{s,k} \cR_{s,k}$. Then $\FG(\cW)$ is the marked group given by
\[       
\FG(\cW):=\big \la  \cS\  |\      \cR,\ \cR_{\cB},\ ij=1 : i\in \cI^+,\ j=i^{-1}\in \cI^-    \big  \ra    
.\] 
These are the relations required to make the natural projection 
\[
\pi:\overline{\cW}\to \FG(\cW)
,\qquad  
i\mapsto i \qquad (\text{for} \ i\in \cS)
\] 
a well-defined homomorphism. To see that $\pi$ is well-defined, let $[\gb]\in \overline{\cW}$. Let $r=i_1,\dots,i_n$ be the sequence of edges of $\gb$. Then $\pi([\gb])$ is the product in $\FG(\cW)$ of the $\cS$-sequence of $r$. Let $\hat{\gb}$ be a gallery obtained from $\gb$ by a $1$-elementary homotopy. Then $\pi([\hat{\gb}])=\pi([\gb])$ by the inclusion of the relations $\cR_{\cB}$, and the relators of the form $ij=1$. Let $\hat{\gb}$ be a gallery obtained from $\gb$ by a $2$-elementary homotopy. Then $\pi([\hat{\gb}])=\pi([\gb])$ by the inclusion of the relators $\cR$. 

%%%%%%%%%%%%%%%%%%%%%%%%%%%%%%%%%%%%%%%%%%%%%%%%%%%%%%%%%%%%%%%%%%%%%%%%
\subsection{The Fundamental Group}
%%%%%%%%%%%%%%%%%%%%%%%%%%%%%%%%%%%%%%%%%%%%%%%%%%%%%%%%%%%%%%%%%%%%%%%%

Continue to let $\cW$ be a connected Weyl graph, and let $\cS$ be a generating set of $\cW$. Let $\gG$ be the following undirected graph,
\begin{enumerate}[label=(\roman*)]
	\item
	the vertices of $\gG$ are the chambers of $\cW$
	\item
	the edges of $\gG$ are sets of the form $\{i, i^{-1}\}$, where $i\in \cI^+$
	\item
	the extremities of the undirected edge $\{i, i^{-1}\}$ are $\iota(i)$ and $\tau(i)$.
\end{enumerate}
Notice that $\gG$ is connected. Let $\gG'$ be a spanning tree of $\gG$, and let $T\subseteq \cI$ be the set of edges which are contained in some edge of $\gG'$. We call $T$ a \emph{spanning tree} of $\cW$. Notice that if a panel groupoid $\cW_s$ is connected, then we can take $T$ to be $\cI_s$. The \emph{fundamental group} $\pi_1(\cW,T)$ of $\cW$ at $T$ is the marked group given by
\[       
\pi_1(\cW,T):=\big \la  \cS\  |\      \cR,\ \cR_{\cB},\ ij=1,\ t=1 : i\in \cI^+,\ j=i^{-1}\in \cI^-,\ t\in T   \big  \ra    
.\] 

\begin{thm} \label{thm:pres} 
	Let $\cW$ be a connected Weyl graph. Let $\cS$ be a generating set of $\cW$ and let $T$ be a spanning tree of $\cW$. Then the local groups of $\overline{\cW}$ are naturally isomorphic to $\pi_1(\cW,T)$.\footnote{ regarding the fact that these isomorphisms are natural, remember that a choice of tree $T$ has been made} 
\end{thm}
\begin{proof}
	Pick a chamber $C\in \cW$, and let $\overline{\cW}_C$ be the local group of $\overline{\cW}$ at $C$. Let
	\[\varphi: \overline{\cW}_C\to \pi_1(\cW,T)\] 
	be the restriction of $\pi:\overline{\cW}\to \FG(\cW)$ to $\overline{\cW}_C$, composed with the quotient map \hbox{$\FG(\cW)\to\pi_1(\cW,T)$}. For each chamber $D\in \cW$, let $g_D\in  \overline{\cW}(C,D)$ be the homotopy class of the gallery corresponding to the unique path in $T$ from $C$ to $D$. Let 
	\[\psi: \pi_1(\cW,T) \to \overline{\cW}_C\] 
	\noindent be the homomorphism mapping a generator $i\in \cS$ to the composition $g_{\iota(i)}; i; g_{\tau(i)}^{-1}\in \overline{\cW}$. This is a well-defined homomorphism because the relations of $\pi_1(\cW,T)$ are satisfied in $\overline{\cW}_C$, since $\overline{\cW}_C$ is a subgroup of $\overline{\cW}$. Notice that $\gf \circ \psi$ is the identity on $\cS$ because $\pi(g_D)$ lies in the kernel of $\FG(\cW)\to\pi_1(\cW,T)$. Also, $\psi \circ \gf$ is the identity because the $g_D$ cancel via contractions, and one recovers the original gallery up to homotopy. Thus, $\varphi$ and $\psi$ are mutually inverse isomorphisms.  
\end{proof}

Therefore we have a natural isomorphism $\overline{\cW}_C \xrightarrow{\sim}\pi_1(\cW,T)$. If we make the choice of a chamber in the universal cover $\Delta$ of $\cW$, we get a well-defined action of $\pi_1(\cW,T)$ on $\Delta$. A method one can employ when calculating the fundamental group of a Weyl graph is to first calculate the universal groups of its $2$-residues, take the union of these presentations, and then quotient out by a spanning tree.   

%%%%%%%%%%%%%%%%%%%%%%%%%%%%%%%%%%%%%%%%%%%%%%%%%%%%%%%%%%%%%%%%%%%%%%%%
\subsection{Petals and Flowers}
%%%%%%%%%%%%%%%%%%%%%%%%%%%%%%%%%%%%%%%%%%%%%%%%%%%%%%%%%%%%%%%%%%%%%%%%

Let us call a rank $2$ (i.e. $S=\{s,t\}$) Weyl graph a \emph{Weyl polygon}, or \emph{Weyl $m_{st}$-gon}. We now show that in order to determine a Weyl polygon, one only has to know the underlying generalized chamber system, and a set of homotopic maximal alternating geodesic pairs, which we call `petals', issuing from a fixed chamber. This will simplify the task of calculating fundamental groups of Weyl graphs in many cases (e.g. see \cite{nor3}).  

Let $C$ be a chamber of a Weyl graph $\cW$. Let $J=\{s,t\}$ be a $2$-element spherical subset of $S$ (so $m_{st}< \infty$). Let the \emph{$J$-flower} $\cF_J(C)$ based at $C$ be the set of maximal $(s,t)$-geodesics and maximal $(t,s)$-geodesics which issue from $C$. Thus, the maximal alternating geodesics of $\cF_J(C)$ are contained in the residue $\cR_J(C)$, which is a Weyl polygon. If $\cW$ is a polygon, so that $J=S$, then we speak simply of the \emph{flower} $\cF(C)$ based at $C$.

A \emph{$J$-petal} is a $2$-element subset of a $J$-flower which contains a maximal $(s,t)$-geodesic $\rho(s,t)$, together with the unique maximal $(t,s)$-geodesic $\rho(t,s)$ such that $\rho(s,t)\sim \rho(t,s)$. A flower naturally induces a set of suites; for each petal $\big\{\rho(s,t), \rho(t,s)\big\}$, take the suite $\rho(s,t)\rho(t,s)^{-1}$.  

\begin{thm} \label{thm:flowerandpetals}
	Let $\cP$ be a Weyl polygon (rank $2$ Weyl graph), and let $C\in\cP$ be a chamber. Let $\cP_C$ denote the Weyl data whose underlying generalized chamber system is that of $\cP$, and whose defining suites are those induced by the flower $\cF(C)$ of $\cP$ at $C$. Then $\cP_C$ is isomorphic to $\cP$. 
\end{thm} 

\begin{proof}
	Let $\go:\cP_C\rightarrow \cP$ be the identity map, which is a morphism of Weyl data since the defining suites of $\cP_C$ are suites of $\cP$. Let $\theta$ be a suite of $\cP$. We claim that $\theta$ is also a suite of $\cP_C$, which proves that $\go:\cP_C\rightarrow \cP$ is an isomorphism by \autoref{prop:charisoofweyl}. 
	
	To see this, let $\gamma$ be a geodesic from $C$ to $\iota(\theta)=\tau(\theta)$. Using $1$-elementary homotopies and strict homotopies which take place in the suites induced by the flower $\cF(C)$ at $C$, we can obtain a geodesic $\gamma'$ which is homotopic to $\gamma \theta \gamma^{-1}$ in both $\cP$ and $\cP_C$. Since $\cP$ has property (W) and $\gamma \theta \gamma^{-1}$ is null-homotopic in $\cP$, any geodesic homotopic to $\gamma \theta \gamma^{-1}$ in $\cP$ must be trivial. Thus, $\gamma'$ is trivial. Therefore $\gamma \theta \gamma^{-1}$ is null-homotopic in $\cP_C$, and so $\theta$ is also null-homotopic in $\cP_C$. 
\end{proof}

Usually, the data of a Weyl polygon in the form of a rank $2$ generalized chamber system equipped with a flower will come from a group acting freely on the chambers (flags) of a generalized polygon. The petals can be determined by inspecting the action.  

\appendix
    
%%%%%%%%%%%%%%%%%%%%%%%%%%%%%%%%%%%%%%%%%%%%%%%%%%%%%%%%%%%%%%%%%%%%%%%%
%%%%%%%%%%%%%%%%%%%%%%%%%%%%%%%%%%%%%%%%%%%%%%%%%%%%%%%%%%%%%%%%%%%%%%%%                                 
\section{Covering Theory of Groupoids} \label{apen}
%%%%%%%%%%%%%%%%%%%%%%%%%%%%%%%%%%%%%%%%%%%%%%%%%%%%%%%%%%%%%%%%%%%%%%%%
%%%%%%%%%%%%%%%%%%%%%%%%%%%%%%%%%%%%%%%%%%%%%%%%%%%%%%%%%%%%%%%%%%%%%%%%

We give a self-contained exposition of covering theory of groupoids. This can be found in e.g. \cite{brown06topology}, however here we adopt a slightly different (basepoint-free) approach to the fundamental group.

%%%%%%%%%%%%%%%%%%%%%%%%%%%%%%%%%%%%%%%%%%%%%%%%%%%%%%%%%%%%%%%%%%%%%%%%
\subsection{Groupoids} \label{groupoids}
%%%%%%%%%%%%%%%%%%%%%%%%%%%%%%%%%%%%%%%%%%%%%%%%%%%%%%%%%%%%%%%%%%%%%%%%

A (small) \emph{groupoid} $\cG=(\cG_0,\cG_1)$ is a non-empty graph (as defined in \autoref{sec:defofgraph}) equipped with a partial function $\cG_1\times \cG_1\to \cG_1$ which assigns to each pair of edges $(g,h)$ such that $\tau(g)=\iota(h)$ their \emph{composition}, which is denoted either by $gh$, or by $g;h$ to avoid ambiguity. We also require the existence of identity edges at each vertex, called \emph{trivial} edges, and for each edge $g$ to have an inverse $g^{-1}$.

For vertices $x,y\in \cG$, let 
\[
\cG(x,y):=  \big \{    i\in \cG_1 :  \iota(i)=x \ \  \text{and}\ \   \tau(i)=y   \big\}
.\]  
The \emph{local group} $\cG_x:=\cG(x,x)$ at $x$ is the group with multiplication the restriction of the composition of $\cG$. Let $\cG(x,-)$ denote the set of edges $i\in \cG$ with $\iota(i)=x$, and let $\cG(-,x)$ denote the set of edges $i\in \cG$ with $\tau(i)=x$. A \emph{setoid} is a groupoid $\cG$ such that $\cG(x,y)$ contains at most one edge for all vertices $x,y\in \cG$. A groupoid $\cG$ is called \emph{connected} if $\cG(x,y)$ is non-empty for all vertices $x,y\in \cG$. A \emph{subgroup} of a groupoid $\cG$ is a subgroup of a local group of $\cG$. For $H\leq \cG_x$ a subgroup of $\cG$, let $H\backslash \cG$ denote the set of \emph{right cosets}, given by
\[
H\backslash \cG:=\big \{ Hg : g\in \cG(x,-)  \big \}
.\]
A \emph{homomorphism} of groupoids is a graph morphism which preserves composition. %In order to define a homomorphism of groupoids $\gf:\cG\to \cG'$, it suffices to define a function on the edges $\gf_1:\cG_1\to \cG'_1$ such that for all edges $g,h\in \cG$ such that $gh$ is defined, $\gf_1(g)\gf_1(h)$ is defined with, 
%\[\gf_1(gh)=\gf_1(g)\gf_1(h).\]
We call a groupoid homomorphism $\gf:\cG\to \cG'$ \emph{faithful} if the restriction of $\gf$ to each local group of $\cG$ is injective. We call a groupoid homomorphism $\gf$ an \emph{embedding}/\emph{surjective}/\emph{isomorphism} if $\gf_1$ is injective/surjective/bijective. 

For vertices $x,x',y,y'\in\cG$ such that $\cG(x,y)$ and $\cG(x',y')$ are non-empty, and for $g\in \cG(x,y)$ and $h\in \cG(x',y')$, we have the bijection
\[
\chi_{gh}:\cG(x,x')\to \cG(y,y')
,\qquad   
k\mapsto g^{-1}kh
.\] 
Let us denote $\chi_{gg}$ by $\chi_{g}$, which is an isomorphism of local groups $\chi_{g}:\cG_x \to \cG_y$.

%%%%%%%%%%%%%%%%%%%%%%%%%%%%%%%%%%%%%%%%%%%%%%%%%%%%%%%%%%%%%%%%%%%%%%%%
\subsection{Outer Homomorphisms} \label{fun group}
%%%%%%%%%%%%%%%%%%%%%%%%%%%%%%%%%%%%%%%%%%%%%%%%%%%%%%%%%%%%%%%%%%%%%%%%

The category of groupoids $\textsf{Grpd}$ (as a subcategory of the \hbox{$2$-category} of categories) is naturally a $2$-category. We now describe structure obtained from $\textsf{Grpd}$ by quotienting out the $2$-morphisms.

For $\gf:G\to H$ a homomorphism of groups and $h\in H$, let $^{h}\gf:G\to H$ denote the conjugate homomorphism $g\mapsto h\gf(g) h^{-1}$. Let an \emph{outer homomorphism} $\Phi:G\to H$ be a conjugacy class of homomorphisms 
\[
\Phi=[\gf]:=\{  ^{h}\gf: h\in H \}
.\] 
For outer homomorphisms $\Phi:G\to H$ and $\Phi':H\to K$, we put 
\[
\Phi' \circ \Phi=  [\varphi']\circ [\gf]:=  [\varphi'\circ \gf]
.\] 
Notice this is well-defined. An outer homomorphism $\Phi$ is called an \emph{outer embedding/outer isomorphism} if one (and therefore every) homomorphism in $\Phi$ is injective/bijective. We say two outer embeddings $\Phi:H\to G$ and $\Phi':K\to G$ are \emph{isomorphic} if there exists an outer isomorphism $\Psi:H\to K$ with $\Phi' \circ \Psi=\Phi$. %One can identify isomorphism classes of outer embeddings into a group $G$ with conjugacy classes of subgroups of $G$.

For $x,y\in \cG$ and $g, g'\in \cG(x,y)$ we have $\chi_{g'}=\ ^{g'^{-1}g } {\chi_{g}}$, and so $[\chi_{g'}]=  [\chi_{g}]$. Therefore, if $\cG(x,y)$ is non-empty, there is a canonical outer isomorphism
\[\Psi_{xy}:\cG_x \to \cG_{y}
,\qquad
\Psi_{xy}:=[\chi_{g}]
\] 
where $g\in \cG(x,y)$. We have
\[
\Psi_{yz} \circ \Psi_{xy}=\Psi_{xz}, \qquad \Psi_{xy}^{-1}=\Psi_{yx},\qquad  \Psi_{xx}=[\text{id}_{\cG_x}]
.\] 
Given a homomorphism of groupoids $\gf:\cG\to \cG'$, let $\gf |_{\cG_x}$ denote the restriction of $\gf$ to the local group $\cG_x$, and let $\Phi_x$ denote the outer homomorphism 
\[
\Phi_x:=[\gf |_{\cG_x}]:\cG_x\to \cG'_{\varphi(x)}
.\] 
It is straightforward to check that
\[
\Phi_y \circ \Psi_{xy}=\Psi_{\varphi(x)\varphi(y)} \circ \Phi_x
.\]
%The internal outer isomorphisms $\Psi_{xy}$ are compatible with the $\Phi_x$ in the following sense.
%\begin{prop} \label{prop:outerpres}
%	Let $\gf:\cG\to \cG'$ be a homomorphism of groupoids. Let $x,y\in \cG$ be vertices, and let $x'=\gf(x)$ and $y'=\gf(y)$. Then, 
%	\[\Phi_y \circ \Psi_{xy}=\Psi_{x'y'} \circ \Phi_x. \]
%\end{prop}
%
%\begin{proof}
%	Let $g\in \cG(x,y)$. Then,
%	\[\Phi_y \circ \Psi_{xy}= [\gf |_{\cG_y}\circ\ \chi_g]= [\chi_{\gf(g)} \circ\gf |_{\cG_x}]=\Psi_{x'y'} \circ \Phi_x. \qedhere
%	\]
%\end{proof}
Let the (basepoint-free) \emph{fundamental group} $\pi_1(\cG)$ of a connected groupoid $\cG$ be the abstract group which is isomorphic to the local groups of $\cG$, and which is equipped with an outer isomorphism $\Psi_x:\pi_1(\cG)\to \cG_x$ to each of the local groups of $\cG$ such that $\Psi_{xy} \circ \Psi_x=\Psi_{y}$ for all vertices $x,y\in \cG$. We associate to a homomorphism of connected groupoids $\gf:\cG \to \cG'$ the outer homomorphism $\pi_1(\gf):\pi_1(\cG)\to \pi_1(\cG')$, called the outer homomorphism \emph{induced} by $\gf$, which is given by
\[
\pi_1(\gf)=\Psi_{\varphi(x)}^{-1} \circ \Phi_{x} \circ \Psi_{x}
.\] 
It is straightforward to check that the definition of $\pi_1(\gf)$ does not depend on the choice of $x$. Given homomorphisms of connected groupoids $\gf:\cG\to \cG'$ and $\gf':\cG'\to \cG''$, we have
\[
\pi_1(\gf' \circ \gf)=\pi_1(\gf') \circ \pi_1(\gf)
.\] 

\begin{prop} \label{prop:whenhomoiso}
	A homomorphism of connected groupoids $\gf:\cG\to \cG'$ is an isomorphism if and only if $\pi_1(\gf)$ is an outer isomorphism and $\gf$ is bijective on vertices.    
\end{prop}

\begin{proof}
	Clearly, if $\gf$ is an isomorphism then $\pi_1(\gf)$ is an outer isomorphism and $\gf$ is bijective on vertices. Conversely, suppose that $\pi_1(\gf)$ is an outer isomorphism and $\gf$ is bijective on vertices. For surjectivity, let $g'\in \cG'(x',y')$ be an edge of $\cG'$, and let $x=\gf^{-1}(x')$ and $y=\gf^{-1}(y')$. Let $g\in \cG(x,y)$ and $h=\gf^{-1}(g'\gf(g)^{-1})$. Then $\gf(hg)=g'$. For injectivity, let $g,h\in \cG$ be edges such that $\gf(g)=\gf(h)$. We must have $\iota(g)=\iota(h)=x$ and $\tau(g)=\tau(h)=y$ because $\gf$ is injective on vertices. Then, $\gf(gh^{-1})=\gf(g)\gf(h)^{-1}= 1$. But $\pi_1(\gf)$ is an outer isomorphism, and so $gh^{-1}=1$. Therefore $g=h$. 
\end{proof}

\subsection{Coverings of Groupoids} \label{covgroup}
%%%%%%%%%%%%%%%%%%%%%%%%%%%%%%%%%%%%%%%%%%%%%%%%%%%%%%%%%%%%%%%%%%%%%%%%

%A reference for covering theory of groupoids, which uses basepoints, is \cite{brown06topology}. Our basepoint-free approach models coverings with outer embeddings of fundamental groups. 

A \emph{covering} of groupoids \index{groupoid!covering} is a surjective groupoid homomorphism 
\[
p:\wt{\cG}\to \cG
\] 
such that for all vertices $\tilde{x}\in \wt{\cG}$, the restriction of $p$ to $\wt{\cG}(\tilde{x},-)$ is a bijection into $\cG(p(\tilde{x}),-)$. Notice that if $\cG$ is connected, then surjectivity automatically follows. We say that a covering $p:\wt{\cG}\to \cG$ is \emph{connected} if $\wt{\cG}$ is connected, which implies that $\cG$ is also connected. We have the following equivalent definition in the case of connected groupoids.

\begin{prop} \label{prop:covchar}
	A covering of connected groupoids is equivalently a groupoid homomorphism $p:\wt{\cG}\to \cG$ such that there exists a vertex $\tilde{x}\in \wt{\cG}$ such that the restriction of $p$ to $\wt{\cG}(\tilde{x},-)$ is a bijection into $\cG(p(\tilde{x}),-)$.    
\end{prop}

\begin{proof}
	Let $\tilde{x},\tilde{y}\in \wt{\cG}_0$, and let $g\in \wt{\cG}(\tilde{x},\tilde{y})$. Let $\gf_g$ be the function,
	\[\gf_g:\cG(\tilde{x},-)\to \cG(\tilde{y},-),\ \  h\mapsto g^{-1}h\]
	\noindent and let $\gf_{p(g)^{-1}}$ be the function,
	\[\gf_{p(g)^{-1}}:\cG(p(\tilde{y}),-)\to \cG(p(\tilde{x}),-),\ \ h\mapsto p(g) h .\]
	\noindent Then,
	\[ p |_{ \wt{\cG}(\tilde{x},-)}=  \gf_{p(g)^{-1}} \circ p | _{\wt{\cG}(\tilde{y},-)} \circ\  \gf_g .\]
	\noindent But $\gf_g$ and $\gf_{p(g)^{-1}}$ are bijections since they have inverses $h\mapsto gh$ and $h\mapsto p(g)^{-1} h$ respectively. Thus, the restriction of $p$ to $\wt{\cG}(\tilde{x},-)$ is a bijection if and only if the restriction of $p$ to $\wt{\cG}(\tilde{y},-)$ is a bijection. The result follows.
\end{proof}

We have the following characterization of isomorphisms amongst coverings.

\begin{prop} \label{prop:injecvert}
	A covering $p:\wt{\cG}\to \cG$ is an isomorphism if and only if $p$ is injective on the vertices of $\wt{\cG}$.     
\end{prop}

\begin{proof}
	Clearly, if $p$ is an isomorphism, then $p$ is injective on vertices. Conversely, suppose that $p$ is injective on vertices, and let $g,g'\in \wt{\cG}$ be edges with $p(g)=p(g')$. Notice that we must have $\iota(g)=\iota(g')$ and $\tau(g)=\tau(g')$ by the fact that $p$ is injective on the vertices of $\wt{\cG}$. Let $x=\iota(g)$ and $y=\tau(g)$. Since $p$ is a covering, its restriction to $\wt{\cG}(x,y)$ is injective, and so we must have $g=g'$. Thus, $p$ is an embedding. Finally, $p$ is surjective by the fact it is a covering.
\end{proof}

Let $p:\wt{\cG}\to \cG$ and $p':\wt{\cG}'\to \cG$ be coverings of a groupoid $\cG$. A \emph{morphism} of coverings $\lambda:p \to p'$ is a groupoid homomorphism $\lambda:\wt{\cG}\to \wt{\cG}'$ such that $p=p' \circ \lambda$. We call two coverings \emph{isomorphic} if there exists a morphism between them which is a groupoid isomorphism. The \emph{composition} of morphisms of coverings is just their composition as groupoid homomorphisms. The following result shows that in particular, a morphism of coverings of connected groupoids is itself a covering.

\begin{prop} \label{prop:groupoidcover}
	Let $p:\wt{\cG}\to \cG$, $p':\wt{\cG}'\to \cG$, and $\lambda:\wt{\cG}\to \wt{\cG}'$ be homomorphisms of connected groupoids with $p=p'\circ \lambda$. If $p$ and $p'$ are coverings, then $\lambda$ is a covering, and if $p$ and $\lambda$ are coverings, then $p'$ is a covering. 
\end{prop}

\begin{proof}
	Pick a vertex $\tilde{x}\in \wt{\cG}$, and let $\tilde{x}'=\lambda(\tilde{x})$, and $x=p(\tilde{x})=p'(\tilde{x}')$. Then,
	\[  p|_{ \wt{\cG}(\tilde{x},-)}=  p'|_{ \wt{\cG'}(\tilde{x}',-)}\circ \  \lambda |_{ \wt{\cG}(\tilde{x},-)}.\] 
	Therefore if $p|_{ \wt{\cG}(\tilde{x},-)}$ and $p'|_{ \wt{\cG'}(\tilde{x}',-)}$ are bijections, then so is $\lambda |_{ \wt{\cG}(\tilde{x},-)}$, and if $p|_{ \wt{\cG}(\tilde{x},-)}$ and $\lambda |_{ \wt{\cG}(\tilde{x},-)}$ are bijections, then so is $p'|_{ \wt{\cG'}(\tilde{x}',-)}$. The result then follows by \autoref{prop:covchar}.
\end{proof}

We now describe the relationship between coverings of connected groupoids and their induced outer homomorphisms. 

\begin{prop}
	If $p:\wt{\cG}\to \cG$ is a covering of connected groupoids, then $\pi_1(p)$ is an outer embedding.
\end{prop}

\begin{proof}
	If $p$ is a covering, then for each vertex $\tilde{x}\in \wt{\cG}$, the restriction $p |_{ \wt{\cG}_{\tilde{x}}}$ is injective, i.e $p$ is faithful. It then follows from the definition that $\pi_1(p)$ is an outer embedding.  
\end{proof}

\begin{prop} \label{prop:conservative}
	A covering $p:\wt{\cG}\to \cG$ of connected groupoids is an isomorphism if and only if $\pi_1(p)$ is an outer isomorphism.
\end{prop}

\begin{proof}
	If $p$ is an isomorphism, then it follows from the definition that $\pi_1(p)$ is an outer isomorphism. Conversely, suppose that $\pi_1(p)$ is an outer isomorphism, and let $\tilde{x},\tilde{y}\in \wt{\cG}$ be vertices such that $p(\tilde{x})=p(\tilde{y})=x$. Then for each $g\in \wt{\cG}(\tilde{x},\tilde{y})$, we have $p(g)\in \cG_x$. But $p |_{ \wt{\cG}_{\tilde{x}}}$ is a bijection into $\cG_x$ because $\pi_1(p)$ is an outer isomorphism. Therefore, since $p$ is a covering, we must have $\tilde{x}=\tilde{y}$, and the result follows by \autoref{prop:injecvert} (or indeed \autoref{prop:whenhomoiso}). 
\end{proof}

In fact, one can recover a covering of connected groupoids (up to isomorphism) from the outer embedding which it induces. 

\begin{thm}[General Lifting] \label{prop:genlift}
	Let $\cG$, $\wt{\cG}$, and $\cG'$ be connected groupoids. Let $p:\wt{\cG}\to \cG$ be a covering, and let $\gf:\cG'\to \cG$ be a homomorphism. Let $\Phi:\pi_1(\cG')\to \pi_1(\wt{\cG})$ be an outer homomorphism with $\pi_1(p) \circ \Phi= \pi_1(\gf)$. Then there exists a homomorphism $\gf':\cG'\to \wt{\cG}$ such that $p\circ \gf'=\gf$ and $\pi_1(\gf')=\Phi$.
\end{thm}

\begin{proof}
	Pick a vertex $x'\in \cG'$ and a generating set $\cG'_{x'}\cup\{g_{y'} :y'\in \cG'_0\}$ based at $x'$. Recall that we can construct $\gf':\cG'\to \wt{\cG}$ by giving $\gf'|_{\cG'_{x'}}$ and the images of the $g_{y'}$.   
	
	Let $x=\gf(x')$, and let $\tilde{x}\in \wt{\cG}$ be any vertex such that $p(\tilde{x})=x$. Let $\gf_{x'}:\cG'_{x'}\to \wt{\cG}_{\tilde{x}}$ denote a homomorphism such that $[\gf_{x'}]=\Psi_{\tilde{x}}\circ \Phi\circ \Psi_{x'}^{-1}$. Then 
	\begin{align*}
	[p|_{\wt{\cG}_{\tilde{x}}}] \circ [\gf_{x'}]&=( \Psi_{x} \circ  \pi_1(p)  \circ \Psi_{\tilde{x}}^{-1} )\circ ( \Psi_{\tilde{x}}\circ \Phi\circ \Psi_{x'}^{-1})\\
	&=\Psi_{x} \circ \pi_1(p) \circ \Phi   \circ \Psi_{x'}^{-1}\\
	&=\Psi_{x} \circ \pi_1(\gf)  \circ \Psi_{x'}^{-1}\\
	&= [\gf|_{\cG'_{x'}}].  
	\end{align*} 
	\noindent So there exists $g\in \cG_x$ with 
	\begin{equation*} \label{eq} \tag{$\spadesuit$}
	\chi_g \circ p|_{\wt{\cG}_{\tilde{x}}} \circ\ \gf_{x'}=\gf|_{\cG'_{x'}}.
	\end{equation*}
	Let $\tilde{g}\in \wt{\cG}(\tilde{x},-)$ be the unique edge such that $p(\tilde{g})=g$.  Let $\tilde{y}=\tau(\tilde{g})$, and begin defining $\gf'$ by putting $\gf'(x')=\tilde{y}$ and $\gf'|_{\cG'_{x'}}= \chi_{\tilde{g}}\circ    \gf_{x'}$. Then
	\begin{align*}
	p|_{\wt{\cG}(\tilde{y})}    \circ\ \gf'|_{\cG'_{x'}}&=  p|_{\wt{\cG}_{\tilde{y}}}    \circ\    \chi_{\tilde{g}}\circ    \gf_{x'}  && \text{by the definition of}\ \gf'|_{\cG'_{x'}}             \\ 
	&= \chi_g \circ p|_{\wt{\cG}_{\tilde{x}}}\circ\   \gf_{x'}  && \text{since}\ p\ \text{is a homomorphism}  \\ 
	&=\gf|_{\cG'_{x'}}  && \text{by (\ref{eq})} 
	\end{align*}
	as required (since we want $p\ \circ\ \gf'=\gf$). We finish defining $\gf'$ by letting $\gf'(g_{y'})$ be the unique edge of $\wt{\cG}(\tilde{y},-)$ such that $p(\gf'(g_{y'}))=\gf(g_{y'})$. Then we have $p \circ \gf'=\gf$ since $p \circ \gf'$ agrees with $\gf$ on the generating set $\cG'_{x'}\cup\{g_{y'} :y'\in \cG'_0\}$. Finally, we have $\pi_1(\gf')=\Phi$ since
	\begin{align*}
	\Phi&=\Psi_{\tilde{x}}^{-1} \circ [\gf_{x'}]\circ \Psi_{x'}  &&      \text{by the definition of}\ \gf_{x'}                     \\
	&=\Psi_{\tilde{y}}^{-1} \circ \Psi_{\til{x} \til{y} } \circ [\gf_{x'}]\circ \Psi_{x'}                       \\
	&=\Psi_{\tilde{y}}^{-1} \circ [\chi_{\tilde{g}}]\circ    [\gf_{x'}]\circ \Psi_{x'}\\
	&=\Psi_{\tilde{y}}^{-1} \circ [\gf'|_{\cG'_{x'}}]\circ \Psi_{x'}  &&      \text{by the definition of}\ \gf'|_{\cG'_{x'}}                          \\
	&=\pi_1(\gf') &&  \text{by the definition of}\ \pi_1(\gf').                                \qedhere
	\end{align*}
\end{proof}

\begin{cor}  \label{cor:genlift}
	Let $\cG$, $\wt{\cG}$ and $\wt{\cG}'$ be connected groupoids. Let $p:\wt{\cG}\to \cG$ and $p':\wt{\cG}'\to \cG$ be coverings. Let $\Phi:\pi_1(\wt{\cG})\to \pi_1(\wt{\cG}')$ be an outer embedding with $\pi_1(p') \circ \Phi= \pi_1(p)$. Then there exists a morphism of coverings $\lambda:p\to p'$ with $\pi_1(\lambda)=\Phi$. Moreover, if $\Phi$ is an outer isomorphism, then any such $\lambda$ is an isomorphism.
\end{cor}

\begin{proof}
	A homomorphism $\lambda:\wt{\cG}\to \wt{\cG}'$ such that $p' \circ  \lambda=p$ and $\pi_1(\lambda)=\Phi$ exists by \autoref{prop:genlift}. Then $\lambda$ is a covering of groupoids by \autoref{prop:groupoidcover}, and so $\lambda$ is an isomorphism if $\Phi$ is an outer isomorphism by \autoref{prop:conservative}.
\end{proof}

\begin{cor}
	Let $\cG$, $\wt{\cG}$, and $\wt{\cG}'$ be connected groupoids. Let $p:\wt{\cG}\to \cG$ and $p':\wt{\cG}'\to \cG$ be coverings. If $\pi_1(p)$ and $\pi_1(p')$ are isomorphic outer embeddings, then $p$ and $p'$ are isomorphic coverings. 
\end{cor}

\begin{proof}
	By hypothesis, there exists an outer isomorphism $\Psi:\pi_1(p)\to \pi_1(p')$ with $\pi_1(p') \circ \Psi=\pi_1(p)$. There exists an isomorphism of coverings $\lambda:p\to p'$ with $\pi_1(\lambda)=\Psi$ by \autoref{cor:genlift}. 
\end{proof}

This shows that the (isomorphism classes of) connected coverings of a connected groupoid $\cG$ naturally inject into the conjugacy classes of subgroups of $\pi_1(\cG)$. We now describe a construction of connected coverings of $\cG$, and then show that this constructs coverings for each conjugacy class of subgroups of $\pi_1(\cG)$.

Let $H\leq \cG_x$ be a subgroup of a connected groupoid $\cG$. For each coset $Hg\subseteq \cG$, pick a representative $g^{*}\in Hg$. We make the convention that $h^{*}=1_x$ for $h\in H$. We construct a connected groupoid, denoted $\wt{\cG}^H$, by letting the vertices of $\wt{\cG}^H$ be the set of cosets $H\backslash\cG$, and letting the edges of $\wt{\cG}^H$ be the set
\[\wt{\cG}^H_1=\big\{ (h,Hg,Hg'): h\in H;\ g,g'\in \cG(x,-)\big \}.\]
\noindent For the extremities, put 
\[\iota(h,Hg,Hg')=Hg \qquad \text{and} \qquad  \tau(h,Hg,Hg')=Hg'\] 
\noindent and for the composition, put
\[(h,Hg,Hg')(h',Hg',Hg'')=(hh',Hg,Hg'').\] 
It is easy to check that this defines a groupoid $\wt{\cG}^H$. Let $p^H:\wt{\cG}^H \to \cG$ be the homomorphism whose map on edges is
\[(h,Hg,Hg')\mapsto (g^{*})^{-1}hg'^{*}.\] 
This implies that for vertices we have $p^H(Hg)=\tau(g)$, and in particular $p^H(H)=x$. Checking that $p^H$ is a homomorphism, we have
\begin{align*}
p^H (hh',Hg,Hg'' )&=  (g^{*})^{-1}hh'g''^{*}\\
&=(g^{*})^{-1}hg'^{*}(g'^{*})^{-1}h'g''^{*}\\
&=p^H(h,Hg,Hg')p^H(h',Hg',Hg'').
\end{align*}

\begin{prop}
	Let $H\leq \cG_x$ be a subgroup of a connected groupoid $\cG$. Then $p^H:\wt{\cG}^H \to \cG$ is a covering of groupoids.
\end{prop}

\begin{proof}
	Consider the restriction of $p^H$ to the edges $(h,H,Hg)$ which issue from the vertex $H\in \wt{\cG}^H$. For injectivity, if $p^H(h,H,Hg)=p^H(h',H,Hg')$, then $hg^{*}=h'g'^{*}$, and so $h=h'$ and $g=g'$. Thus, $(h,H,Hg)=(h',H,Hg')$. For surjectivity, let $g\in \cG(x,-)$, and let $h\in H$ such that $g=hg^{*}$. Then we have $p^{H}(h,H,Hg)=g$. The fact that $p^H$ is a covering then follows by \autoref{prop:covchar}. 
\end{proof}

We call $p^{H}:\wt{\cG}^H \to \cG$ \emph{the covering based at $H$}. Notice that the local group $\wt{\cG}^H_H$ of $\wt{\cG}^H$ at $H$ is naturally isomorphic to $H$.  

\begin{prop}
	Let $\cG$ be a connected groupoid, and let $\Phi:K\to \pi_1(\cG)$ be an outer embedding. Then there exists a covering $p:\wt{\cG}\to \cG$ such that $\pi_1(p)$ is isomorphic to $\Phi$.
\end{prop}

\begin{proof}
	Pick a subgroup $H\leq \cG_x$ of $\cG$ such that $\Psi_x^{-1}|_ H$ is isomorphic to $\Phi$. Put $p=p^{H}:\wt{\cG}^H \to \cG$. Let $\gf_H:\wt{\cG}^H_{H} \to  H\leq \cG_x$ be the identity map and put $\Phi_H=[p|_{\wt{\cG}^H_H}]$. Notice that $\gf_H$ is just the embedding $p|_{\wt{\cG}^H_H}$ restricted to its image. Pick representatives $\psi_x^{-1}\in \Psi_x^{-1}$ and $\psi_H\in\Psi_H$. Then $[\gf_H] \circ \Psi_H$ is an outer isomorphism, and
	\begin{align*}
	\pi_1(p) =\ &             \Psi_x^{-1} \circ \Phi_H \circ \Psi_H      \\
	=\ &	[   \psi_x^{-1} \circ\  p|_{\wt{\cG}^H_H} \circ\ \psi_H  ]			\\
	=\ &  [   \psi_x^{-1}|_ H\circ\  \gf_H \circ \psi_H  ]	\\
	=\ &  [   \psi_x^{-1}|_ H]\circ\  [\gf_H] \circ[\psi_H  ]	\\
	=\ & \Psi_x^{-1}|_ H\circ\  [\gf_H] \circ \Psi_H     . 		  
	\end{align*}
	Therefore $\pi_1(p)$ and $\Psi_x^{-1}|_ H$ are isomorphic via $[\gf_H] \circ \Psi_H$. Then, since $\Psi_x^{-1}|_ H$ is isomorphic to $\Phi$, we have that $\pi_1(p)$ is also isomorphic to $\Phi$.
	%Recall that $\pi_1(p)=\Psi_{x}^{-1}\circ[p|_{\wt{\cG}^H_{H}}] \circ \Psi_{H}$.
\end{proof}

This demonstrates that the connected coverings of a connected groupoid $\cG$ are naturally in bijection with outer embeddings in $\pi_1(\cG)$ (up to isomorphism), which in turn are naturally in bijection with the conjugacy classes of subgroups of $\pi_1(\cG)$. 

Given a covering $p:\wt{\cG}\to \cG$ of connected groupoids, $p$ is called a \emph{universal cover} if for any covering $p':\wt{\cG}' \to \cG$ such that $\wt{\cG}'$ is connected, there exists a covering morphism $\lambda:p\to p'$. Given the 1-1 correspondence between coverings and conjugacy classes of subgroups, we see that a covering $p:\wt{\cG}\to \cG$ is universal if and only if $\wt{\cG}$ is a connected setoid. Thus, each connected groupoid $\cG$ has a unique universal cover (up to isomorphism).  

Groups act by automorphisms on groupoids. We say a group $G$ acts \emph{freely} on a groupoid $\cG$ if the action of $G$ on $\cG_1$ is free. The \emph{deck transformation group} $\Aut(p)$ of a covering $p:\wt{\cG}\to \cG$ is the group of automorphisms of $\wt{\cG}$ which commute with $p$. Since $\Aut(p)\leq \Aut(\wt{\cG})$, a covering $p:\wt{\cG}\to \cG$ determines a faithful action of $\Aut(p)$ on the left of $\wt{\cG}$. 

\begin{prop} \label{prop:freeaction}
	Let $\cG$ and $\wt{\cG}$ be connected groupoids, and let $p:\wt{\cG}\to \cG$ be a covering. Then $\Aut(p)$ acts freely on $\wt{\cG}$.   
\end{prop}

\begin{proof}
	We show that $\Aut(p)$ acts freely on vertices. Let $\gamma\in \Aut(p)$, and suppose there exists $x\in \wt{\cG}_0$ with $\gamma \cdot x=x$. Since $p$ is injective on $\wt{\cG}(x,-)$, we have $\gamma \cdot g=g$ for all $g\in \wt{\cG}(x,-)$. But $\wt{\cG}(x,-)$ generates $\wt{\cG}$, and so $\gamma=1$. 
\end{proof}

A covering $p:\wt{\cG}\to \cG$ of connected groupoids is called \emph{regular} if its associated conjugacy class of subgroups is a single normal subgroup. If $p$ is regular we identify $\pi_1(p)$ with the single embedding it contains, and we identify $\pi_1(\wt{\cG})$ with its $\pi_1(p)$-image in $\pi_1(\cG)$. 

\begin{prop} \label{prop:fib}
	Let $\cG$ and $\wt{\cG}$ be connected groupoids, and let $p:\wt{\cG}\to \cG$ be a regular covering. Then the action of $\Aut(p)$ restricted to the $p$-preimage of a vertex or an edge is regular.
\end{prop}

\begin{proof}
	We know that these actions are free by \autoref{prop:freeaction}. First, we show that the action is transitive in the case of a vertex. Let $\tilde{x},\tilde{y}\in \wt{\cG}$ be vertices with $p(\tilde{x})=p(\tilde{y})$. We construct a deck transformation $\gamma\in \Aut(p)$, with $\gamma\cdot \tilde{x}=\tilde{y}$, by defining $\gamma$ on a generating set $\wt{\cG}_{\tilde{x}}\cup \{ g_y : y\in \wt{\cG}_0  \}$. Let $\gamma|_{\wt{\cG}_{\tilde{x}}}:\wt{\cG}_{\tilde{x}} \to \wt{\cG}_{\tilde{y}}$ be defined by $\gamma\cdot g= g_{\tilde{y}}^{-1}gg_{\tilde{y}}$, and let $\gamma\cdot g_y$ be the unique edge of $\wt{\cG}(\tilde{y},-)$ such that $p(\gamma\cdot g_y)=p(g_y)$. Notice that $\gamma$ is a covering morphism $\gamma:p\to p$ because, for $g\in \wt{\cG}_{\tilde{x}}$, we have $p(\gamma\cdot g)=p( g_{\tilde{y}}^{-1}gg_{\tilde{y}})=p(g)$ since $p$ is regular. Then $\pi_1(\gamma)$ is an outer isomorphism, and so $\gamma$ is an automorphism by \autoref{cor:genlift}.
	
	In the case of an edge, let $g,g'\in \wt{\cG}$ be edges with $p(g)=p(g')$. Then we've just shown that there exists $\gamma\in \Aut(p)$ such that $\gamma \cdot \iota(g)=\iota(g')$. Thus $\gamma\cdot g= g'$, since $p$ is a covering.
\end{proof}

\begin{thm} \label{prop:natouteriso}
	Let $\cG$ and $\wt{\cG}$ be connected groupoids, and let $p:\wt{\cG}\to \cG$ be a regular covering. Then there exists a natural outer isomorphism, 
	\[\Psi: \pi_1(\cG)/\pi_1(\wt{\cG}) \to \Aut(p).\]  
\end{thm}

\begin{proof}
	Pick vertices $x\in \cG$ and $\tilde{x}\in \wt{\cG}$ such that $p(\tilde{x})=x$. For $g\in \cG(x,-)$, let $\tilde{g}\in \cG(\tilde{x},-)$ be the unique edge such that $p(\tilde{g})=g$. Let $\varphi:\cG_x\to \Aut(p)$ be the surjective homomorphism such that $\varphi(g)\cdot \tilde{x}=\tau(\tilde{g})$. This is well-defined by \autoref{prop:fib}. To see that $\varphi$ is a homomorphism, let $g,h\in \cG_x$, and put $k=gh$. Then $\tau(\wt{k})=\varphi(g)\cdot \tau(\tilde{h})$ since $\varphi(g)\cdot \tilde{h}$ must be in the $p$-preimage of $h$. Thus,
	\[    \varphi(gh)\cdot \tilde{x}=\varphi(k)\cdot \tilde{x}=\tau(\wt{k})=\varphi(g)\cdot \tau(\tilde{h})=  \varphi(g)\cdot \varphi(h)\cdot \tilde{x}  .    \] 
	To see that $\gf$ is surjective, let $\gamma\in \Aut(p)$ and pick $g_\gamma\in \wt{\cG}(\tilde{x},\gamma\cdot \tilde{x})$. Then $p(g_\gamma)\in \cG_x$, and $\gf(p(g_\gamma))=\gamma$. 
	
	Let $\Phi:\pi_1(\cG)\to \Aut(p)$ be the outer homomorphism $\Phi=[\gf]\circ \Psi_x$. We now show that $\Phi$ does not depend on the choice of $x$ and $\tilde{x}$. Suppose that we make a different choice of vertices $y\in \cG$ and $\tilde{y}\in \wt{\cG}$ such that $p(\tilde{y})=y$. Let $\varphi':\cG_{y}\to \Aut(p)$ be the new homomorphism. Pick $\tilde{g}'\in \wt{\cG}(\tilde{y},\tilde{x})$ and let $g'=p(\tilde{g}')$. Let $\chi_{g'}:\cG_{y}\to \cG_{x}$ be the usual isomorphism $g\mapsto g'^{-1}gg'$. Then $\varphi'=\varphi  \circ \chi_{g'}$, and so, 
	\[[\gf']\circ \Psi_y=[\varphi  \circ \chi_{g'}]\circ \Psi_y=  [\varphi]  \circ \Psi_{yx}\circ \Psi_y =[\gf]\circ \Psi_x.\]  
	For $g\in \cG_x$, we have $\gf(g)=1$ if and only if $\tilde{g}$ is a loop. Therefore the kernel of $\varphi$ is $p(\cG_{\tilde{x}})\leq \cG_x$, and so the kernel of each group homomorphism in $\Phi$ is $\pi_1(\wt{\cG})$. Let $\Psi$ be the set of isomorphisms obtained by factoring out the kernels of the homomorphisms in $\Phi$. Then $\Psi$ is an outer isomorphism $\Psi:\pi_1(\cG)/\pi_1(\wt{\cG}) \to \Aut(p)$.      
\end{proof}

We now show that if a group $G$ acts freely on a connected groupoid $\cG$, then there exists a groupoid $\cG'$ and a regular covering $\cG \to \cG'$ of which $G$ is naturally the automorphism group. We associate to the free action of a group $G$ on a groupoid $\cG$ the \emph{quotient groupoid} $G\backslash \cG$, which is the groupoid defined as follows; the set of vertices of $G\backslash \cG$ is the set of orbits of vertices $G\backslash \cG_0=\big\{  [x]     : x\in \cG_0  \big   \}$, the set of edges of $G\backslash \cG$ is the set of orbits of edges $G\backslash \cG_1=\big\{  [g]     : g\in \cG_1  \big   \}$, and for the extremities of edges we have $\iota([g])= [\iota(g)]$ and $\tau([g])= [\tau(g)]$. We also put $1_{[x]}= [1_x]$, $[g]^{-1}= [g^{-1}]$, and the composition $[g][g']$ is defined if there exists an edge $g''\in [g']$ such that $gg''$ is defined, in which case we put $[g][g']=     [gg'']$. Checking that this is well-defined groupoid is a tedious exercise. The \emph{quotient map} $\pi:\cG\to G\backslash \cG$ is the homomorphism such that $x\mapsto [x]$ for $x\in \cG_0$, and $g\mapsto [g]$ for $g\in \cG_1$.

\begin{thm}  \label{prop:quot}
	Let $G$ be a group which acts freely on a groupoid $\cG$. Then $\pi:\cG\to G\backslash \cG$ is a covering of groupoids. Moreover, if $\cG$ is connected, then $G$ is naturally isomorphic to $\Aut(\pi)$.       
\end{thm} 

\begin{proof}
Notice that $\pi:\cG\to G\backslash \cG$ is clearly a homomorphism. To see that $\pi$ is a covering, let $x\in \cG$ be a vertex, and let $[g]$ be an edge which issues from $[x]$. Let $\gamma\in G$ be the element such that $\gamma\cdot \iota(g)=x$. Then $\gamma\cdot g$ is an edge which issues from $x$ with $\pi(\gamma\cdot g)=[g]$. Suppose that $g'$ is an edge which also issues from $x$ with $\pi(g')=[g]$. Then there exists $\gamma'\in G$ with $\gamma'\cdot g= g'$. Then $\gamma \gamma'^{-1} \cdot x=x$, and so $\gamma=\gamma'$ since $G$ acts freely on $\cG$. Therefore $g'=\gamma\cdot g$. Finally, $\pi$ is clearly surjective on vertices. This proves that $\pi$ is a covering.
	
	We have a natural embedding $\gf:G\hookrightarrow \Aut(\pi)$. To see that $\gf$ is surjective in the case where $\cG$ is connected, let $a\in \Aut(\pi)$, and for any edge $g\in \cG$, let $\gamma\in G$ such that $\gamma \cdot g= a\cdot g$. Then $\gf(g)=\gamma$ by \autoref{prop:freeaction}. 
\end{proof}

\begin{prop}  
	Let $\cG$ and $\wt{\cG}$ be connected groupoids, let $p:\wt{\cG}\to \cG$ be a regular covering, and let $\pi:\wt{\cG}\rightarrow \Aut(p)\backslash\wt{\cG}$ be the quotient map associated to the action of $\Aut(p)$. Then there exists a unique isomorphism $\psi:\Aut(p)\backslash\wt{\cG}\rightarrow  \cG$ such that $p=\psi\circ\pi$. 
\end{prop}

\begin{proof}
	Since we want $p=\psi\circ\pi$, we have no choice but to let $\psi:\Aut(p)\backslash\wt{\cG}\rightarrow  \cG$ be the homomorphism whose map on edges is, \[[g]\mapsto p(g), \ \ \ \ \text{for } g\in \wt{\cG}          .             \] 
	This is well-defined since for $\gamma\in \Aut(p)$, we have $p(\gamma \cdot g)=p(g)$. Checking that $\psi$ is a homomorphism, we have, 
	\[[g][g']=[g g']\mapsto p(gg')=p(g)p(g').\] 
	The restriction of $\psi$ to edges of $\Aut(p)\backslash\wt{\cG}$ is a bijection because it has the inverse $g\mapsto \pi( p^{-1}(g))$. This inverse is well-defined by \autoref{prop:fib}. Thus, $\psi$ is an isomorphism. 
\end{proof} 

\bibliographystyle{alpha}
\bibliography{sample}

%\nocite{*}
%\printindex

\end{document}